\documentclass{article} 
\usepackage{amsmath,amssymb,amsthm} 
\usepackage[OT2,T1]{fontenc}
\usepackage[utf8]{inputenc}
\usepackage{enumitem}
\usepackage{graphicx}
\usepackage{geometry}
\usepackage{hyperref}
\usepackage{tikz}
\usetikzlibrary{cd}
\usepackage{tikz-cd}
\usepackage{xcolor}
\usepackage{ stmaryrd }
\usepackage{mathrsfs}
\usepackage{titling}
\usepackage{fullpage}
\usepackage{mathtools}
\usepackage[english]{babel}
\usepackage{setspace}
\usepackage[normalem]{ulem}
\usepackage{mathdots}
\usepackage{microtype}

\usepackage{arydshln}

\newtheorem{proposition}{Proposition}[section]
\newtheorem{definition}[proposition]{Definition}

\newtheorem{theorem}[proposition]{Theorem}
\newtheorem{lemma}[proposition]{Lemma}
\newtheorem{corollary}[proposition]{Corollary}
\newtheorem{example}[proposition]{Example}

\numberwithin{equation}{section}

\newtheorem{thm}{Theorem}[section]

\newtheorem{thmtheta}{Theorem}[section]

\newtheorem*{examplenonumber}{Example}
\newtheorem{question}{Question}

\newtheorem{remark}[proposition]{Remark}

\newcommand{\G}{\mathbb{G}}
\DeclareMathOperator{\GL}{GL}

\DeclareMathOperator{\Sp}{Sp}

\DeclareMathOperator{\Id}{Id}

\DeclareMathOperator{\Hom}{Hom}
\DeclareMathOperator{\Aut}{Aut}

\DeclareMathOperator{\Gal}{Gal}

\DeclareMathOperator{\Sym}{Sym}

\DeclareMathOperator{\image}{image}
\DeclareMathOperator{\coker}{coker}
\DeclareMathOperator{\Stab}{Stab}
\DeclareMathOperator{\Isom}{Isom}
\DeclareMathOperator{\Mat}{Mat}

\DeclareMathOperator{\Spec}{Spec}

\DeclareMathOperator{\Pic}{Pic}
\DeclareMathOperator{\NS}{NS}

\newcommand{\sh}[1]{\mathscr{#1}}

\newcommand{\A}{\mathbb{A}}
\renewcommand{\P}{\mathbb{P}}

\renewcommand{\O}{\mathcal{O}}
\newcommand{\GIT}{\mathbin{/\mkern-6mu/}}
\newcommand{\HH}{\mathrm{H}}
\DeclareMathOperator{\CH}{CH}

\newcommand{\Q}{\mathbb{Q}}

\newcommand{\Z}{\mathbb{Z}}
\newcommand{\F}{\mathbb{F}}

\DeclareSymbolFont{cyrletters}{OT2}{wncyr}{m}{n}
\DeclareMathSymbol{\Sha}{\mathalpha}{cyrletters}{"58}

\DeclareUnicodeCharacter{00A0}{ }
\makeatletter
\newcommand{\extp}{\@ifnextchar^\@extp{\@extp^{\,}}}
\def\@extp^#1{\mathop{\bigwedge\nolimits^{\!#1}}}
\makeatother

\setlist[itemize]{topsep=0pt}
\setlist[enumerate]{topsep=0pt}

\DeclareMathOperator{\sep}{sep}
\newcommand{\PicS}{\mathbf{Pic}}
\newcommand{\NSS}{\mathbf{NS}}
\DeclareMathOperator{\ThetaCat}{ThetaGrp}
\DeclareMathOperator{\SymThetaCat}{SymThetaGrp}
\newcommand{\AutS}{\mathbf{Aut}}
\newcommand{\HomS}{\mathbf{Hom}}
\newcommand{\IsomS}{\mathbf{Isom}}
\DeclareMathOperator{\sym}{sym}
\DeclareMathOperator{\PolTor}{PolTor}
\newcommand{\pargroup}[2]{\mathrm{Sp}_{2 #1}^{#2}}
\newcommand{\SpS}{\mathbf{Sp}}

\usepackage{microtype}

\title{Polarizations, torsors and theta groups}

\author{Jef Laga}

\begin{document}

\maketitle

\begin{abstract}
    Let $\lambda\colon A\rightarrow A^{\vee}$ be a polarization on an abelian variety over a field $k$.
    If $k$ is not algebraically closed, there might not exist an ample line bundle on $A$ defined over $k$ that represents $\lambda$.
    To remedy this, Poonen and Stoll have asked the following question: does there exist a line bundle on an $A$-torsor that represents $\lambda$?

    We give a criterion for the existence of such a torsor and line bundle which only depends on the kernel of $\lambda$. Using this criterion, we show that the answer to the question is yes when the polarization has odd or small even degree. On the other hand, we show that for every $g\geq 7$, there exists a polarized $g$-dimensional abelian variety for which the answer to the question is no.
\end{abstract}


\section{Introduction}

\subsection{Context}\label{subsection: intro context}

Let $A$ be an abelian variety over a field $k$.
If $L$ is a line bundle on $A$, define the homomorphism 
\begin{align*}
\phi_L\colon A\rightarrow A^{\vee}, \;
a\mapsto t_a^{*}L \otimes L^{-1},
\end{align*}
where $t_a\colon A\rightarrow A$ denotes translation by $a$ and $A^{\vee} = \PicS_A^0$ denotes the dual abelian variety.
By definition, a polarization on $A$ is a homomorphism $\lambda\colon A\rightarrow A^{\vee}$ such that, over an algebraic closure $\bar{k}$ of $k$, $\lambda_{\bar{k}} = \phi_L$ for some ample line bundle $L$ on $A_{\bar{k}}$.
It is natural to ask: is this base change to $\bar{k}$ necessary? In other words, can we choose $L$ to be defined over $k$?
The following example shows that this is not always possible:
\begin{examplenonumber}
{
Let $C/k$ be a (smooth, projective, geometrically connected) curve of genus $g$ with Jacobian variety $J = \PicS_C^0$, a $g$-dimensional abelian variety. 
The locus of effective line bundles defines a divisor $\Theta\subset \PicS_{C}^{g-1}$, the theta divisor.
If $x\in \PicS_C^{g-1}(k)$, then the translate of $\O_{\PicS_{C}^{g-1}}(\Theta)$ by $x$ defines a line bundle $L_x$ on $J_{\bar{k}}$.
The morphism $\phi_{L_x}\colon J_{\bar{k}}\rightarrow J_{\bar{k}}^{\vee}$ is independent of the choice of $x$, hence descends to a morphism $\lambda\colon J\rightarrow J^{\vee}$; this is the canonical principal polarization on a Jacobian.
The association $x\mapsto L_x$ induces a bijection between $\PicS_C^{g-1}(k)$ and the set $\{[L]\in \Pic(J)\colon \phi_L = \lambda\}$.
In particular, the latter set is empty if $\PicS^{g-1}_C(k) = \varnothing$.
If $g=2$, there are many curves $C$ over $\Q$ with $\PicS_{C}^{g-1}(\Q) = \varnothing$; see \cite[Theorem 23]{PoonenStoll-CasselsTate}.
}
\end{examplenonumber}
In this example, the polarization on $J$ is still constructed by a line bundle that is defined over $k$, namely the one induced by the theta divisor on the $J$-torsor $\PicS_C^{g-1}$.
In general, given an $A$-torsor $X$ and a line bundle $L$ on $X$, the definition of $\phi_L\colon A\rightarrow A^{\vee}$ still makes sense since the translation-by-$a$ map $t_a\colon X\rightarrow X$ is well defined and since there is a canonical identification $\PicS^0_X = \PicS^0_A$, see Section \ref{subsec: polarizations on abelian varieties} for details. 
The following basic question is therefore very natural, and has been explicitly raised by Poonen and Stoll \cite[\S4, p.\,1120]{PoonenStoll-CasselsTate} around 25 years ago.

\begin{question}\label{question: main question poonen-stoll}
    If $\lambda\colon A\rightarrow A^{\vee}$ is a polarization on an abelian variety, does there always exist an $A$-torsor $X$ and a line bundle $L$ on $X$ such that $\lambda = \phi_L$?
\end{question}

It is known that the answer to Question \ref{question: main question poonen-stoll} is yes in all of the following cases:
\begin{itemize}
\item $(A, \lambda)$ is principally polarized \cite[Corollary 3.0.7]{alexeev-completemodulisemiabelian}. 
\item $k$ is a local field; then there even exists a line bundle $L$ on $A$ with $\lambda = \phi_L$ by \cite[\S4, Lemma 1]{PoonenStoll-CasselsTate}.
\item $\lambda = 2\mu$ for some polarization $\mu$; in that case the pullback $L = (1, \mu)^*\mathcal{P}$ of the Poincar\'e bundle $\mathcal{P}$ on $A\times A^{\vee}$ satisfies $\phi_L = \lambda$ by \cite[Proposition 6.10]{MumfordFogartyKirwan-GIT}.
\item At least one of the conditions of \cite[Proposition 3.12]{PoonenRains-thetacharacteristics} is satisfied (for example, $k$ is finite); in that case there even exists a \emph{symmetric} $L$ on $A$ representing $\lambda$.
\end{itemize}

This question also relates to the works of Alexeev and Olsson on compactifying moduli spaces of polarized abelian varieties \cite{alexeev-completemodulisemiabelian, olsson-compactifymoduli}: in these works the moduli space of pairs $(A, \lambda)$ is replaced by a moduli space of triples $(A, X, L)$, and it is natural to wonder whether the association $(A, X, L)\mapsto (A, \phi_L)$ admits an inverse in general.

The definition of $\phi_L$ even makes sense if we only assume the isomorphism class of $L$ to be Galois-invariant, i.e., if $L$ defines an element of $\PicS_X(k) = \Pic(X_{k^{\sep}})^{\Gal(k^{\sep}|k)}$, where $k^{\sep}$ is a separable closure of $k$.
We can therefore ask the (a priori) weaker variant of Question \ref{question: main question poonen-stoll}:
\begin{question}\label{question: main question weaker variant}
     If $\lambda\colon A\rightarrow A^{\vee}$ is a polarization on an abelian variety over $k$, does there always exist an $A$-torsor $X$ and an element $[L]\in \Pic(X_{k^{\sep}})^{\Gal(k^{\sep}|k)}$ such that $\lambda = \phi_L$?
\end{question}

The purpose of this paper is to show that even though Questions \ref{question: main question poonen-stoll} and \ref{question: main question weaker variant} often have a positive answer, they have a negative answer in general.

\subsection{Results}

\begin{thm}\label{theorem: intro odd degree case}
    Let $\lambda\colon A\rightarrow A^{\vee}$ be a polarization whose degree is odd and invertible in $k$.
    Then there exists an $A$-torsor $X$ and a line bundle $L$ on $X$ with $\phi_L = \lambda$.
\end{thm}

To state a more general theorem, recall that if $\lambda\colon A\rightarrow A^{\vee}$ is a polarization whose degree is invertible in $k$, there exists a unique sequence of positive integers $d_1, \dots, d_g$, each one dividing the next, with $g=\dim(A)$ and such that the kernel $A[\lambda]$ of $\lambda$ satisfies
\[
A[\lambda](\bar{k}) \simeq (\Z/d_1\Z)^2 \times \cdots \times (\Z/d_g\Z)^2.
\]
The tuple $D = (d_1, \dots, d_g)$ is called the type of $(A, \lambda)$.
For each $i$, let $2^{n_i}$ be the largest power of $2$ dividing $d_i$ and write $D_2 = (2^{n_1}, \dots, 2^{n_g})$. 

\begin{thm}\label{theorem: intro positive result}
    Let $\lambda\colon A\rightarrow A^{\vee}$ be a polarization whose degree is invertible in $k$. 
    Suppose that $D_2 = (1, \dots, 1)$, $(1, \dots, 1,2)$, $(1,\dots, 1,2,2)$ or $(1, \dots, 1,2,2,2)$. 
    In the last case, additionally suppose that $k$ has characteristic zero.
    Then there exists an $A$-torsor $X$ and a line bundle $L$ on $X$ with $\phi_L = \lambda$.
\end{thm}

Using Theorem \ref{theorem: intro positive result} and previous results, we show that the answer to Question \ref{question: main question poonen-stoll} is always yes when $\dim A\leq 2$, see Proposition \ref{proposition: answer to question yes when dim A is 2}.
However, this pattern does not persist: the next theorem shows that there exist polarized abelian varieties of every dimension at least $7$ for which the answer to Questions \ref{question: main question poonen-stoll} and \ref{question: main question weaker variant} is no.

\begin{thm}\label{theorem: intro negative result}
    Consider a sequence of the form
    \[D = (\overbrace{1, \dots, 1}^{\geq 3\text{ times}}, \overbrace{2,\dots, 2}^{\geq 4\text{ times}}).\]
    Then there exists a field $k$ of characteristic zero and a polarized abelian variety $(A, \lambda)$ of type $D$ over $k$ with the property that there does not exist an $A$-torsor $X$ and an element $[L]\in \PicS_X(k)$ such that $\lambda = \phi_L$.
    In other words, the answer to Question \ref{question: main question weaker variant} (and hence to Question \ref{question: main question poonen-stoll}) is no for $(A, \lambda)$.
\end{thm}

The counterexample of Theorem \ref{theorem: intro negative result} is generic, in the following sense.
For a type $D = (d_1, \dots, d_g)$, let $\mathcal{A}_{g,D}\rightarrow \Spec(\mathbb{C})$ be the moduli stack of $g$-dimensional abelian varieties with a polarization of type $D$.
If $p\geq 3$ is a prime not dividing $d_g$, let $\mathcal{A}_{g,D}[p]\rightarrow \mathcal{A}_{g,D}$ be the cover given by adding full level-$p$ structure, a smooth quasi-projective scheme.
Denote the generic fiber of the universal abelian scheme over $\mathcal{A}_{g,D}[p]$ by $A$; this is an abelian variety over the function field $k = \mathbb{C}(\mathcal{A}_{g,D}[p])$ equipped with a polarization $\lambda$ of type $D$.
We prove that if $D$ is of the form of Theorem \ref{theorem: intro negative result}, then the answer to Question \ref{question: main question weaker variant} is no for $(A, \lambda)$.
We don't know if such examples exists over arithmetically interesting fields such as $\mathbb{Q}$ or number fields.
We also don't know whether the answer to Question \ref{question: main question poonen-stoll} is always yes when $\dim A \in \{3,4,5,6\}$; this would involve analyzing types $D$ where some $d_i$ is divisible by $4$.
See Section \ref{subsection: survey of known cases} for a survey of known results towards Question \ref{question: main question poonen-stoll} and its variants.

\subsection{Methods}

Assume that the degree of $\lambda\colon A\rightarrow A^{\vee}$ is invertible in $k$ and let $D = (d_1, \dots,d_g)$ be the type of $(A, \lambda)$.
Then the kernel $A[\lambda]$ is a finite \'etale group scheme (in other words, a Galois module) of order $(d_1\cdots d_g)^2$ equipped with an alternating nondegenerate pairing $e_{\lambda}\colon A[\lambda]\times A[\lambda]\rightarrow \G_m$, the Weil pairing.
The key first step is to prove that the answer to Question \ref{question: main question poonen-stoll} only depends on the isomorphism class of the pair $(A[\lambda],e_{\lambda})$, as we now explain.

Define a \emph{theta group} for $(A[\lambda],e_{\lambda})$ to be a central extension of algebraic $k$-groups
\[
1\rightarrow \mathbb{G}_m \rightarrow \mathcal{G}\rightarrow A[\lambda]\rightarrow 1
\]
such that for all $\tilde{x}, \tilde{y} \in \mathcal{G}$ lifting $x,y\in A[\lambda]$, the commutator $[\tilde{x}, \tilde{y}]\in \mathbb{G}_m$ equals $e_{\lambda}(x,y)$. 
If $k$ is algebraically closed, there is a unique isomorphism class of theta groups for $(A[\lambda], e_{\lambda})$.
In general, both the existence and uniqueness of theta groups can fail.
The primordial example of a theta group is due to Mumford \cite{Mumford-equationsdefiningabelianvarieties}: if $L$ is a line bundle on $A$ with $\phi_L = \lambda$, then
\[
\mathcal{G}(L) = \{(a,\varphi) \colon a\in A[\lambda],\, \varphi\colon L \xrightarrow{\sim} t_a^* L\}
\]
is a theta group for $(A[\lambda], e_{\lambda})$.
We call a theta group $\mathcal{G}$ for $(A[\lambda], e_{\lambda})$ \emph{linear} if it admits an algebraic representation of dimension $d_1\cdots d_g$ on which $\G_m$ acts via scalar multiplication.
The formula $(a, \varphi)\cdot s = t_{-a}^*(\varphi(s))$ defines an action of $\mathcal{G}(L)$ on $\HH^0(A, L)$ (which has dimension $d_1\cdots d_g$ by Riemann--Roch), showing that Mumford theta groups $\mathcal{G}(L)$ are linear.
In this notation, we can state: 
\begin{thmtheta}\label{theorem: intro theta group condition}
    Let $\lambda\colon A\rightarrow A^{\vee}$ be a polarization whose degree is invertible in $k$.
    Then the following statements are equivalent:
 \begin{enumerate}
       \item There exists an $A$-torsor $X$ and a line bundle $L$ on $X$ such that $\lambda =\phi_L$.
        \item There exists a linear theta group for $(A[\lambda], e_{\lambda})$.
    \end{enumerate}
    The following statements are also equivalent:
    \begin{enumerate}
        \item There exists an $A$-torsor $X$ and an element $[L]\in \PicS_X(k) = \Pic(X_{k^{\sep}})^{\Gal(k^{\sep}|k)}$ such that $\lambda =\phi_L$.
        \item There exists a theta group for $(A[\lambda], e_{\lambda})$.
    \end{enumerate}
\end{thmtheta}

The implications $(1) \Rightarrow (2)$ are straightforward and follow from generalizing the definition of $\mathcal{G}(L)$ to line bundles on $A$-torsors.
The reverse implications are more significant.
Once the situation is appropriately categorified, they follow from the fact that a fully faithful morphism between two gerbes is an isomorphism.
Theorem \ref{theorem: intro theta group condition} can be upgraded to an equivalence of categories and works over an arbitrary base scheme; see Theorems \ref{theorem: main iso gerbes} and \ref{theorem:shrodinger gerbe and obstruction gerbe are isomorphic}.
We also prove a version of Theorem \ref{theorem: intro theta group condition} for symmetric line bundles on ``symmetric torsors''; see Theorem \ref{theorem: main iso symmetric version} and Corollary \ref{corollary: equivalent conditions existence symmetric line bundle class on torsor}.

To prove Theorem \ref{theorem: intro odd degree case}, it suffices to prove by Theorem \ref{theorem: intro theta group condition} that if the degree of $\lambda$ is odd then there exists a linear theta group for $(A[\lambda], e_{\lambda})$. 
In this case, we can construct such a group ``by hand''.
Since the degree is odd the alternating pairing $e_{\lambda}$ admits a square root, i.e., there exists an alternating pairing $b$ on $A[\lambda]$ such that $b^2 = e$.
Endowing $\mathcal{G} = \G_m \times M$ with the multiplication $(\lambda,x)\cdot (\lambda',y) = (\lambda \lambda' b(x,y), x+y)$ defines a theta group for $(A[\lambda],e_{\lambda})$, and a result of Polishchuk \cite{Polishchuk-analogueWeilrepresentation} shows that $\mathcal{G}$ is linear. 
Again, Theorem \ref{theorem: intro odd degree case} works over an arbitrary base.
If we insist that the pair $(X,L)$ representing $\lambda$ is symmetric in a certain sense, then this pair is essentially unique; see Corollary \ref{corollary: strong form odd order polarization represented by symmetric bundle}.

The proof of Theorem \ref{theorem: intro positive result} relies on the following surprising consequence of Theorem \ref{theorem: intro theta group condition}: if there exists a different polarized abelian variety $(B, \mu)$ and an isomorphism $(A[\lambda], e_{\lambda}) \simeq (B[\mu], e_{\mu})$, then the answer to Question \ref{question: main question poonen-stoll} is the same for $(A, \lambda)$ and $(B, \mu)$!
To see this in action, consider the simplest nontrivial case where $(A, \lambda)$ is of type $D = (1, \dots, 1,2)$, so that $A[\lambda](\bar{k}) \simeq (\Z/2\Z)^2$.
The Galois action on $A[\lambda]$ is encoded by a representation $\rho_{A,\lambda}\colon \Gal_k \rightarrow \Sp_2(\F_2) = \GL_2(\F_2)$,  where $\Gal_k = \Gal(k^{\sep}|k)$.
But every such representation is isomorphic to the $2$-torsion representation of an elliptic curve: there exists an elliptic curve $E/k$ and an isomorphism of Galois modules $A[\lambda]\simeq E[2]$ intertwining the Weil pairings on both sides.
Since the answer to Question \ref{question: main question poonen-stoll} is clearly yes for $E$ with the polarization $[2]\colon E\rightarrow E$, the same is true for $(A, \lambda)$.
In the cases where $D = (1, \dots, 1,2,2)$ or $(1, \dots, 1,2,2,2)$, we similarly show that there exists a curve $C/k$ with Jacobian $J$ and an isomorphism $(A[\lambda],e_{\lambda})\simeq (J[2], e_2)$. 
In the first case, we use genus-$2$ hyperelliptic curves and the exceptional isomorphism $\Sp_4(\F_2)\simeq S_6$. 
In the second case, we use plane quartic curves with a flex point and the isomorphism $\Sp_6(\F_2)\simeq W(E_7)^+$, where $W(E_7)^+$ is the index-$2$ subgroup of the Weyl group of $E_7$.
As a by-product of our argument, we answer a question of Chidambaram \cite[Question 1.2]{Chidambaram-modpgaloisrepsnotarisingfromAVs} affirmatively.

To explain the proof of Theorem \ref{theorem: intro negative result}, we relate our problem to Galois cohomology.
For simplicity, assume that $k$ is of characteristic zero and contains a primitive $n$-th root of unity $\zeta_n$ for all $n\geq 1$.
Define the abelian group $M_D = (\Z/d_1\Z)^2\times \cdots \times (\Z/d_g\Z)^2$ and let $e_D\colon M_D\times M_D \rightarrow k^{\times}$ be the direct sums of the standard alternating pairings $e_i$ on $(\Z/d_i\Z)^2$ defined by $e_i(v,w) = \zeta_{d_i}^{\det(v \, w)}$.
Then $(A[\lambda],e_{\lambda})$ is a twist of the split pair $(M_D,e_D)$, in the sense that there exists an isomorphism $A[\lambda](\bar{k})\simeq M_D$ intertwining $e_{\lambda}$ with $e_D$.
A choice of such an isomorphism determines a continuous homomorphism $\rho_{A,\lambda}\colon \Gal_k\rightarrow \Sp(M_D)$, where we write $\Sp(M_D) = \Aut(M_D, e_D)$.

We can also explicitly construct a standard theta group $\mathcal{G}_D$ for the split pair $(M_D,e_D)$; see Section \ref{subsec: abstract theta groups}.
Let $\Aut(\mathcal{G}_D)$ be the group of automorphisms $\mathcal{G}_D$ that induce the identity on $\G_m$. 
This determines an exact sequence of finite groups
\begin{align}\label{equation: intro Aut exact sequence}
    1\rightarrow M_D\rightarrow \Aut(\mathcal{G}_D)\rightarrow \Sp(M_D)\rightarrow 1,
\end{align}
where $M_D$ is identified with the set of inner automorphisms of $\mathcal{G}_D$.
Since theta groups for $(A[\lambda],e_{\lambda})$ are twists of $\mathcal{G}_D$, Theorem \ref{theorem: intro theta group condition} has the following concrete consequence (Lemma \ref{lemma: theta group exists iff embedding problem has solution}): the answer to Question \ref{question: main question weaker variant} is yes if and only if
\begin{center}
$\rho_{A, \lambda}\colon \Gal_k\rightarrow \Sp(M_D)$ lifts to a homomorphism $\tilde{\rho}_{A, \lambda}\colon \Gal_k\rightarrow \Aut(\mathcal{G}_D)$ under \eqref{equation: intro Aut exact sequence}.
\end{center}
Lifting problems (sometimes called embedding problems) for Galois representations such as this one are well studied, see for example \cite{embeddingproblem}.
Let $c_D\in \HH^2(\Sp(M_D), M_D)$ be the element in group cohomology classifying the extension \eqref{equation: intro Aut exact sequence}.
If $d_1\dots d_g$ is odd, then \eqref{equation: intro Aut exact sequence} canonically splits and hence $c_D=0$; this explains the bare hands construction of a theta group for $(A[\lambda],e_{\lambda})$ in the proof of Theorem \ref{theorem: intro odd degree case}.
On the other hand, if $\lambda$ has even degree then \eqref{equation: intro Aut exact sequence} is typically not split.
For example, if $D = (1, \dots, 1, 2, \dots, 2)$, where $2$ occurs $m$ times, then $c_D\neq 0$ if $m\geq 3$.
Even so, the proof of Theorem \ref{theorem: intro positive result} shows that if $m=3$ then every continuous representation $\rho\colon \Gal_k\rightarrow \Sp(M_D)$ lifts to a representation $\tilde{\rho}\colon \Gal_k\rightarrow \Aut(\mathcal{G}_D)$!
Serre has studied this curious phenomenon and has dubbed it negligible cohomology \cite[Chapter III, Appendix 2]{Serre-galoiscohomology}.
Generally speaking, if $K$ is a field, $G$ is a finite group and $M$ a $G$-module, a group cohomology class $c\in \HH^i(G,M)$ is said to be \emph{negligible over $K$} if for every field extension $L/K$ and continuous homomorphism $f\colon \Gal_L\rightarrow G$, $f^*(c) = 0$ in $\HH^i(\Gal_L, M)$. 
In this terminology, every Galois representation $\Gal_k\rightarrow \Sp(M_D)$ lifts to $\Aut(\mathcal{G}_D)$ if and only if $c_D\in \HH^2(\Sp(M_D), M_D)$ is negligible over $\Q(\mu_{\infty}) = \cup_{n\geq 1} \Q(\zeta_n)$.

In a remarkable recent advance, Merkurjev and Scavia have determined the subgroup of negligible classes of $\HH^2(G, M)$ over fields containing sufficiently many roots of unity \cite{MerkurjevScavia-negligiblecohomology}.
As an application of their result, they show that there exist $3$-dimensional Galois representations mod $p$ that do not lift mod $p^2$ for every odd prime $p$.
Applying their computation to our setting, we show that $c_D$ is \emph{not} negligible if $D=(1, \dots, 1,2, \dots, 2)$ and $2$ occurs at least $4$ times.
This implies that there exists a field $k$ and Galois representation $\rho\colon \Gal_k\rightarrow \Sp(M_D)$ that does not lift to $\Aut(\mathcal{G}_D)$. 
However, this does not quite show Theorem \ref{theorem: intro negative result} yet, since we don't know whether the non-liftable representation $\rho\colon \Gal_k\rightarrow \Sp(M_D)$ is of the form $\rho_{A,\lambda}$ for some polarized abelian variety $(A, \lambda)$ of type $D$.

To overcome this wrinkle, we use an even more recent result of Totaro who generalized the Merkurjev--Scavia computation to the setting of \'etale cohomology and twisted Chow groups \cite{Totaro-Chowringtwistedcoefficients}.
Let $\mathcal{A}_{g,D}[p]\rightarrow \Spec(\mathbb{C})$ denote the moduli space of $g$-dimensional polarized abelian varieties of type $D$ with symplectic level-$p$ structure for some prime $p$ not dividing $d_g$.
Then we define a class $\tilde{c}_{D}\in \HH^2(\mathcal{A}_{g,D}[p], M_D)$, which (ignoring the level-$p$ structure) can be thought of as the universal obstruction to lifting the representations $\rho_{A,\lambda}$ to $\Aut(\mathcal{G}_D)$. 
Using Totaro's results and calculations with Picard groups of covers of $\mathcal{A}_{g,D}[p]$, we show that, if $D$ is of the form of Theorem \ref{theorem: intro negative result}, then $\tilde{c}_D$ is nonzero when restricted to the generic point of $\mathcal{A}_{g,D}[p]$. 
This implies that the answer to Question \ref{question: main question weaker variant} is no for the generic polarized abelian variety over the function field of $\mathcal{A}_{g,D}[p]$, proving Theorem \ref{theorem: intro negative result}.
To calculate these Picard groups, we use Borel's determination of the stable cohomology of arithmetic groups \cite{Borel-stablerealcohomologyarithmeticgroupsI} and we explicitly calculate abelianizations of certain arithmetic subgroups of $\Sp_{2g}(\Q)$.

\subsection{Organization}

In the preliminary Section \ref{section: preliminaries}, we fix our notation and collect some facts from algebraic geometry and abelian schemes over general bases.
In Section \ref{sec: polarizations and theta groups}, we systematically analyze theta groups and prove Theorem \ref{theorem: intro theta group condition} and its analogue for symmetric line bundles.
In Section \ref{section: symplectic modules admitting a theta group}, we prove Theorems \ref{theorem: intro odd degree case} and \ref{theorem: intro positive result}.
We also collect all known results towards Question \ref{question: main question poonen-stoll} and its variants and discuss logical implications between these variants in Section \ref{subsection: survey of known cases}.
In the final Section \ref{section: negative result}, we discuss moduli spaces of abelian varieties and group cohomology of certain subgroups of twisted symplectic groups $\pargroup{g}{D}(\Z)$, proving Theorem \ref{theorem: intro negative result}.

Readers who are only interested in the proof of Theorem \ref{theorem: intro negative result} (which only uses the easy direction of Theorem \ref{theorem: intro theta group condition}) can jump straight to Section \ref{section: negative result} after reading Sections \ref{subsec: symplectic modules}, \ref{subsec: abstract theta groups} and \ref{subsec: Mumford theta groups}.

\subsection{Acknowledgements}

As will be clear to the reader, this paper owes a great intellectual debt to the mathematics of David Mumford.
We thank 
Shiva Chidambaram, 
Adam Morgan, 
Ben Moonen, 
Bjorn Poonen, 
Oscar Randal-Williams, 
Federico Scavia and 
Gerard van der Geer
for helpful conversations.
This research was carried out while the author was a Research Fellow at St John's College, University of Cambridge, which we thank for excellent working conditions.

\section{Preliminaries}\label{section: preliminaries}

We start by setting up notation and recalling facts about torsors, Picard functors, abelian schemes, polarizations, stacks and gerbes. 
These are mostly standard and can be skipped on a first reading.

\subsection{Schemes and sheaves}\label{subsec: notation}

If $X\rightarrow S$ is a morphism of schemes and $T\rightarrow S$ a morphism we write $X_T$ for the base change $X\times_ST$, viewed as a $T$-scheme. If $T = \Spec A$ is an affine scheme we also write $X_A$ for $X_T$. 
Write $X(T)$ for the set of sections of $X_T\rightarrow T$.
Again we write $X(A)$ when $T = \Spec A$.
We will often use the Yoneda embedding of the category of $S$-schemes into the category of sheaves on the (big) \'etale or fppf site on $S$.
Under this embedding, finite \'etale morphisms $X\rightarrow S$ correspond to 
finite \'etale locally constant sheaves \cite[Tag \href{https://stacks.math.columbia.edu/tag/03RV}{03RV}]{stacksproject}.
If $\sh{F}$ is a sheaf of abelian groups on the fppf site on $S$, write $\HH^i(S, \sh{F})$ for its associated fppf cohomology groups.
If this sheaf arises as the fppf sheaffication of an \'etale sheaf $\sh{G}$, then $\HH^i(S, \sh{F})$ coincides with \'etale cohomology of $\sh{G}$ \cite[Tag \href{https://stacks.math.columbia.edu/tag/0DDU}{0DDU}]{stacksproject}, which we also denote by $\HH^i(S, \sh{G})$.

If $k$ is a field, let $k^{\sep}\subset \bar{k}$ be a choice of separable and algebraic closure and let $\Gal_k = \Gal(k^{\sep}/k)$ be its absolute Galois group.
We will often use the equivalence of categories $X\mapsto X(k^{\sep})$ between finite \'etale schemes over $k$ (called finite $k$-sets) and finite sets with a continuous $\Gal_k$-action.
Given a finite \'etale $k$-scheme $X$, denote the cardinality of $X(k^{\sep})$ by $\#X$.
Similar remarks apply to the equivalence between finite \'etale group schemes over $k$ (finite $k$-groups) and finite groups with a continuous $\Gal_k$-action.

\subsection{Torsors}

Let $G$ be a fppf group scheme or group sheaf on the fppf site over a scheme $S$.
We write $\HH^1(S,G)$ for the set of isomorphism classes of (left) sheaf torsors under $G$ over $S$. If $S = \Spec R$ we write $\HH^1(R,G)$ for the same object.
If $G$ is commutative, $\HH^1(S,G)$ coincides with fppf cohomology with coefficients in $G$, so there is no conflict of notation \cite[Tag \href{https://stacks.math.columbia.edu/tag/03AG}{03AG}]{stacksproject}. 
If $G$ is smooth, then every torsor is trivialized \'etale locally, and if further $S$ is the spectrum of a field $k$, then $\HH^1(k,G)$ coincides with (possibly non-abelian) Galois cohomology $\HH^1(\Gal_k, G(k^{\sep}))$ defined using cocycles \cite[Chapter III]{Serre-galoiscohomology}.

Every $G$-torsor is represented by an algebraic space, but not necessarily by a scheme, even if $G\rightarrow S$ is an abelian scheme \cite[Section XIII 3.2, page 200]{Raynaud-Faisceauxamples}.
However, if either $G\rightarrow S$ is affine, or $G\rightarrow S$ is separated of finite presentation and $S$ is locally noetherian and $\dim S\leq 1$, then every $G$-torsor is representable by a scheme, see \cite[Theorem 6.5.10]{Poonen-rationalpointsonvarieties}.

\subsection{Picard schemes}

If $X$ is a scheme, we denote the Picard group of isomorphism classes of line bundles (equivalently, invertible sheaves) on $X$ by $\Pic(X)$.
Let $X\rightarrow S$ be a smooth proper morphism of schemes with geometrically connected fibers. 
The Picard functor $\PicS_{X/S}$ (or simply $\PicS_X$ whenever $S$ is clear from the context) is the \'etale sheaffication on the category of $S$-schemes of the presheaf $T\mapsto \Pic(X_T)/\Pic(T)$, see \cite[Section 8.1]{BLR-NeronModels}. (This presheaf is already a sheaf in the Zariski topology. Moreover, the sheaffication step is unnecessary when $X\rightarrow S$ has a section.)
The Picard functor is represented by an algebraic space over $S$ by results of Raynaud \cite[Section 8.3]{BLR-NeronModels}, and we denote this algebraic space again by $\PicS_{X/S}$ or $\PicS_X$.
If $X\rightarrow S$ is projective or if $S$ is the spectrum of a field, then $\PicS_X$ is represented by a scheme \cite[Section 8.2]{BLR-NeronModels}.
Let $\PicS_{X/S}^0\subset \PicS_{X/S}$ (or simply $\PicS_X^0$) be the subsheaf consisting of those elements whose restriction to every closed point $s\colon \Spec(k) \rightarrow S$ lies in the identity component of the group scheme $\PicS_{X_s/k}$.
This is an open subfunctor of $\PicS_{X}$ \cite[Section 8.4, p.\,233]{BLR-NeronModels}, so is representable whenever $\PicS_{X}$ is.

There exists an exact sequence \cite[Section 8.1, Proposition 4]{BLR-NeronModels}
\begin{align}\label{eq: exact sequence usual pic scheme pic}
    1\rightarrow \Pic(S) \rightarrow \Pic(X) \rightarrow \PicS_{X/S}(S) 
    \xrightarrow{\mathrm{ob}} \HH^2(S, \G_m) \rightarrow \HH^2(X, \G_m)
\end{align}
and $\mathrm{ob}$ measures the obstruction for a class $\ell\in \PicS_{X/S}(S)$ to be represented by a line bundle on $X$.
Given a line bundle $L$ on $X$, we denote the image of its isomorphism class under the map $\Pic(X)\rightarrow \PicS_X(S)$ by $[L] \in \PicS_X(S)$. 

Suppose that $X\rightarrow S$ admits a section $e\colon S\rightarrow X$.
Then the obstruction map $\mathrm{ob}\colon \PicS_X(S)\rightarrow \HH^2(S, \G_m)$ of \eqref{eq: exact sequence usual pic scheme pic} vanishes.
Moreover $\PicS_{X}$ is isomorphic to the functor sending an $S$-scheme $T$ to isomorphism classes rigidified line bundles on $X_T$, i.e., the data of a line bundle $L$ on $X_T$ together with an isomorphism between $e^*L$ and the trivial line bundle on $T$.
If $S = \Spec(k)$ where $k$ is a field, then $\PicS_{X}(k)=\Pic(X_{k^{\sep}})^{\Gal_k}$ and the inclusion $\Pic(X) \subset \PicS_{X/k}(k)$ is an equality when $X(k) \neq  \varnothing$, but not in general.

Define the N\'eron--Severi group sheaf as the sheaf quotient $\NSS_X = \NSS_{X/S} = \PicS_{X/S}/\PicS_{X/S}^0$ and define the N\'eron--Severi group as its group of $S$-points $\NSS_{X/S}(S)$.
The exact sequence of sheaves $1\rightarrow \PicS^0_X \rightarrow \PicS_{X}\rightarrow \NSS_{X}\rightarrow 1$ induces an exact sequence 
\begin{align}\label{equation:exactseqPicNS}
1\rightarrow \PicS_{X}^0(S) \rightarrow \PicS_{X}(S) \rightarrow \NSS_{X}(S)\rightarrow \HH^1(S, \PicS_{X}^0).
\end{align}

\subsection{Polarizations on abelian varieties}\label{subsec: polarizations on abelian varieties}

For general references concerning abelian schemes, see \cite[Chapter 6]{MumfordFogartyKirwan-GIT} and \cite[Chapter 1, Section 1]{FaltingsChai-degenerations}, or \cite[Section 1]{AchterCasalainaMartin-puttingpinPrym} for a nice summary. 
Let $A\rightarrow S$ be an abelian scheme, i.e., a smooth proper group scheme with geometrically connected fibers.
Then $A^{\vee} \coloneqq \PicS_{A/S}^0$ is representable by an abelian scheme over $S$ (see \cite[Theorem 1.9]{FaltingsChai-degenerations}) and is called the \emph{dual abelian scheme} of $A$.
If $L$ is a line bundle on $A$, then $[L] \in A^{\vee}(S)$ if and only if for every geometric point $s= \Spec(k)\rightarrow S$ and every $x\in A_{s}(k)$, $L_{s}$ is isomorphic to $t_x^* L_{s}$, where $t_x\colon A_{s}\rightarrow A_s$ denotes translation by $x$; see \cite[Section 2.1]{olsson-compactifymoduli}.

Let $X\rightarrow S$ be a (left) $A$-torsor. 
If $T\rightarrow S$ is a surjective \'etale morphism and $\alpha\colon X_T\xrightarrow{\sim} A_T$ is a trivialization, then $\alpha$ induces isomorphisms $(\PicS_{X}^0)_T\xrightarrow{\sim} (\PicS^0_{A})_T$ and $(\NSS_{X})_T\xrightarrow{\sim} (\NSS_{A})_T$. 
Since translation by a point $t_x\colon A\rightarrow A$ induces the identity on $\PicS_A^0$ and $\NSS_A$, these isomorphisms do not depend on the choice of trivialization.
Therefore, they descend to canonical isomorphisms $\PicS_{X}^0 \simeq \PicS_{A}^0$ and  $\NSS_{X} \simeq \NSS_{A}$; we will make these identifications without further mention (see also \cite[Theorem 3.0.3]{alexeev-completemodulisemiabelian}). 
We thus obtain the fundamental exact sequence
\begin{align}\label{equation: exact sequence Pic_X of torsor}
    1 \rightarrow A^{\vee}\rightarrow \PicS_X \rightarrow \NSS_A\rightarrow 1.
\end{align}
Every line bundle $L$ on $A$ induces a morphism $\phi_{L} \colon A\rightarrow A^{\vee}$ that on $T$-points is given by $a\mapsto [t^*_aL \otimes L^{-1}]$, where $t_a\colon A_T\rightarrow A_T$ denotes translation by $a\in A(T)$.
By the theorem of the square, $\phi_L$ is a homomorphism.
It only depends on the image of $L$ in $\NSS_{A/S}(S)$.
It is symmetric, in the sense that $\phi_L^{\vee} = \phi_L$ using the double duality $A = A^{\vee \vee}$. 
More generally, if $X$ is an $A$-torsor and $L$ is a line bundle on $X$, the same formula defines a homomorphism $\phi_{L}\colon A\rightarrow A^{\vee}$, using the identification $\PicS_X^0 = A^{\vee}$ (see also \cite[Lemma 3.0.1]{alexeev-completemodulisemiabelian}).
Even more generally, if $\ell \in \PicS_{X}(S)$ then by descent there exists a unique morphism $\phi_{\ell}\colon A\rightarrow A^{\vee}$ such that $(\phi_{\ell})_T = \phi_{L}$ whenever $T\rightarrow S$ is surjective \'etale and $L$ is a line bundle on $X_T$ with $[L] = \ell_T$ under \eqref{eq: exact sequence usual pic scheme pic}.

By definition (\cite[Definition 1.6]{FaltingsChai-degenerations} or \cite[Chapter 6, Section 2]{MumfordFogartyKirwan-GIT}), a homomorphism of abelian $S$-schemes $\lambda\colon A\rightarrow A^{\vee}$ is called a \emph{polarization} if there exists a surjective \'etale morphism $T\rightarrow S$ such that $\lambda_T = \phi_{L}$ for some relatively ample line bundle $L$ on $A_T\rightarrow T$.
If $L$ and $L'$ are line bundles on $A$, then $\phi_L = \phi_{L'}$ if and only if $L$ and $L'$ have the same image in $\NSS_{A/S}(S)$.
Therefore the assignment $L\mapsto \phi_L$ induces a bijection between the set of elements of $\NSS_{A/S}(S)$ that are \'etale locally the image of an ample line bundle on $A$ and the set of polarizations on $A$. 
Consequently, we may and often do identify a polarization $\lambda$ with the corresponding element of $\NSS_{A/S}(S)$.
If $X$ is an $A$-torsor and $\ell\in \PicS_X(S)$ satisfies $\phi_{\ell} = \lambda$, then we say that $\ell$ \emph{represents} $\lambda$.
Given any class $\lambda \in \NSS_{A/S}(S)$, write $\PicS^{\lambda}_A$ for the preimage of $\lambda$ under the projection $\PicS_A\rightarrow \NSS_A$, and similarly define $\PicS_X^{\lambda}$ for any $A$-torsor $X$ using \eqref{equation: exact sequence Pic_X of torsor}.
The class $\PicS_X^{\lambda}$ is a torsor under $\PicS^0_X  = \PicS^0_A = A^{\vee}$.
By definition, the class of $\PicS_X^{\lambda}$ in $\HH^1(S, \PicS_X^0)$ equals the image of $\lambda \in \NS_X(S)$ under the connecting map of \eqref{equation:exactseqPicNS}.

Given a polarization $\lambda$ on an abelian scheme $A$ over a field $k$, the order of the kernel $A[\lambda]$ (as a group scheme) is called the degree of $\lambda$, denoted by $\deg(\lambda)$.
If $\deg(\lambda)$ is invertible in $k$, there further exist unique positive integers $d_1 \mid d_2 \mid\dots \mid d_g$ with $g = \dim A$ such that $A[\lambda](k^{\sep}) \simeq (\Z/d_1\Z)^2\times \dots \times (\Z/d_g\Z)^2$; we call $D = (d_1, \dots, d_g)$ the type of $\lambda$.
If $(A, \lambda)$ is a polarized abelian scheme over a general scheme $S$, then the function which associates to a point $s\in S$ the degree of $(A, \lambda)_s$ is locally constant; if this degree is invertible for every $s$, we say the degree of $(A, \lambda)$ is invertible on $S$.
If the degree of $(A, \lambda)$ is invertible on $S$, then the type of $(A, \lambda)_s$ is also locally constant.
In this paper, we will only consider polarizations whose degree is invertible on the base scheme.

\subsection{Stacks and gerbes}

We will make use of basic properties of stacks and gerbes; we briefly recall these notions here.
Let $\mathcal{C}$ be a site, for example the (big) \'etale or fppf site of a scheme $S$.
Given an object $U$ of $\mathcal{C}$, let $\mathcal{C}/U$ be the site whose underlying category is the slice category over $U$ and whose coverings are restrictions of coverings in $\mathcal{C}$.
Given a category fibered in groupoids $\sh{X}\rightarrow \mathcal{C}$ and two objects $x,y$ of $\sh{X}(U)$, let $\IsomS(x,y)$ denote the presheaf on $\mathcal{C}/U$ with the property that $\IsomS(x,y)(T) = \Hom_{\sh{X}(T)}(f^*x,f^*y)$ for all objects $f\colon T\rightarrow U$ in $\mathcal{C}/U$. Similarly write $\AutS(x) = \IsomS(x,x)$.

Recall from \cite[Section 4]{Vistoli-grothendiecktopologies} or \cite[Tag \href{https://stacks.math.columbia.edu/tag/02ZH}{02ZH}]{stacksproject} that a \emph{stack} is a category fibered in groupoids $\sh{X} \rightarrow \mathcal{C}$ such that for every object $U$ of $\mathcal{C}$, the presheaf $\IsomS(x,y)$ is a sheaf for all objects $x,y\in\sh{X}(U)$ and such that the pullback functor $\sh{X}(U)\rightarrow \sh{X}(\{U_i\rightarrow U\})$ to the groupoid $\sh{X}(\{U_i\rightarrow U\})$ of descent data in the sense of \cite[Section 4.1.2]{Vistoli-grothendiecktopologies} is an equivalence of categories for every covering $\{U_i\rightarrow U\}$.

As an example, let $\sh{G}$ be a sheaf of groups on $\mathcal{C}$.
The \emph{classifying stack} for $\sh{G}$, denoted $\mathrm{B}\sh{G}$, is the stack on $\mathcal{C}$ such that for each $U$ of $\mathcal{C}$, the objects of the groupoid $(\mathrm{B}\sh{G})(U)$ are $\sh{G}_U$-torsors on $\mathcal{C}/U$, and morphisms given by isomorphisms of $\sh{G}_U$-torsors.
This is indeed a stack by \cite[Tag \href{https://stacks.math.columbia.edu/tag/04TQ}{04TQ}]{stacksproject}.

A stack $\sh{X}$ on $\mathcal{C}$ is a \emph{gerbe} \cite[Tag \href{https://stacks.math.columbia.edu/tag/06NY}{06NY}]{stacksproject} if it satisfies the following two properties: for every object $U$ of $\mathcal{C}$, there exists a cover $\{U_i \rightarrow U\}$ such that the groupoid $\sh{X}(U_i)$ is non-empty for all $i$; for every object $U$ of $\mathcal{C}$ and objects $x,y$ in $\sh{X}(U)$, there exists a covering $\{U_i \rightarrow U\}$ such that $x|_{U_i} \simeq y|_{U_i}$ for all $i$.
For example, classifying stacks of sheaves of groups are gerbes.

We will encounter a few gerbes in this paper in Section \ref{sec: polarizations and theta groups}. 
We will often use the following formal but very useful lemma: 

\begin{lemma}\label{lemma: fully faithful morphism gerbes}
    Let $F\colon \sh{X}\rightarrow \sh{Y}$ be a morphism between gerbes such that for every object $U$ of $\mathcal{C}$ and $x$ of $\sh{X}(U)$, the induced map $\Aut_{\sh{X}(U)}(x)\rightarrow \Aut_{\sh{Y}(U)}(F(x))$ is an isomorphism. 
    Then $F$ is an isomorphism. 
    In particular, $F_U$ is an equivalence of groupoids for all $U$.
\end{lemma}
\begin{proof}
    This must be well known; we give a quick proof.
    It suffices to prove that $F$ is fully faithful and essentially surjective.
    To prove $F$ is fully faithful, we need to show that for every $U$ of $\mathcal{C}$ and $x,y$ of $\sh{X}(U)$, the map of sheaves $\IsomS(x,y) \rightarrow \IsomS(F(x),F(y))$ is an isomorphism.
    Since this can be checked locally and since $\sh{X}$ is a gerbe, we may assume $x = y$, in which case it follows from our assumptions.
    We now prove that for every $U$ and $y$ in $\sh{Y}(U)$, $y$ lies in the essential image of $F_U$.
    Since objects in $\sh{X}$ exists locally and since all objects in $\sh{Y}$ are locally isomorphic, there exists a covering $\{U_i \rightarrow U\}$, objects $x_i$ in $\sh{X}(U_i)$ and isomorphisms $F_{U_i}(x_i)\simeq y|_{U_i}$.
    Since $F_V$ is fully faithful for each $V$, there exists a unique descent datum $D = (\{x_i\}, \{\alpha_{ij}\})$ with respect to $\{U_i\rightarrow U\}$ that maps to the descent datum of $y$ under $F$.
    Since $\sh{X}$ is a stack, there exists an object $x$ in $\sh{X}(U)$ corresponding to this descent datum, and this object satisfies $F_U(x) \simeq y$.
\end{proof}

Gerbes are twisted versions of classifying stacks, in the following sense.

\begin{lemma}\label{lemma: nonempty gerbe is neutral}
Let $\sh{X}$ be a gerbe, suppose that $S$ is a terminal object of $\mathcal{C}$ and suppose that there exists an object $x$ in $\sh{X}(S)$.
Then the assignment $y\mapsto \IsomS(y,x)$ induces an isomorphism $\sh{X}\simeq \mathrm{B}\AutS(x)$.
\end{lemma}
\begin{proof}
    Since $\sh{X}$ is a gerbe, $x$ and $y$ are locally isomorphic, so $\IsomS(y,x)$ with its natural left action is a torsor under $\AutS(x)$.
    Moreover if $y,y'$ are objects of $\sh{X}(T)$, then the map $\Isom_{\sh{X}(T)}(y,y') \rightarrow \Isom_{B\sh{G}(T)}(\IsomS(y,x),\IsomS(y',x))$ is a bijection.
    Now apply Lemma \ref{lemma: fully faithful morphism gerbes}.
\end{proof}
Whenever we are in the situation of Lemma \ref{lemma: nonempty gerbe is neutral}, we say the gerbe is \emph{neutral}.

When all automorphism groups are commutative, we can say more.
Suppose that $\sh{X}$ is a gerbe such that for every object $U$ of $\mathcal{C}$ and object $x$ of $\sh{X}(U)$, the group sheaf $\AutS(x)$ on the slice category $\mathcal{C}/U$ is commutative. 
Then by \cite[Tag \href{https://stacks.math.columbia.edu/tag/0CJY}{06NY}]{stacksproject} there exists a sheaf of commutative groups $\sh{G}$ on $\mathcal{C}$ such that $\sh{G}|_U\simeq \AutS(x)$ for every object $x$ of $\sh{X}(U)$ and every $U$ of $\mathcal{C}$.
In that case, we say the gerbe $\sh{X}$ is \emph{banded by} $\sh{G}$.

\begin{lemma}\label{lemma:classificationgerbesH2}
    Let $\mathcal{C}$ be the big etale or fppf site of a scheme $S$ and let $\sh{G}$ be a sheaf of commutative groups on $\mathcal{C}$.
    Then there exists a canonical bijection $\sh{X}\mapsto [\sh{X}]$, compatible with base change on $S$, between the set of isomorphism classes of gerbes banded by $\sh{G}$ and $\HH^2(S, \sh{G})$, sending a neutral gerbe to the trivial class.
\end{lemma}
\begin{proof}
See \cite[Th\'eor\`eme IV.3.4.2]{Giraud-cohomologienonabelienne}.
\end{proof}

\section{Polarizations and theta groups}\label{sec: polarizations and theta groups}

We start by discussing formal properties of commutative finite \'etale group schemes with a nondegenerate alternating pairing such as $(A[\lambda],e_{\lambda})$, which we call symplectic modules.
We then discuss abstract theta groups, Mumford theta groups and linear theta groups.
Theorem \ref{theorem: intro theta group condition} follows from combining Corollaries \ref{corollary: theorem theta first part body of text} and \ref{corollary: theorem theta second part body of text}.
We end by proving an analogue of Theorem \ref{theorem: intro theta group condition} for symmetric line bundles on symmetric torsors.

\subsection{Symplectic modules}\label{subsec: symplectic modules}

Let $S$ be a scheme.

\begin{definition}
A \emph{symplectic module} over $S$ is a pair $(M,e)$, where $M\rightarrow S$ is a commutative finite \'etale group scheme and $e\colon M\times M \rightarrow \G_m$ is a nondegenerate alternating pairing. 
(Nondegenerate here means that $e$ induces an isomorphism $M\xrightarrow{\sim} M^{\vee}\coloneqq \HomS(M,\G_m)$.)
\end{definition}

If $(M,e)$ is a symplectic module and $T\rightarrow S$ a morphism, the base change $(M,e)_T = (M_T,e_T)$ is a symplectic module over $T$.
An isomorphism of symplectic modules $(M,e)\xrightarrow{\sim} (M',e')$ is simply an isomorphism of $S$-group group schemes $M\rightarrow M'$ intertwining $e$ and $e'$; denote the set of such isomorphisms by $\Isom((M,e), (M',e'))$.
We recall the classification of symplectic modules over a separably closed field.

\begin{definition}\label{definition:type}
    Let $g\geq 1$ be an integer.
    A \emph{type of length }$g$ is a sequence of positive integers $D = (d_1, \dots, d_g)$ such that  $d_i$ divides $d_{i+1}$ for all $i=1,\dots, g-1$.
    Let $\#D = d_1\cdots d_g$. 
\end{definition}

Given a type $D$, we define the \emph{standard symplectic module $(M_D,e_D)$ of type $D$} over $\Z[1/\#D]$ as follows. 
Consider the constant group scheme $K_D = \Z/d_1\Z \times \dots \times \Z/d_g \Z$ and let $K_D^{\vee} = \HomS(K_D, \G_m) = \mu_{d_1}\times \cdots\times\mu_{d_g}$ be its Cartier dual.
Let $M_D = K_D \times K_D^{\vee}$, and define the pairing $e_D$ via $e_D((x,\chi), (x',\chi')) = \chi'(x)\chi(x')^{-1}$.
Since $e_D$ is isotropic on $K_D$ and $K_D^{\vee}$ and puts these two groups in perfect duality, this pairing is nondegenerate.
If $\#D$ is invertible on $S$, we obtain by base change the standard symplectic module of type $D$ over $S$ which we again denote by $(M_D,e_D)$.

\begin{lemma}\label{lemma: symplectic modules locally constant}
    Let $(M,e)$ be a symplectic module over $S$ and assume $S$ connected.
    Then there exists a unique type $D = (d_1, \dots, d_g)$ with the property that $d_1\geq 2$, $\#D$ invertible on $S$ and such that $(M,e)_T \simeq (M_D, e_D)$ for some surjective \'etale $T\rightarrow S$.
\end{lemma}
\begin{proof}
    There exists a surjective \'etale morphism $T\rightarrow S$ such that $M_T$ is constant, corresponding to a finite abelian group $A$.
    Since $M^{\vee}\simeq M$ is also \'etale, the order of $A$ is invertible on $S$. After a further base change we may assume $M_T^{\vee}$ is constant and corresponds to $\Hom(A, \Q/\Z)$. 
    Then $e$ corresponds to a nondegenerate bilinear alternating form $A\times A\rightarrow \Q/\Z$.
    In this case, it is folklore and elementary that there exists a Lagrangian decomposition $A \simeq L  \times L'$ that puts $L'$ and $L$ in duality; see for example \cite[Corollary 5.7]{PrasadShapiroVemuri-symplectic}.
    The lemma follows by taking $D$ to be the sequence of elementary divisors of $L$.
\end{proof}

We say a symplectic module $(M,e)$ \emph{has type $D$} if for every geometric point $\bar{s}\colon \Spec(k)\rightarrow S$ with $k$ a separably closed field, the pullback $(M,e)_{\bar{s}}$ is isomorphic to $(M_D,e_D)$ over $k$. 
In that case, an argument similar to the proof of Lemma \ref{lemma: symplectic modules locally constant} shows that $(M,e)_T\simeq (M_D,e_D)$ for some \'etale surjective base change $T\rightarrow S$.

\begin{definition}
    Let $(M,e)$ and $(M',e')$ be symplectic modules over $S$.
    Let $\IsomS((M,e),(M',e'))$ be the \'etale sheaf on $S$ with $\IsomS((M,e),(M',e'))(T) = \Isom((M,e)_T,(M',e')_T)$ for all morphisms $T\rightarrow S$.
    We write $\SpS(M) = \IsomS((M,e),(M,e))$ if the pairing $e$ is clear from the context.
    Similarly we write $\Sp(M) = \Isom((M,e),(M,e))$.
\end{definition}
Since $\IsomS((M,e), (M',e'))$ is finite and locally constant by Lemma \ref{lemma: symplectic modules locally constant}, it is represented by a finite \'etale scheme over $S$.
Hence $\SpS(M)\rightarrow S$ is a finite \'etale group scheme too.
Given a symplectic module $(M,e)/S$ of type $D$, $\IsomS((M,e),(M_D,e_D))$ receives a left action of $\SpS(M_D)$.
\begin{lemma}\label{lemma: symplectic modules vs SpD torsors}
    The assignment $(M,e) \mapsto \IsomS((M,e),(M_D,e_D))$ induces a bijection between isomorphism classes of the following objects:
    \begin{itemize}
        \item Symplectic modules over $S$ of type $D$;
        \item $\SpS(M_D)$-torsors over $S$.
    \end{itemize}
\end{lemma}
\begin{proof}
    This follows from \'etale descent (sometimes called the twisting principle) and Lemma \ref{lemma: symplectic modules locally constant}.
    (More formally, the category of symplectic modules of type $D$ is a stack in the \'etale topology on $S$ and is a gerbe by Lemma \ref{lemma: symplectic modules locally constant}. Now apply Lemma \ref{lemma: nonempty gerbe is neutral}.)    
\end{proof}

\begin{example}
If $D = (2,2,\dots,2)$ has length $n$, then $\SpS(M_D)$ is the constant group scheme $\Sp_{2n}(\F_2)$.
If $S$ is the spectrum of a field $k$, then isomorphism classes symplectic modules of type $D$ are in bijection with $\HH^1(k, \Sp_{2n}(\F_2))$, which is itself in bijection with conjugacy classes of continuous representations $\Gal_k\rightarrow \Sp_{2n}(\F_2)$.
\end{example}

\subsection{Abstract theta groups}\label{subsec: abstract theta groups}

Let $S$ be a scheme and $M\rightarrow S$ a finite \'etale commutative group scheme.
Let $1\rightarrow \G_m\rightarrow \mathcal{G}\rightarrow M\rightarrow 1$ be a central extension of group schemes.
If $T\rightarrow S$ is a morphism and $x,y\in M(T)$ lift to elements $\tilde{x}, \tilde{y} \in \mathcal{G}(T)$, then the commutator $[\tilde{x}, \tilde{y}] = \tilde{x}\tilde{y}\tilde{x}^{-1} \tilde{y}^{-1}$ lies in the subgroup $\G_m(T)$ and is independent of the choice of lift of $x$ and $y$.
By descent, it follows that the assignment $(x,y)\mapsto [\tilde{x}, \tilde{y}]$ defines a morphism $M \times M\rightarrow \G_m$.
This morphism is easily checked to be bilinear and alternating; it is called the \emph{commutator pairing} associated to the central extension.

\begin{definition}
Let $(M,e)$ be a symplectic module over $S$.
A \emph{theta group for $(M,e)$} is a central extension of group schemes
\begin{align}
1 \rightarrow \G_m \rightarrow \mathcal{G} \rightarrow M \rightarrow 1
\end{align}
whose commutator pairing equals $e$.
\end{definition}
The data of a theta group includes not just the group scheme $\mathcal{G}$ but also the central extension structure; we will usually suppress this additional data in the notation.
If $\mathcal{G}$ is a theta group for $(M,e)$ and $T\rightarrow S$ is a morphism, then $\mathcal{G}_T$ is a theta group for $(M,e)_T$.

Let $\mathcal{G}$ and $\mathcal{G}'$ be theta groups for the symplectic modules $(M,e)$ and $(M',e')$ respectively.
Define $\Isom(\mathcal{G}, \mathcal{G}')$ to be the set of isomorphisms of group schemes $\alpha\colon \mathcal{G} \rightarrow \mathcal{G}'$ such that $\alpha|_{\G_m}\colon \G_m\rightarrow \G_m$ is the identity.
If such an $\alpha$ exists, we say $\mathcal{G}$ and $\mathcal{G}'$ are isomorphic.
Define $\IsomS(\mathcal{G}, \mathcal{G}')$ to be \'etale sheaf on $S$ with $\IsomS(\mathcal{G}, \mathcal{G}')(T)=\Isom(\mathcal{G}_T, \mathcal{G}'_T)$ for every morphism $T\rightarrow S$.
There exists a morphism 
\begin{align}\label{eqn: iso theta groups induces iso symplectic modules}
\IsomS(\mathcal{G}, \mathcal{G}') \rightarrow \IsomS((M,e),(M',e'))
\end{align}
sending $\alpha\colon \mathcal{G} \rightarrow \mathcal{G}'$ to the induced isomorphism when quotienting out $\G_m$.
Given $\beta\in \Isom((M,e),(M',e'))$, let $\IsomS(\mathcal{G},\mathcal{G}';\beta)\subset \IsomS(\mathcal{G}, \mathcal{G}')$ be the fiber of \eqref{eqn: iso theta groups induces iso symplectic modules} above $\beta$, and similarly let $\Isom(\mathcal{G},\mathcal{G}';\beta)=\IsomS(\mathcal{G},\mathcal{G}';\beta)(S)$.
If $(M,e) = (M',e')$, we call elements of $\Isom(\mathcal{G}, \mathcal{G}';\Id_M)$ \emph{framed isomorphisms} and we say $\mathcal{G}$ and $\mathcal{G}'$ are \emph{framed isomorphic} if $\Isom(\mathcal{G}, \mathcal{G}';\Id_M)\neq \varnothing$. 
If $\mathcal{G} = \mathcal{G}'$, write $\AutS(\mathcal{G}) = \IsomS(\mathcal{G}, \mathcal{G})$ and $\AutS(\mathcal{G}; \Id)=\IsomS(\mathcal{G}, \mathcal{G}; \Id_M)$ and their $S$-points by $\Aut(\mathcal{G})$ and $\Aut(\mathcal{G};\Id)$.

We first classify theta groups over separably closed fields.
Given a type $D$ such that $\#D$ is invertible on $S$, we define the \emph{standard theta group of type }$D$ over $S$, following \cite[page 294]{Mumford-equationsdefiningabelianvarieties}.
In the notation of \S\ref{subsec: symplectic modules}, let $\mathcal{G}_D  =\G_m \times M_D =  \G_m \times K_D \times K_D^{\vee}$ and define the group law via 
\begin{align}\label{equation:standardthetagroupformula}
(\lambda, x, \chi)\cdot (\lambda', x', \chi') = (\lambda \lambda' \chi'(x), x+x', \chi+ \chi').
\end{align}
(We write the group law on $\G_m$ multiplicatively and on $K_D, K_D^{\vee}$ additively.)
A calculation then shows that, with the obvious inclusion from $\G_m$ and projection to $M_D$, $\mathcal{G}_D$ is a theta group for $(M_D,e_D)$.

\begin{lemma}\label{lemma: theta groups locally isomorphic}
    Let $(M,e)$ be a symplectic module over $S$ and let $\mathcal{G}, \mathcal{G}'$ be theta groups for $(M,e)$.
    Then there exists an \'etale surjective $T\rightarrow S$ such that $\mathcal{G}_T$ and $\mathcal{G}'_T$ are framed isomorphic, in other words such that $\IsomS(\mathcal{G}, \mathcal{G}';\Id)(T)\neq \varnothing$.
\end{lemma}
\begin{proof}
    The usual proof when $S$ is the spectrum of an algebraically closed field generalizes to an arbitrary base scheme $S$.
    Indeed, since the statement is local on $S$ we may assume by a standard argument that $S$ is connected. 
    (We explain the standard argument once: we may assume $S = \Spec(A)$ is affine, so there exists a finitely generated $\Z$-algebra $A_0\subset A$ such that $\mathcal{G}, \mathcal{G}'$ and $(M,e)$ are defined over $S_0 = \Spec(A_0)$, so we may replace $S$ by $S_0$ and assume $S$ is Noetherian, in which case the connected components of $S$ are open.)
    By Lemma \ref{lemma: symplectic modules locally constant} we may assume after \'etale base change that $(M,e) = (M_D,e_D)$ and $\mathcal{G}' = \mathcal{G}_D$ for some type $D$.
    After a further base change, there exist homomorphisms $K_D\rightarrow \mathcal{G}$, $K_D^{\vee}\rightarrow \mathcal{G}$ that split the homomorphisms $\mathcal{G}|_{K_D} \rightarrow K_D$ and $\mathcal{G}|_{K_D^{\vee}} \rightarrow K_D^{\vee}$ respectively.
    Using these splittings, we can write down an explicit framed isomorphism between $\mathcal{G}$ and $\mathcal{G}_D$, as in \cite[Corollary of Theorem 1]{Mumford-equationsdefiningabelianvarieties}.
\end{proof}

\begin{lemma}\label{lemma: automorphism groups theta groups}
    Let $\mathcal{G}$ be a theta group for a symplectic module $(M,e)$ over $S$.
    \begin{enumerate}
    \item The forgetful map $\AutS(\mathcal{G})\rightarrow \SpS(M)$ is surjective.
    \item Given a morphism $T\rightarrow S$ and $m \in M(T)$, let $\alpha_m$ be the isomorphism $\mathcal{G}_T\rightarrow \mathcal{G}_T$ defined by $\tilde{x}\mapsto e(m,x)\tilde{x}$, where $x\in M(T)$ is the image of $\tilde{x}$.
    Then the assignment $m\mapsto \alpha_m$ induces an isomorphism of group schemes $M\simeq \AutS(\mathcal{G};\Id)$. 
    \end{enumerate}
\end{lemma}
\begin{proof}
    We may assume $S$ is connected.
    Furthermore we may assume by Lemmas \ref{lemma: symplectic modules locally constant} and \ref{lemma: theta groups locally isomorphic} that $M = M_D$ and $\mathcal{G} = \mathcal{G}_D$ for some type $D$.
    Part 1 then follows from \cite[Lemma 6.3.7]{olsson-compactifymoduli}.
    (Sketch of proof: it suffices to prove that for every $S$-scheme $T$, every $\alpha \in \SpS(M)(T)$ lifts locally to $\AutS(\mathcal{G})$.
    This follows from the fact that $\alpha^*\mathcal{G}_T$ is a theta group for $(M,e)_T$, so must be \'etale locally isomorphic to $\mathcal{G}_T$.)
    Part 2 follows from a direct computation whose proof is identical to the case where $S$ is the spectrum of a field \cite[Lemma 6.6.6]{BirkenhakeLange-AVs}.
    The exactness of \eqref{eq: first exact sequence AutG} follows from Part 1 and using the isomorphism $M\simeq \AutS(\mathcal{G};\Id)$ of Part 2.
\end{proof}

Consequently, by using the isomorphism $\alpha\colon M\xrightarrow{\sim} \AutS(\mathcal{G};\Id)$ we obtain the exact sequence of finite \'etale group schemes
\begin{align}\label{eq: first exact sequence AutG}
    1\rightarrow M \rightarrow \AutS(\mathcal{G})\rightarrow \SpS(M)\rightarrow 1.
\end{align}
This sequence will play a fundamental role in this paper, because of Lemma \ref{lemma: theta group exists iff class lifts}.

\begin{corollary}\label{corollary: iso classes theta groups vs torsors}
    Let $S$ be a scheme and $D$ a type with $\#D$ invertible on $S$.
    Then the assignment $\mathcal{G} \mapsto \IsomS(\mathcal{G}, \mathcal{G}_D)$ induces a bijection between:
    \begin{itemize}
        \item Isomorphism classes of theta groups for $(M,e)$ over $S$, where $(M,e)$ is a symplectic module of type $D$ over $S$;
        \item $\AutS(\mathcal{G}_D)$-torsors over $S$.
    \end{itemize}
    Suppose that there exists a theta group $\mathcal{G}$ for the symplectic module $(M,e)$ over $S$.
    Then the assignment $\mathcal{G}' \mapsto \IsomS(\mathcal{G}', \mathcal{G};\Id_M)$ induces a bijection between:
    \begin{itemize}
        \item Framed isomorphism classes of theta groups for $(M,e)$ over $S$;
        \item $\AutS(\mathcal{G};\Id)$-torsors over $S$.
    \end{itemize}
\end{corollary}
\begin{proof}
     This follows from Lemmas \ref{lemma: symplectic modules locally constant} and \ref{lemma: theta groups locally isomorphic} and \'etale descent, more precisely Lemma \ref{lemma: nonempty gerbe is neutral}. (Noting that theta groups are always affine over $S$, so every descent datum is effective by \cite[Theorem 4.3.5]{Poonen-rationalpointsonvarieties}.)
\end{proof}

Let $(M,e)$ be a finite symplectic module over $S$ of type $D$.
Its isomorphism class corresponds to an element $c_M\in \HH^1(S, \SpS(M_D))$ using Lemma \ref{lemma: symplectic modules vs SpD torsors}.
\begin{lemma}\label{lemma: theta group exists iff class lifts}
    There exists a theta group for $(M,e)$ over $S$ if and only if the class $c_M$ lifts along the map $\HH^1(S, \AutS(\mathcal{G}_D)) \rightarrow \HH^1(S,\SpS(M_D))$.
\end{lemma}
\begin{proof}
    This follows from Lemma \ref{lemma: symplectic modules vs SpD torsors}, Corollary \ref{corollary: iso classes theta groups vs torsors} and the fact that if $\mathcal{G}$ is a theta group for $(M',e')$, then the pushout of the $\AutS(\mathcal{G}_D)$-torsor $\IsomS(\mathcal{G}', \mathcal{G}_D)$ along $\AutS(\mathcal{G}_D)\rightarrow \Sp(M_D)$ is isomorphic to $\IsomS(M',M_D)$. 
\end{proof}

Let $(M,e)$ be a symplectic module over $S$. 
Consider the fibered category $\ThetaCat_{(M,e)}$ where for an $S$-scheme $T$, the groupoid $\ThetaCat_{(M,e)}(T)$ has objects theta groups for $(M,e)_T$, and morphisms given by framed isomorphisms between theta groups.
Then $\ThetaCat_{(M,e)}$ is a stack in the \'etale topology on $S$.
By Lemma \ref{lemma: theta groups locally isomorphic}, it is in fact a gerbe, and by Part 2 of Lemma \ref{lemma: automorphism groups theta groups}, this gerbe is banded by the group scheme $M$.
By Lemma \ref{lemma:classificationgerbesH2}, this gerbe defines a class $[\ThetaCat_{(M,e)}]\in \HH^2(S,M)$, and this class vanishes if and only if there exists a theta group for $(M,e)$.
It is possible to use the formalism of gerbes to show that $[\ThetaCat_{(M,e)}]$ is the image of $c_M$ under a connecting homomorphism associated to the exact sequence \eqref{eq: first exact sequence AutG} for the pair $(M_D, \mathcal{G}_D)$, but we will not need this.

It will sometimes be useful to refine theta groups to finite \'etale group schemes (this is somewhat implicit in \cite[Chapitre I, Proposition 5.7]{moretbailly-pinceaux}).
Let $(M,e)$ be a symplectic module of type $D = (d_1, \dots, d_g)$.
Given an integer $n\geq 1$ and a group scheme $G\rightarrow S$, let $G[n]$ denote the kernel of the multiplication-by-$n$ morphism $[n]\colon G\rightarrow G$.

\begin{lemma}\label{lemma: equivalence Gm-theta groups and mun-theta groups}
    Let $n = d_g$ if $d_g$ is odd and $n = 2d_g$ if $d_g$ is even.
    If $\mathcal{G}$ is a theta group for $(M,e)$, then $\mathcal{G}[n]$ is a finite \'etale closed subgroup scheme of $\mathcal{G}$ fitting in a central extension $1\rightarrow \mu_n \rightarrow \mathcal{G}[n]\rightarrow M\rightarrow 1$.
    Moreover, the assignment $\mathcal{G}\rightarrow \mathcal{G}[n]$ induces an equivalence between the following groupoids:
    \begin{itemize}
        \item Theta groups for $(M,e)$ over $S$, with morphisms given by isomorphisms of theta groups;
        \item Central extensions of the form $1\rightarrow \mu_{n} \rightarrow \mathcal{H}\rightarrow M\rightarrow 1$ whose commutator pairing equals $e$, with morphisms given by isomorphisms of group schemes $\mathcal{H}\rightarrow \mathcal{H}'$ that restrict to the identity on $\mu_n$.
    \end{itemize}
\end{lemma}
\begin{proof}
    For every $k\geq 1$, $\mathcal{G}[k]$ is a closed subscheme of $\mathcal{G}$.
    The identity $(xy)^k = x^k y^k [x,y]^{k(k+1)/2} = e(x,y)^{k(k+1)/2}x^ky^k$ for $x,y\in \mathcal{G}$ and the fact that $M$ is killed by $d_g$ shows that $\mathcal{G}[n]$ is closed under multiplication and inversion, so is indeed a closed subgroup scheme of $\mathcal{G}$.
    Since $d_g$ is invertible on $S$, $n$ is also invertible on $S$ and so $\mathcal{G}[n]$ is finite \'etale.
    To prove the equivalence of groupoids, we describe an inverse and leave the remaining verifications to the reader. 
    If $\mathcal{H}$ is a central $\mu_n$-extension of $M$ whose commutator pairing equals $e$, define $\mathcal{G}$ as the (sheaf) quotient of $\G_m \times \mathcal{H}$ by $\{(\lambda, \lambda^{-1})\colon \lambda \in \mu_n\}$.
    Then $\mathcal{G}$ is representable by a group scheme which is a theta group for $(M,e)$, and $\mathcal{H}\mapsto \mathcal{G}$ is the desired inverse.
\end{proof}

\subsection{Mumford theta groups}\label{subsec: Mumford theta groups}

Let $A\rightarrow S$ be a $g$-dimensional abelian scheme and let $\lambda\colon A\rightarrow A^{\vee}$ be a polarization whose degree is invertible on $S$.
Then the kernel $A[\lambda]$ is a finite \'etale $S$-group scheme.
Since $\lambda$ is self-dual, $A[\lambda]$ is Cartier dual to itself.
This self-duality is witnessed by a nondegenerate alternating pairing $e_{\lambda} \colon A[\lambda ] \times A[\lambda] \rightarrow \G_m$, called the \emph{Weil pairing}.
Therefore the pair $(A[\lambda],e_{\lambda})$ is a symplectic module over $S$, in the sense of Section \ref{subsec: symplectic modules}.
We say $(A, \lambda)$ is of type $D$ if $D$ has length $g$ and $(A[\lambda],e_{\lambda})$ is of type $D$.

Let $X\rightarrow S$ be an $A$-torsor and let $L$ be a line bundle on $X$ such that $\phi_{L} = \lambda$, see \S\ref{subsec: polarizations on abelian varieties} for the notation.
If $a\in A(S)$, let $t_a\colon X\rightarrow X$ be the translation-by-$a$ morphism.
Following Mumford \cite[p.\, 289]{Mumford-equationsdefiningabelianvarieties}, let $\mathcal{G}(L)\rightarrow S$ be the group scheme such that for every $S$-scheme $T$,
\[
\mathcal{G}(L)(T) = \{(a, \varphi) \mid  a\in A[\lambda](T)
,\, \varphi \text{ is an isomorphism of line bundles } L_T \xrightarrow{\sim} t_a^*L_T \}.
\]
The group operation is given by $(a, \varphi)\cdot (b, \psi) = (a+b,t_b^*\varphi \circ \psi)$, where $t_b^*\varphi\circ \psi$ is the composite $L\xrightarrow{\psi} t_b^* L \xrightarrow{t_b^*\varphi} t_b^*(t_a^*L) = t_{a+b}^* L$.
By \cite[\S6, Proposition 1]{Mumford-equationsdefiningAVs-part2} (who assumes $X=A$ but whose proof continues to hold for general $X$), $\mathcal{G}(L)$ is representable by a group scheme and the forgetful map $(a, \varphi)\mapsto a$ defines a central extension
\[
1\rightarrow \G_m\rightarrow \mathcal{G}(L)\rightarrow A[\lambda]\rightarrow 1.
\]
\begin{proposition}
    $\mathcal{G}(L)$ is a theta group for $(A[\lambda],e_{\lambda})$.
\end{proposition}
\begin{proof}
    It suffices to prove that the commutator pairing of $\mathcal{G}(L)$ is $e_{\lambda}$. 
    To prove this, we may assume $S$ is the spectrum of an algebraically closed field and $X=A$, in which case it is well known, see for example \cite[p. 44-46]{Mumford-tataIII} or \cite[Section 12.2]{Polishchuk-abelianvarietiesthetafunctions}. 
\end{proof}

Let $M$ be another line bundle on $X$ with $\phi_M = \lambda$.
Let $\gamma\colon L\rightarrow M$ be an isomorphism.
Then the assignment $(a, \varphi)\mapsto (a, t_a^* (\gamma) \varphi \gamma^{-1})$ induces an isomorphism $F_{\gamma}\colon\mathcal{G}(L) \rightarrow \mathcal{G}(M)$ of group schemes that restricts to the identity on $\G_m$ and induces the identity on $A[\lambda]$.
In the notation of Section \ref{subsec: abstract theta groups}, $F_{\gamma}$ is a framed isomorphism.
Any other isomorphism $\gamma'\colon L\rightarrow  M$ differs from $\gamma$ by a nonzero scalar, hence $F_{\gamma'} = F_{\gamma}$.
We conclude that if $L$ and $M$ are isomorphic line bundles on $X$ with $\phi_L = \phi_M = \lambda$, then there is a canonical framed isomorphism between $\mathcal{G}(L)$ and $\mathcal{G}(M)$.
Similarly, if $M$ is a line bundle on $S$ and $p\colon X\rightarrow S$ is the structure map, the assignment $(a, \varphi)\mapsto (a, \varphi \otimes p^*\Id_M)$ induces an isomorphism $\mathcal{G}(L) \xrightarrow{\sim} \mathcal{G}(L\otimes p^*M)$.
We use these observations to extend the scope of the construction of $\mathcal{G}(L)$, as follows.

Let $\ell$ be an element of $\PicS_{X/S}(S)$ with $\phi_{\ell} = \lambda$ (see \S\ref{subsec: polarizations on abelian varieties}).
Let $f\colon S'\rightarrow S$ be an \'etale surjective morphism such that $f^*\ell = [L]$ for some line bundle $L$ on $X_{S'}$, using the sequence \eqref{eq: exact sequence usual pic scheme pic}.
Let $p_1, p_2\colon S' \times_S S'\rightarrow S'$ denote the two projections. 
Since $p_1^*f^*\ell = p_2^* f^*\ell$, there exists a line bundle $M$ on $S'\times_S S'$ and an isomorphism $\gamma\colon p_1^*L \xrightarrow{\sim} p_2^*L \otimes p^*M$.
By the previous paragraph, this induces a framed isomorphism $F\colon p_1^*\mathcal{G}(L) \xrightarrow{\sim}  p_2^*\mathcal{G}(L)$ of theta groups that is independent of the choice of $\gamma$ and $M$.
Therefore $F$ defines a descent datum.
Since theta groups are affine over $S$, every such descent datum is effective.
Therefore, there exists a theta group $\mathcal{G}$ for $(A[\lambda],e_{\lambda})$ over $S$, unique up to unique framed isomorphism, whose base change along $f$ corresponds to $\mathcal{G}(L)$ equipped with the descent datum $F$.
We denote such a group by $\mathcal{G}(\ell)$.
Up to unique isomorphism, $\mathcal{G}(\ell)$ is independent of the choice of cover $f$.

In conclusion, we have just shown that if there exists an $A$-torsor $X\rightarrow S$ and an element $\ell\in \PicS_X^{\lambda}(S)$, then there exists a theta group $\mathcal{G}(\ell)$ for $(A[\lambda],e_{\lambda})$. 
In the next section, we show that the converse holds (Theorem \ref{theorem: main iso gerbes}).


\subsection{A theta group criterion for representing a polarization}\label{subsec: theta group criterion}

Let $A\rightarrow S$ be an abelian scheme, let $\lambda\colon A\rightarrow A^{\vee}$ be a polarization whose degree is invertible on $S$ and let $X$ be an $A$-torsor.
Given $a\in A(S)$, recall that we denote the translation-by-$a$ maps $A\rightarrow A$ and $X\rightarrow X$ by $t_a$.
Given $x\in X(S)$, write $t_x\colon A\rightarrow X$ for the unique $A$-equivariant map sending the zero section to $x$.

The scheme $\PicS_A^{\lambda}$ of line bundles representing $\lambda$ is a torsor under $A^{\vee} = \PicS^0_A$ by tensoring with degree-zero line bundles.
Pulling back the $A^{\vee}$-action on $\PicS_A^{\lambda}$ along $\lambda\colon A\rightarrow A^{\vee}$ defines an $A$-action on $\PicS^{\lambda}_A$.
This action is explicitly given by $a\cdot [L]  = \lambda(a)\otimes [L] = \phi_L(a)\otimes [L]= t_a^*[L]$.

Given an element $\ell \in \PicS_X^{\lambda}(S)$, the assignment $x\mapsto t_x^* \ell$ defines a morphism $\psi_{\ell}\colon X\rightarrow \PicS_A^{\lambda}$.
This morphism is equivariant with respect to the $A$-action on the source and target, since $\psi_{\ell}(t_a(x)) = t^*_{t_a(x)}\ell = t_a^*t_x^*\ell=t_a^*\psi_{\ell}(x)$.

\begin{lemma}\label{lemma: PicXlambda versus A-equivariant maps}
In the above notation, the assignment $\ell\mapsto \psi_{\ell}$ induces a bijection
\begin{align}\label{equation: PicXlambda versus A-equivariant maps}
    \PicS_X^{\lambda}(S) \xrightarrow{1\colon1}
    \{A\text{-equivariant morphisms }X\rightarrow \PicS_A^{\lambda}  \}.
\end{align}
If $f\colon Y\rightarrow X$ is a morphism of $A$-torsors and $\ell \in \PicS_X^{\lambda}(S)$, then $\psi_{f^*\ell} = \psi_{\ell}\circ f$.
\end{lemma}
\begin{proof}
Since both sides of \eqref{equation: PicXlambda versus A-equivariant maps} satisfy \'etale descent, to prove bijectivity we may assume (after applying an \'etale base change) that $X\rightarrow S$ has a section.
Let $x$ be such a section.
Then giving an $A$-equivariant morphism $X\rightarrow \PicS_A^{\lambda}$ is the same as giving its value at $x\in X(S)$.
But $\psi_{\ell}(x) = t_x^*\ell$ and $t_x^*\colon \PicS_X^{\lambda}\rightarrow \PicS_A^{\lambda}$ is an isomorphism.
So \eqref{equation: PicXlambda versus A-equivariant maps} is a bijection when $X\rightarrow S$ has a section, hence a bijection in general.
The final sentence is a computation: $\psi_{f^*\ell}(y) = t_y^* f^* \ell = (f\circ t_y)^* \ell = t_{f(y)} ^*\ell = \psi_{\ell}\circ f$.
\end{proof}

\begin{remark}
    We can construct an explicit inverse to \eqref{equation: PicXlambda versus A-equivariant maps}, as follows. 
    Let $\psi\colon X\rightarrow \PicS_A^{\lambda}$ be $A$-equivariant.
    Let $S'\rightarrow S$ be an \'etale surjective morphism such that there exists an element $x\in X(S')$ .
    Let $\ell= (t^{-1}_{x})^*\psi(x)$.
    Then $\ell$ is independent of the choice of $x$, descends to an element of $\PicS_X^{\lambda}(S)$ and satisfies $\psi_{\ell} = \psi$.
\end{remark}

\begin{corollary}\label{corollary:PiclambdaX nonempty iff lifts under H^1(lambda)}
    In the above notation, $\PicS_X^{\lambda}(S)\neq \varnothing$ if and only if $[X]$ maps to $[\PicS^{\lambda}_A]$ under the map $\HH^1(\lambda)\colon \HH^1(S,A)\rightarrow \HH^1(S,A^{\vee})$.
\end{corollary}
\begin{proof}
    By definition of pushout, $[X]$ maps to $[\PicS_A^{\lambda}]$ under $\HH^1(\lambda)$ if and only if there exists an $A$-equivariant morphism $X\rightarrow \PicS^{\lambda}_A$.
    Conclude by Lemma \ref{lemma: PicXlambda versus A-equivariant maps}.
\end{proof}

Let $\PolTor_{(A,\lambda)}$ be the category fibered in groupoids such that for every $S$-scheme $T$, the groupoid $\PolTor_{(A,\lambda)}(T)$ is defined as follows: objects are pairs $(X, \ell)$ where $X\rightarrow T$ is a torsor under $A_T\rightarrow T$ and $\ell \in \PicS^{\lambda}_X(T)$; morphisms $(X, \ell) \rightarrow (X', \ell')$ are isomorphisms of torsors $f\colon X\rightarrow X'$ such that $f^*(\ell') = \ell$.
This is a stack in the \'etale topology over $S$. 
Since every $A$-torsor is \'etale locally trivial and since the translation action of $A$ on $\PicS^{\lambda}_A$ is transitive, $\PolTor_{(A,\lambda)}$ is a gerbe.

On the other hand, consider the quotient stack $[A\backslash \PicS^{\lambda}_A]$. 
By definition, if $T$ is an $S$-scheme then the groupoid $[A\backslash \PicS^{\lambda}_A](T)$ is defined as follows: objects are pairs $(X, \psi)$, where $X\rightarrow T$ is a torsor under $A_T\rightarrow T$ and $\psi\colon X\rightarrow \PicS_A^{\lambda}$ an $A$-equivariant morphism; morphisms $(X, \psi) \rightarrow (X', \psi')$ are isomorphisms of $A$-torsors $f\colon X\rightarrow X'$ such that $\psi'\circ f  = \psi$.
Lemma \ref{lemma: PicXlambda versus A-equivariant maps} immediately implies:

\begin{lemma}\label{lemma: explicit description quotient stack}
The assignment $(X, \ell)\mapsto (X, \psi_{\ell})$ defines an isomorphism of stacks $\PolTor_{(A,\lambda)} \rightarrow [A\backslash \PicS^{\lambda}_A]$. 
\end{lemma}

We will now compare these gerbes to the gerbe $\ThetaCat_{(A[\lambda], e_{\lambda})}$ introduced in \S\ref{subsec: abstract theta groups}.
Recall that for an $S$-scheme $T$, $\ThetaCat_{(A[\lambda], e_{\lambda})}(T)$ is the groupoid of theta groups for $(A[\lambda],e_{\lambda})_T$ with morphisms given by framed isomorphisms.
Note that the assignment $(X, \ell)\mapsto \mathcal{G}(\ell)$ of Section \ref{subsec: Mumford theta groups} can be upgraded to a functor $\Theta\colon \PolTor_{(A,\lambda)}\rightarrow \ThetaCat_{(A[\lambda],e_{\lambda})}$, by sending a morphism $f\colon (X, \ell)\rightarrow (X', \ell')$ to $(f^*)^{-1}\colon \mathcal{G}(\ell) \rightarrow \mathcal{G}(\ell')$, where $f^*\colon \mathcal{G}(\ell') \rightarrow \mathcal{G}(f^*\ell')=  \mathcal{G}(\ell)$ is the unique framed isomorphism of theta groups with the property that, after some \'etale surjective base change $T\rightarrow S$ over which $\ell'_T = [M]$ for line bundle $M$ on $X'$, $(f^*)_T\colon \mathcal{G}(M)\rightarrow \mathcal{G}((f_T)^*M)$ is given by $(x, \varphi)\mapsto (x,(f_T)^*\varphi)$.

The next theorem is one of the main technical results of this paper, so we restate our assumptions.

\begin{theorem}\label{theorem: main iso gerbes}
Let $A\rightarrow S$ be an abelian scheme and let $\lambda\colon A\rightarrow A^{\vee}$ be a polarization whose degree is invertible on $S$.
Then the functor $\Theta\colon \PolTor_{(A,\lambda)}\rightarrow \ThetaCat_{(A[\lambda], e_{\lambda})}$ sending $(X, \ell)$ to $\mathcal{G}(\ell)$ is an isomorphism of gerbes.
Consequently, there are isomorphisms
\begin{align}\label{eq: main theorem iso gerbes}
[A\backslash \PicS^{\lambda}_A] \simeq \PolTor_{(A,\lambda)} \simeq \ThetaCat_{(A[\lambda], e_{\lambda})}
\end{align}
and $(X, \ell)\mapsto \mathcal{G}(\ell)$ induces an equivalence between the following groupoids:
\begin{enumerate}
\item Pairs $(X, \ell)$, where $X$ is an $A$-torsor and $\ell \in \PicS_{X/S}(S)$ is an element with $\phi_{\ell} = \lambda$, with morphisms $(X, \ell)\rightarrow (X', \ell')$ given by isomorphisms of torsors $f\colon X\rightarrow X'$ such that $f^*\ell' = \ell$;
\item Theta groups for $(A[\lambda], e_{\lambda})$ over $S$, with morphisms given by framed isomorphisms.
\end{enumerate}
\end{theorem}
\begin{proof}
Since the source and target of $\Theta$ are gerbes, it suffices to show by Lemma \ref{lemma: fully faithful morphism gerbes} that 
$\Theta_x\colon \AutS(x)\rightarrow \AutS(\Theta(x))$ is an isomorphism for every $S$-scheme $T$ and object $x$ of $\PolTor_{(A,\lambda)}(T)$.
Since this can be checked \'etale locally, we may assume that $T=S$, that $x =(A, [L])$ for some line bundle $L$ on $A$ with $\phi_L = \lambda$, and that $A[\lambda]$ is a constant group scheme.
We now explicitly describe the map $\Theta_{(A,[L])}\colon \AutS((A,[L]))\rightarrow \AutS(\mathcal{G}(L);\Id)$ and show that it is an isomorphism.

First note that $a\mapsto t_a$ induces an isomorphism $A[\lambda] \simeq \AutS((A, [L]))$, so it suffices to prove that the map $a\mapsto \Theta_{(A,[L])}(t_a)$ is an isomorphism $A[\lambda]\rightarrow \AutS(\mathcal{G}(L);\Id)$.
If $a\in A(S)$, then $(x, \varphi)\mapsto (x, t_a^*\varphi)$ defines an isomorphsm $\Phi_a\colon \mathcal{G}(L) \rightarrow \mathcal{G}(t_a^*L)$.
On the other hand, if $a \in A[\lambda](S)$ then there exists an isomorphism $\gamma\colon L\rightarrow t_a^*L$.
We have seen in Section \ref{subsec: Mumford theta groups} that the assignment $(x, \varphi)\mapsto (x, t_x^* (\gamma) \varphi \gamma^{-1})$ induces an isomorphism $F_a\colon\mathcal{G}(L) \rightarrow \mathcal{G}(t_a^*L)$ which does not depend on the choice of $\gamma$.
By definition, $\Theta_{(A,[L])}(t_a)$ equals $\Phi^{-1}_a \circ F_a\in \Aut(\mathcal{G}(L); \Id)$. 
We compute
  \[
  (\Phi^{-1}_a \circ F_a)((x, \varphi)) 
  =(x, t_a^*((t_x^* \gamma) \varphi \gamma^{-1}))
  =(a,\eta)\cdot (x,\varphi)\cdot (a,\eta)^{-1}
  = e_{\lambda}(a,x) (x, \varphi).
  \]
Therefore the map $a\mapsto \Theta_{(A,[L])}(t_a)$ is exactly the isomorphism $\alpha\colon A[\lambda]\rightarrow \AutS(\mathcal{G}(L);\Id)$ of Lemma \ref{lemma: automorphism groups theta groups}(2). 
We conclude that $\Theta$ is an equivalence and $\PolTor_{(A, \lambda)}\simeq \ThetaCat_{(A[\lambda],e_{\lambda})}$.
The first isomorphism of \eqref{eq: main theorem iso gerbes} follows from Lemma \ref{lemma: explicit description quotient stack}.
The equivalence of groupoids follows from taking $S$-points of $\Theta$. 
\end{proof}

\begin{remark}
Given a theta group $\mathcal{G}$ for $(A[\lambda], e_{\lambda})$ it is possible to explicitly construct a pair $(X,\ell)$ in $\PolTor_{(A,\lambda)}(S)$ with $\mathcal{G}(\ell)\simeq \mathcal{G}$; we sketch the details.
For an $S$-scheme $T$, let $X(T)$ be the set of pairs $(M, \alpha)$, where $M \in \PicS_A^{\lambda}(T)$ and $\alpha\in \IsomS(\mathcal{G}(M),\mathcal{G};\Id)(T)$.
This is represented by a scheme $X\rightarrow S$, and the forgetful map $X\rightarrow \PicS_A^{\lambda}$ is a torsor under $\AutS(\mathcal{G};\Id)\simeq A[\lambda]$.
We can extend this $A[\lambda]$-action on $X$ to an $A$-action, via the formula $a\cdot (M, \alpha) = (t_a^*M, \alpha\circ \Phi_a^{-1})$.
(A similar computation to the proof of Theorem \ref{theorem: main iso gerbes} shows that this action indeed restricts to the given $A[\lambda]$-action.)
This defines an $A$-equivariant map $X\rightarrow \PicS^{\lambda}_A$, hence a pair $(X, \ell)$ representing $\lambda$ by Lemma \ref{lemma: PicXlambda versus A-equivariant maps}.
\end{remark}

\begin{corollary}\label{corollary: theorem theta first part body of text}
In the notation of Theorem \ref{theorem: main iso gerbes}, the following statements are equivalent: 
\begin{itemize}
\item There exists an $A$-torsor $X\rightarrow S$ and $\ell \in \PicS_{X/S}(S)$ with $\phi_{\ell} = \lambda$;
\item There exists a theta group for $(A[\lambda], e_{\lambda})$ over $S$.
\item The class $[\PicS^{\lambda}_A] \in \HH^1(S, A^{\vee})$ lies in the image of the map $\HH^1(S, A)\rightarrow \HH^1(S, A^{\vee})$ induced by $\lambda$.
\end{itemize}
\end{corollary}
\begin{proof}
Combine Theorem \ref{theorem: main iso gerbes} and Corollary \ref{corollary:PiclambdaX nonempty iff lifts under H^1(lambda)}.
\end{proof}

\subsection{Linear theta groups}\label{subsection: linear theta groups}

Under the equivalence of groupoids of Theorem \ref{theorem: main iso gerbes} between theta groups and pairs $(X, \ell)$ with $\ell \in \PicS_X(S)$ representing $\lambda$, it is natural to ask: when we can choose $\ell$ to be of the form $[L]$ for some line bundle $L$ on $X$?
In other words, when is $\lambda$ represented by an actual line bundle on $X$, not just an element of $\PicS_X(S)$?
The answer is given by Theorem \ref{theorem:shrodinger gerbe and obstruction gerbe are isomorphic} and uses the concept of linear theta groups.

Let $S$ be a scheme and let $D = (d_1, \dots, d_g)$ be a type with $\#D$ invertible on $S$.
If $\sh{F}$ is a quasi-coherent sheaf on $S$, let $\AutS(\sh{F})$ be the sheaf of groups on the \'etale site of $S$ with $\AutS(\sh{F})(T) = \Aut_{\O_T}(f^*\sh{F})$ for every morphism $f\colon T\rightarrow S$.
For example, if $\sh{F} = \O_S^{\oplus n}$ then $\AutS(\sh{F})  = \GL_{n,S}$. 

\begin{definition}
    Let $\mathcal{G}$ be a theta group for a symplectic module $(M,e)$ of type $D$ over $S$.
    A \emph{representation} for $\mathcal{G}$ is a quasi-coherent sheaf on $\sh{F}$ endowed with a homomorphism $\rho\colon \mathcal{G}\rightarrow \AutS(\sh{F})$. 
    We say that $\rho$ (or by abuse of notation $\sh{F}$) has \emph{weight }$1$ if the subgroup $\G_{m,S}$ of $\mathcal{G}$ acts on $\sh{F}$ via scalar multiplication, i.e., via restriction of the standard $\O_S$-action on $\sh{F}$ to $\G_{m,S}\leq \O_S$.
    We say $\sh{F}$ is a \emph{Schr\"odinger representation} for $\mathcal{G}$ if it is of weight $1$ and if $\sh{F}$ is a locally free of rank $\#D= d_1\cdots d_g$.    
    If such a representation exists, we say that $\mathcal{G}$ is \emph{linear}.
\end{definition}

\begin{example}[The motivating example]\label{example:globalsectionsisShrodinger}
    Let $(A,\lambda)$ be a polarized abelian scheme of type $D$ over $S$ and $L$ a line bundle on an $A$-torsor $X$ with $\phi_L = \lambda$.
    If $S$ is the spectrum of a field $k$, then the space of global sections $\HH^0(X,L)$, endowed with the action $(a, \varphi)\cdot s = t_{-a}^*(\varphi(s))$, is a weight-$1$ representation of $\mathcal{G}(L)$.
    By Riemann--Roch, $\dim \HH^0(X,L) = \#D$, so $\HH^0(X,L)$ is a Schr\"odinger representation for $\mathcal{G}(L)$.
    If $S$ is a general base scheme and $\pi\colon A\rightarrow S$ the structure morphism, then $\pi_*L$ again has the structure of a weight-$1$ representation.
    By cohomology and base change and the result for fields, it is locally free of rank $\#D$, so $\pi_*L$ is a Schr\"odinger representation.
\end{example}

Let $\mathcal{G}_D$ be the standard theta group of type $D$ for $(M_D,e_D)$ constructed in \S\ref{subsec: abstract theta groups}.
Then we can construct a Schr\"odinger representation $\sh{V}_D$ for $\mathcal{G}_D$, following Mumford \cite[p.\, 297]{Mumford-equationsdefiningabelianvarieties}.
Recall that $M_D = K_D \times K_D^{\vee}$, where $K_D$ is a contant group scheme. 
Let $\sh{V}_D$ be the $\O_S$-module with $\sh{V}_D(U) = \{\text{functions } K_D\rightarrow \O_S(U)\}$ for every open $U\subset S$.
Then $\sh{V}_D$ is free of rank $\#K_D = \#D$.
Let $\mathcal{G}_D$ act on $\sh{V}_D$ via the formula
\[
({(\lambda, x,\chi)}\cdot f)(y) = \lambda \chi(y)f(y+x).
\]
A calculation shows that this action is well defined and that $\sh{V}_D$ is a Schr\"odinger representation for $\mathcal{G}_D$.

If $S$ is the spectrum of an algebraically closed field, then every theta group has a unique Schr\"odinger representation; this is an algebraic version of the Stone--von Neumann theorem \cite[page 295, Proposition 3]{Mumford-equationsdefiningabelianvarieties}.
Over an arbitrary base, we have:
\begin{proposition}\label{proposition:categoryofweight1representations}
	Let $\mathcal{G}$ be a theta group for a symplectic module $(M,e)$ over $S$. 
	\begin{enumerate}
	\item There exists an \'etale surjective morphism $T\rightarrow S$ such that there exists a Schr\"odinger representation for $\mathcal{G}_T$.
	\item If $\mathcal{G}$ is linear and $\sh{V}$ is a Schr\"odinger representation for $\mathcal{G}$, then the assignment $\sh{F}\mapsto \sh{V} \otimes \sh{F}$ defines an equivalence between the category of quasi-coherent sheaves on $S$ and the category of weight-$1$ representations of $\mathcal{G}$.
	(The $\mathcal{G}$-action on $\sh{V}\otimes \sh{F}$ is induced by the given $\mathcal{G}$-action on $\sh{V}$ and the trivial one on $\sh{F}$.)
	\item If $\sh{V}, \sh{V}'$ are Schr\"odinger representations for $\mathcal{G}$, then there exists an invertible sheaf $\sh{L}$ on $S$ and an isomorphism of $\mathcal{G}$-representations $\sh{V}' \simeq \sh{V} \otimes \sh{L}$. 
	\end{enumerate}
\end{proposition}
\begin{proof}
	\begin{enumerate}
	\item We may assume that $S$ is connected and (by Lemmas \ref{lemma: symplectic modules locally constant} and \ref{lemma: theta groups locally isomorphic}) that $\mathcal{G} = \mathcal{G}_D$, in which case we may take $\sh{V}_D$ as Schr\"odinger representation.
	\item This is \cite[p.113, Corollaire 2.4.3]{moretbailly-pinceaux}.
	\item Part 2 implies that $\sh{V}' = \sh{V}\otimes \sh{L}$ for some quasi-coherent $\sh{L}$ on $S$.
    To show $\sh{L}$ is invertible, we may take an \'etale base change of $S$, hence assume $\mathcal{G} = \mathcal{G}_D$ for some type $D$.
    Now apply \cite[\S6, Proposition 2]{Mumford-equationsdefiningAVs-part2}.
    \end{enumerate}
\end{proof}
Let $\mathcal{G}$ be a theta group for a symplectic module $(M,e)$ over $S$.
Let $\mathrm{Sch}_{\mathcal{G}}$ be the fibered category over $S$-schemes such that for an $S$-scheme $T$, the groupoid $\mathrm{Sch}_{\mathcal{G}}(T)$ has objects Schr\"odinger representations for $\mathcal{G}_T$ and morphisms given by $\mathcal{G}_T$-equivariant isomorphisms.
Descent for quasi-coherent sheaves \cite[Section 6.1, Theorem 4]{BLR-NeronModels} shows that $\mathrm{Sch}_{\mathcal{G}}$ is a stack in the \'etale topology, and Proposition \ref{proposition:categoryofweight1representations} shows that $\mathrm{Sch}_{\mathcal{G}}$ is a gerbe.
Moreover, that same proposition shows that if $\sh{V}$ is a Schr\"odinger representation for $\mathcal{G}$, then $\AutS_{\mathrm{Sch}_{\mathcal{G}}(S)}(\sh{V}) \simeq \G_{m,S}$.
Hence $\mathrm{Sch}_{\mathcal{G}}$ is a gerbe banded by $\G_m$ and by Lemma \ref{lemma:classificationgerbesH2} defines a class $[\mathrm{Sch}_{\mathcal{G}}]\in \HH^2(S,\G_m)$ which vanishes if and only if $\mathcal{G}$ is linear.

On the other hand, let $\pi\colon A\rightarrow S$ be an abelian scheme, let $X$ be an $A$-torsor and let $\ell \in \PicS_{X}(S)$. 
Let $\mathrm{Ob}(\ell)$ be the category fibered in groupoids such that for an $S$-scheme $T$, the groupoid $\mathrm{Ob}(\ell)(T)$ has objects line bundles $L$ on $X_T$ such that $[L] = \ell_T$, and morphisms $L\rightarrow L'$ are given by isomorphisms of line bundles.
The exact sequence \eqref{eq: exact sequence usual pic scheme pic} shows that $\mathrm{Ob}(\ell)$ is a gerbe, which is again banded by $\G_m$.
Hence $\mathrm{Ob}(\ell)$ defines a class $[\mathrm{Ob}(\ell)] \in \HH^2(S, \G_m)$, which equals the class $\mathrm{ob}(\ell)$ of $\ell$ under the connecting homomorphism $\PicS_{X}(S)\rightarrow \HH^2(S, \G_m)$ of \eqref{eq: exact sequence usual pic scheme pic}.

\begin{theorem}\label{theorem:shrodinger gerbe and obstruction gerbe are isomorphic}
	Let $\pi\colon A\rightarrow S$ be an abelian scheme, $\lambda\colon A\rightarrow A^{\vee}$ a polarization whose degree is invertible on $S$, $X$ an $A$-torsor and $\ell \in \PicS_X(S)$ an element with $\phi_{\ell} = \lambda$ and with theta group $\mathcal{G}(\ell)$.
	Then the assignment $L\mapsto \pi_*L$ induces an isomorphism $\mathrm{Ob}(\ell)\xrightarrow{\sim}\mathrm{Sch}_{\mathcal{G}(\ell)}$.
	Consequently, there exists a line bundle $L$ on $X$ with $[L] =\ell$  if and only if the theta group $\mathcal{G}(\ell)$ is linear. 
\end{theorem}
\begin{proof}
    The morphism is well defined by Example \ref{example:globalsectionsisShrodinger}.
	By Lemma \ref{lemma: fully faithful morphism gerbes}, we just need to verify that if $L$ is an object of $\mathrm{Ob}(\ell)(T)$ for some $T\rightarrow S$, then $\Aut(L)\rightarrow \Aut((\pi_T)_*L)$ is an isomorphism.
	This follows from that fact that every automorphism of $L$ is given by multiplication by an element of $\lambda\in \G_m(T)$, and that the induced automorphism of the locally free sheaf $(\pi_T)_*L$ is again multiplication by $\lambda$.
\end{proof}

\begin{corollary}\label{corollary: theorem theta second part body of text}
	Let $A\rightarrow S$ be an abelian scheme and $\lambda$ a polarization on $A$ whose degree is invertible on $S$.
	Then the following statements are equivalent:
	\begin{itemize}
	\item There exists an $A$-torsor $X$ and line bundle $L$ on $X$ with $\phi_L = \lambda$;
	\item There exists a linear theta group for $(A[\lambda], e_{\lambda})$. 
	\end{itemize}
\end{corollary}
\begin{proof}
	Combine Theorems \ref{theorem: main iso gerbes} and \ref{theorem:shrodinger gerbe and obstruction gerbe are isomorphic}.
\end{proof}

\begin{proof}[Proof of Theorem \ref{theorem: intro theta group condition}]
    Combine Corollaires \ref{corollary: theorem theta first part body of text} and \ref{corollary: theorem theta second part body of text}.
\end{proof}

The following bound on the order of the class $[\mathrm{Sch}_{\mathcal{G}}] \in \HH^2(S, \G_m)$ will be useful in Section \ref{section: symplectic modules admitting a theta group}.
\begin{proposition}\label{proposition: bounds order line bundle obstruction class}
   Let $\mathcal{G}$ be a theta group for a symplectic module $(M,e)$ of type $D = (d_1, \dots, d_g)$ over $S$.
    Then $(\#D)[\mathrm{Sch}_{\mathcal{G}}]=0$ in $\HH^2(S, \G_m)$, where we recall that $\#D = d_1\cdots d_g$.
\end{proposition}
\begin{proof}
This is a result of Polishchuk \cite[Proposition 2.1]{Polishchuk-analogueWeilrepresentation}.
\end{proof}

We record a case where the existence problem of theta groups is trivial:

\begin{proposition}\label{proposition: theta groups exist over fields of cohomological dimension 1}
    Let $k$ be a field of cohomological dimension $\leq 1$ in the sense of \cite[Definition 1.4.3]{Poonen-rationalpointsonvarieties}. (For example, $k$ is finite or the function field of a curve over an algebraically closed field.)
    Then for every symplectic module $(M,e)$ over $k$, there exists a linear theta group for $(M,e)$.
\end{proposition}
\begin{proof}
    Below Lemma \ref{lemma: theta group exists iff class lifts}, we have defined a class $[\ThetaCat_{(M,e)}]\in \HH^2(k, M)$ which vanishes if and only if there exists a theta group for $(M,e)$. 
    By definition of cohomological dimension, $\HH^2(k, M) = 0$ and so there exists a theta group $\mathcal{G}$ for $(M,e)$.
    Moreover, since $\HH^2(k , \G_m)=0$, the class $[\mathrm{Sch}_{\mathcal{G}}]$ vanishes and so $\mathcal{G}$ is linear.
\end{proof}

\subsection{Symmetric line bundles representing $\lambda$}\label{subsec: symmetric line bundles representing lambda}

In this section, we prove an analogue of Theorem \ref{theorem: main iso gerbes} for symmetric line bundles on symmetrized torsors (Theorem \ref{theorem: main iso symmetric version}).
We start by recalling the well-known situation for symmetric line bundles on $A$.

Let $A\rightarrow S$ be an abelian scheme. Denote the inversion map by $[-1]\colon A\rightarrow A$.
Recall that a line bundle $L$ on $A$ is \emph{symmetric} if $L\simeq [-1]^*L$.
Since $[-1]^*$ acts as inversion on $\PicS_A^0$ and as the identity on $\NSS_A$, taking $[-1]^*$-fixed points of the exact sequence \eqref{equation: exact sequence Pic_X of torsor} for $X=A$ results in the exact sequence
\begin{align}\label{equation:symmetricPicNS_A}
1\rightarrow A^{\vee}[2]\rightarrow \PicS_A^{\sym} \rightarrow \NSS_A\rightarrow 1,
\end{align}
see \cite[Section 3.2]{PoonenRains-thetacharacteristics}.
A line bundle $L$ on $A$ is symmetric if and only if its class $[L] \in \PicS_A(S)$ lies in $\PicS_A^{\sym}(S)$. 
Write $\PicS_A^{\sym, \lambda}$ for the fiber of an element $\lambda \in \NSS_A(S)$ under the projection map.
This sequence shows that $\PicS_A^{\sym,\lambda}$ is a torsor under $A^{\vee}[2]$, which has an $S$-point if and only if there exists a symmetric line bundle $L$ on $A$ representing $\lambda$.
The class of $\PicS_A^{\sym, \lambda}$ maps to the class of $\PicS_A^{\lambda}$ under the map $\HH^1(S, A^{\vee}[2])\rightarrow \HH^1(S, A^{\vee})$.

In contrast to $\PicS_A^{\lambda}$, the class of $\PicS_A^{\sym, \lambda}$ is known to have a very concrete alternative description.
Consider the nondegenerate pairing $e_2\colon A[2] \times A^{\vee}[2] \rightarrow \mu_2$ induced by the identification $A^{\vee}[2] = A[2]^{\vee}$, and let $e_2^{\lambda} \colon A[2]\times A[2]\rightarrow \mu_2$ be the alternating pairing defined by $e_2^{\lambda}(x,y) = e_2(x, \lambda(y))$. 
We say a map of $S$-schemes $q\colon A[2]\rightarrow \mu_2$ is a \emph{quadratic refinement} of $e^{\lambda}_{2}$ if $q(x+y)q(x)q(y) = e^{\lambda}_{2}(x,y)$ for all $x,y\in A[2]$.
The scheme of quadratic refinments for $e_2^{\lambda}$ is a torsor under $A^{\vee}[2]$.
It is well known (see \cite[Section 13.1]{Polishchuk-abelianvarietiesthetafunctions} and \cite[Proposition 3.6]{PoonenRains-thetacharacteristics}) that this torsor is isomorphic to $\PicS_A^{\sym, \lambda}$. 
Consequently, there exists a symmetric line bundle on $A$ representing $\lambda$ if and only if there exists a quadratic refinement $q\colon A[2]\rightarrow \mu_2$ of $e_{2}^{\lambda}$.
This condition is often satisfied in practice (for example, when $k$ is finite); see \cite[Proposition 3.12]{PoonenRains-thetacharacteristics} for a list of sufficient conditions. 


We now consider the weaker question whether there exists a symmetric line bundle on a symmetric $A$-torsor representing $\lambda$. 
\begin{definition}
Let $X$ be an $A$-torsor.
An \emph{inversion on }$X$ is a morphism of $S$-schemes $\tau\colon X\rightarrow X$ such that $\tau^2=  \Id_X$ and such that $\tau(a+x) = (-a) + x$ for all $a\in A(T)$, $x\in X(T)$ and $S$-schemes $T$.
If such a $\tau$ exists, we say $X$ is \emph{symmetric}.
A \emph{symmetrized $A$-torsor} is a pair $(X, \tau)$, where $X$ is an $A$-torsor and $\tau$ is an inversion.
A morphism between symmetrized $A$-torsors $(X, \tau), (X', \tau')$ is an isomorphism of $A$-torsors $X\rightarrow X'$ intertwining $\tau$ and $\tau'$.
\end{definition}

If $(X,\tau)$ is a symmetrized torsor, then $X^{\tau}$ is a torsor under $A[2]$, and the assignment $(X, \tau)\mapsto X^{\tau}$ induces an equivalence of categories between the category of symmetrized torsors for $A$ and the category of $A[2]$-torsors.
An $A$-torsor $X$ is symmetric if and only if its class $[X]\in \HH^1(S, A)$ lies in the image of the map $\HH^1(S, A[2])\rightarrow \HH^1(S, A)$, if and only if $2[X]=0$.

Given a symmetrized torsor $(X, \tau)$ for $A$, let $\PicS_X^{\tau}\subset \PicS_X$ denote the fixed points of $\tau^*\colon \PicS_X\rightarrow \PicS_X$.
We say a line bundle $L$ on $X$ is \emph{$\tau$-symmetric} if it defines an element of $\PicS_X^{\tau}(S)$.
If $(X,\tau) = (A, [-1])$, then $\PicS_X^{\tau} = \PicS_A^{\mathrm{sym}}$ and $\tau$-symmetric line bundles are the same as symmetric line bundles.
Similarly to \eqref{equation:symmetricPicNS_A}, taking $\tau^*$-fixed points of the exact sequence \eqref{equation: exact sequence Pic_X of torsor} results in the exact sequence
\begin{align*}
    1\rightarrow A^{\vee}[2]\rightarrow \PicS_X^{\tau}\rightarrow \NSS_A\rightarrow 1.
\end{align*}
Let $\lambda\colon A\rightarrow A^{\vee}$ be a polarization whose degree is invertible on $S$.
Write $\PicS^{\tau, \lambda}_X$ for the fiber of $\PicS^{\tau}_X\rightarrow \NSS_X\simeq \NSS_A$ above $\lambda \in \NSS_A(S)$.
We will prove a criterion for the existence of an element $\ell\in \PicS^{\tau,\lambda}_X(S)$ for some symmetrized $A$-torsor $(X, \tau)$ in terms of theta groups.

Let $\PolTor_{(A,\lambda)}^{\sym}$ be the stack such that for each $S$-scheme $T$, the groupoid $\PolTor_{(A,\lambda)}^{\sym}(T)$ has: objects given by triples $(X, \tau, \ell)$, where $(X, \tau)$ is a symmetrized torsor for $A_T\rightarrow T$ and $\ell \in \PicS_X^{\tau, \lambda}(T)$; morphisms $(X, \tau, \ell) \rightarrow (X', \tau',\ell')$ given by isomorphisms of symmetrized torsors $f\colon (X,\tau)\rightarrow (X', \tau')$ such that $f^*(\ell') = \ell$.
This is a stack in the \'etale topology over $S$. 
On the other hand, consider the quotient stack $[A[2]\backslash \PicS^{\sym,\lambda}_A]$, where $A[2]$ acts on $\PicS^{\sym,\lambda}$ by translation of line bundles.
Similarly to Lemma \ref{lemma: explicit description quotient stack} and using the notation of that lemma, we have:
\begin{lemma}\label{lemma: explicit description quotient stack symmetric case}
The assignment $(X, \tau, \ell)\mapsto (X^{\tau}, \psi_{\ell}|_{X^{\tau}})$ defines an isomorphism of stacks  $\PolTor_{(A,\lambda)}^{\sym} \rightarrow [A[2]\backslash \PicS^{\sym, \lambda}_A]$. 
\end{lemma}
\begin{proof}
    This follows from taking $\tau$-fixed points on both sides of Lemma \ref{lemma: PicXlambda versus A-equivariant maps} and using the identity $\psi_{\tau^*\ell} = [-1]^* \circ\psi_{\ell} \circ \tau$ for all $\ell \in \PicS_X^{\lambda}(S)$.
\end{proof}

The next definition has been considered before, see for example \cite[Remark 5.22]{MorganSmith-CasselsTateFiniteGalois} and \cite[Section 1]{Polishchuk-analogueWeilrepresentation}.

\begin{definition}\label{definition: inversion on theta group}
Let $(M,e)$ be a symplectic module over $S$ and $\mathcal{G}$ a theta group for $(M,e)$.
An \emph{inversion} on $\mathcal{G}$ is an element $\iota\in \Aut(\mathcal{G})$ such that $\iota$ maps to $-\Id_M$ under the map $\Aut(\mathcal{G}) \rightarrow \Sp(M)$ of \eqref{eq: first exact sequence AutG}.
If an inversion on $\mathcal{G}$ exists, we say that $\mathcal{G}$ is \emph{symmetric}.
A \emph{symmetrized theta group} for $(M,e)$ is a pair $(\mathcal{G}, \iota)$ where $\mathcal{G}$ is a theta group for $(M,e)$ and $\iota$ an inversion on $\mathcal{G}$.
An isomorphism (resp. framed isomorphism) between symmetrized theta groups $(\mathcal{G}, \iota)\rightarrow (\mathcal{G}', \iota')$ is an isomorphism (resp. framed isomorphism) $f\colon \mathcal{G}\rightarrow \mathcal{G}'$ of theta groups such that $f\circ \iota = \iota' \circ f$.
\end{definition}

If $(M,e)$ is a symplectic module over $S$, let $\ThetaCat^{\sym}_{(M,e)}$ be the fibered category such that for any $S$-scheme $T$, the groupoid $\ThetaCat^{\sym}_{(M,e)}(T)$ has objects symmetrized theta groups for $(M,e)_T$, and morphisms given by framed isomorphisms.

Given a triple $(X, \tau, \ell)$ in $\PolTor_{(A,\lambda)}^{\sym}(S)$, we can upgrade the theta group $\mathcal{G}(\ell)$ (defined in \S\ref{subsec: Mumford theta groups}) to a symmetrized theta group, as follows.
First suppose $\ell$ is represented by a line bundle $L$ on $X$ that satisfies $\tau^*L \simeq L$. (This can always be achieved after an \'etale surjective base change.)
Then the assignment $(a, \varphi)\mapsto (-a, \tau^*\varphi)$ induces an isomorphism $\mathcal{G}(L)\rightarrow \mathcal{G}(\tau^*L)$.
A choice of isomorphism $\tau^*L\rightarrow L$ induces a framed isomorphism $\mathcal{G}(\tau^*L)\rightarrow \mathcal{G}(L)$ independent of this choice (see \S\ref{subsec: Mumford theta groups}).
Their composition is an inversion $\iota_L\colon \mathcal{G}(L)\rightarrow \mathcal{G}(L)$ on $\mathcal{G}(L)$.
In general, there exists a unique inversion $\iota_{\ell}\colon \mathcal{G}(\ell)\rightarrow \mathcal{G}(\ell)$ such that if $T\rightarrow S$ is an \'etale surjective map with $\ell_T = [L]$ satisfying $\tau^*L\simeq L$, then $(\iota_{\ell})_T = \iota_L$.
We define the functor $\Theta^{\sym}\colon \PolTor^{\sym}_{(A, \lambda)}\rightarrow \ThetaCat_{(A[\lambda],e_{\lambda})}$ by sending $(X, \tau, \ell)$ to $(\mathcal{G}(\ell), \iota_{\ell})$, and with morphisms defined similarly to the functor $\Theta\colon \PolTor_{(A,\lambda)}\rightarrow \ThetaCat_{(A[\lambda],e_{\lambda})}$ of Theorem \ref{theorem: main iso gerbes}.

\begin{theorem}\label{theorem: main iso symmetric version}
Let $A\rightarrow S$ be an abelian scheme and let $\lambda\colon A\rightarrow A^{\vee}$ be a polarization whose degree is invertible on $S$.
Then the functor $\Theta^{\sym}\colon \PolTor_{(A,\lambda)}^{\sym}\rightarrow \SymThetaCat_{(A[\lambda], e_{\lambda})}$ sending $(X, \tau, \ell)$ to $(\mathcal{G}(\ell), \iota_{\ell})$ is an isomorphism of stacks.
Consequently, there are isomorphisms
\begin{align}\label{eq: main theorem iso gerbes symmetric version}
[A[2]\backslash \PicS^{\sym,\lambda}_A] \simeq \PolTor_{(A,\lambda)}^{\sym} \simeq \SymThetaCat_{(A[\lambda], e_{\lambda})}
\end{align}
and $(X, \tau, \ell)\mapsto (\mathcal{G}(\ell), \iota_{\ell})$ induces an equivalence between the following groupoids:
\begin{enumerate}
\item Triples $(X,\tau, \ell)$, where $(X,\tau)$ is a symmetrized $A$-torsor and $\ell \in \PicS_{X}^{\tau}(S)$ is an element with $\phi_{\ell} = \lambda$, with isomorphisms $(X, \tau, \ell)\rightarrow (X', \tau',\ell')$ given by isomorphisms of symmetrized torsors $f\colon (X, \tau)\rightarrow (X', \tau')$ such that $f^*\ell' = \ell$;
\item Symmetrized theta groups for $(A[\lambda], e_{\lambda})$ over $S$, with isomorphisms given by framed isomorphisms.
\end{enumerate}
\end{theorem}
\begin{proof}
This follows from taking ``fixed points'' of the isomorphism of Theorem \ref{theorem: main iso gerbes} under a certain duality operation.
More precisely, let $T$ be an $S$-scheme and $(X, \ell)$ an object of $\PolTor_{(A, \lambda)}(T)$.
Let $X'$ be the $A$-torsor whose underlying $S$-scheme equals $X$, but where the $A$-action $+'\colon A\times X'\rightarrow X'$ is given by $a+' x = (-a) + x$, where on the right hand side we use the $A$-action on $X$.
Let $\ell'$ equal $\ell \in \PicS_X^{\lambda}(S)$, seen as an element of $\PicS_{X'}^{\lambda}(S)$.
Define the functor $\Phi\colon \PolTor_{(A, \lambda)}\rightarrow \PolTor_{(A, \lambda)}$ by sending $(X, \ell)$ to $(X', \ell')$ and leaving the morphisms unchanged.
Unwinding the definition of an inversion on an $A$-torsor, we find that the stack $\PolTor^{\sym}_{(A, \lambda)}$ is equivalent to the stack $\PolTor_{(A, \lambda)}^{\Phi}$ of pairs $((X,\ell), \tau)$, where $(X, \ell)$ is an object of $\PolTor_{(A, \lambda)}$ and $\tau$ is an isomorphism $(X, \ell) \xrightarrow{\sim} \Phi((X, \ell))$. 
On the other hand, given a theta group $\mathcal{G}$ for $A[\lambda]$, let $\mathcal{G}'$ be the theta group whose underlying group scheme equals $\mathcal{G}$ but whose projection map $\mathcal{G}'\rightarrow A[\lambda]$ equals the composite $\mathcal{G}\rightarrow A[\lambda]\xrightarrow{-\Id} A[\lambda]$.
Let $\Psi\colon \ThetaCat_{(A[\lambda],e_{\lambda})}\rightarrow \ThetaCat_{(A[\lambda],e_{\lambda})}$ be the functor sending $\mathcal{G}$ to $\mathcal{G}'$ and leaving the morphisms unchanged.
Then $\ThetaCat^{\sym}_{(A[\lambda],e_{\lambda})}$ is equivalent to the stack $\ThetaCat^{\Psi}_{(A[\lambda],e_{\lambda})}$ of pairs $(\mathcal{G}, \iota)$, where $\mathcal{G}$ is an object of $\ThetaCat$ and $\iota$ an isomorphism $\mathcal{G}\xrightarrow{\sim} \Psi(\mathcal{G})$.
Since the functors $\Theta\circ \Phi$ and $\Psi\circ \Theta$ are isomorphic, the functor $\Theta$ of Theorem \ref{theorem: main iso gerbes} induces an isomorphism of stacks $\PolTor_{(A,\lambda)}^{\Phi}\xrightarrow{\sim}\ThetaCat^{\Psi}_{(A[\lambda],e_{\lambda})}$.
After identifying $\PolTor_{(A,\lambda)}^{\Phi}$ with $\PolTor_{(A,\lambda)}^{\sym}$ and $\ThetaCat^{\Psi}_{(A[\lambda],e_{\lambda})}$ with $\ThetaCat_{(A[\lambda],e_{\lambda})}^{\sym}$, this isomorphism equals $\Theta^{\sym}$, proving that $\Theta^{\sym}$ is an isomorphism.
The first isomorphism of \eqref{eq: main theorem iso gerbes symmetric version} follows from Lemma \ref{lemma: explicit description quotient stack symmetric case}.
The equivalence of groupoids follows from taking $S$-points of $\Theta^{\sym}$.
\end{proof}

\begin{corollary}\label{corollary: equivalent conditions existence symmetric line bundle class on torsor}
    Let $A/S$ be an abelian scheme and $\lambda$ a polarization whose degree is invertible on $S$. 
    The following statements are equivalent:
    \begin{enumerate}
        \item There exists a symmetrized $A$-torsor $(X, \tau)$ and a $\tau$-symmetric $\ell\in \PicS_X(S)$ with $\phi_{\ell} = \lambda$.
        \item There exists a symmetric theta group for $(A[\lambda], e_{\lambda})$. 
         \item The class $[\PicS_A^{\sym,\lambda}]\in \HH^1(S,A^{\vee}[2])$ lies in the image of $\HH^1(\lambda)\colon \HH^1(S, A[2])\rightarrow \HH^1(S, A^{\vee}[2])$.
    \end{enumerate}
    The following statements are also equivalent:
    \begin{enumerate}
        \item There exists a symmetrized $A$-torsor $(X, \tau)$ and a $\tau$-symmetric line bundle $L$ on $X$ with $\phi_L = \lambda$.
        \item There exists a linear symmetric theta group for $(A[\lambda],e_{\lambda})$.
    \end{enumerate}
\end{corollary}
\begin{proof}
    The equivalence between the first three statements follows from the isomorphisms of Theorem \ref{theorem: main iso symmetric version}.
    Together with Theorem \ref{theorem:shrodinger gerbe and obstruction gerbe are isomorphic}, it implies the equivalence between the last two statements.
\end{proof}

The following lemma will be useful in Section \ref{subsection: symplectic modules odd order}.
It uses concepts and notation of Section \ref{subsection: linear theta groups}.

\begin{lemma}\label{lemma: symmetric theta group has 8-torsion obstruction class}
    Let $\mathcal{G}$ be a symmetric theta group for a symplectic module $(M,e)$ over $S$.
    Then $8[\mathrm{Sch}_{\mathcal{G}}] =0$ in $\HH^2(S, \G_m)$. 
\end{lemma}
\begin{proof}
    This follows from results of Polishchuk, specifically by combining \cite[Theorem 1.4 and Proposition 2.2]{Polishchuk-analogueWeilrepresentation}.
\end{proof}

\section{Proofs of Theorems \ref{theorem: intro odd degree case} and \ref{theorem: intro positive result}}\label{section: symplectic modules admitting a theta group}

We use the formalism of the previous section to show that the answer to Question \ref{question: main question poonen-stoll} is yes for many pairs $(A, \lambda)$. 
Theorem \ref{theorem: intro odd degree case} follows from Corollary \ref{corollary: strong form odd order polarization represented by symmetric bundle} and Theorem \ref{theorem: intro positive result} follows from Corollary \ref{corollary: symplectic module with M(2) small has theta group}.
In Section \ref{subsection: symplectic F2 vector spaces as Jacobians}, which might be of independent interest, we answer a question of Chidambaram affirmatively \cite{Chidambaram-modpgaloisrepsnotarisingfromAVs}.
Section \ref{subsection: survey of known cases} surveys all variants of Question \ref{question: main question poonen-stoll} and logical implications between these.

\subsection{Symplectic modules of odd order}\label{subsection: symplectic modules odd order}

Let $S$ be a scheme and $D = (d_1, \dots, d_g)$ a type such that $d_g$ is invertible on $S$.
In Sections \ref{subsec: symplectic modules} and \ref{subsec: abstract theta groups} we have constructed group schemes $M_D, \mathcal{G}_D$ and a sequence of finite \'etale group schemes
\begin{align}\label{equation: fundamental exact sequence type D}
1\rightarrow M_D \rightarrow \AutS(\mathcal{G}_D)\rightarrow \SpS(M_D) \rightarrow 1
\end{align}
which is exact by Lemma \ref{lemma: automorphism groups theta groups}.
We now introduce an automorphism of this sequence, which will show that it splits when $\#D$ is odd.
In the notation of \eqref{equation:standardthetagroupformula}, let $\iota\in \Aut(\mathcal{G}_D)$ be the element defined by the formula $
\iota(\lambda,x,\chi) = (\lambda,-x,-\chi)$.
Then $\iota$ lifts $-\Id_{M_D}$, in other words is an inversion on $\mathcal{G}_D$ in the notation of Section \ref{subsec: symmetric line bundles representing lambda}.
Conjugation by $\iota$ defines an automorphism $\Phi\colon \AutS(\mathcal{G}_D)\rightarrow \AutS(\mathcal{G}_D)$ that restricts to $-\Id$ on $M_D$ and induces the identity on $\SpS(M_D)$.
In other words, we obtain a commutative diagram:
\begin{equation}  \label{equation: involution of fundamental exact sequence}
\begin{tikzcd}
	1 & {M_D} & {\AutS(\mathcal{G}_D)} & {\SpS(M_D)} & 1 \\
	1 & {M_D} & {\AutS(\mathcal{G}_D)} & {\SpS(M_D)} & 1
	\arrow[from=1-1, to=1-2]
	\arrow[from=1-2, to=1-3]
	\arrow["{-\Id}", from=1-2, to=2-2]
	\arrow[from=1-3, to=1-4]
	\arrow["\Phi", from=1-3, to=2-3]
	\arrow[from=1-4, to=1-5]
	\arrow["\Id", from=1-4, to=2-4]
	\arrow[from=2-1, to=2-2]
	\arrow[from=2-2, to=2-3]
	\arrow[from=2-3, to=2-4]
	\arrow[from=2-4, to=2-5]
\end{tikzcd}
\end{equation}
Let $\AutS(\mathcal{G}_D, \iota)$ be the fixed point subgroup scheme of $\Phi\colon \AutS(\mathcal{G}_D)\rightarrow \AutS(\mathcal{G}_D)$.
\begin{lemma}\label{lemma: splitting fundamental sequence odd order}
    Suppose that $\#D = d_1\cdots d_g$ is odd.
    Then the restriction $\AutS(\mathcal{G}_D, \iota)\rightarrow \SpS(M_D)$ is an isomorphism.
    Consequently, the sequence \eqref{equation: fundamental exact sequence type D} splits.
\end{lemma}
\begin{proof}
    The claim can be checked after an \'etale base change, so we may assume all the group schemes in \eqref{equation: fundamental exact sequence type D} are split and identify them with their $S$-points.
    Given an element $m\in M_D$, let $\alpha_m \in \Aut(\mathcal{G}_D)$ be the framed automorphism constructed in Lemma \ref{lemma: automorphism groups theta groups}.
    Let $g\in \Sp(M_D)$ be an element and $\tilde{g}\in \Aut(\mathcal{G}_D)$ an arbitrary lift.
    Then $\Phi(\tilde{g}) = \alpha_m \tilde{g}$ for some $m\in M_D$.
    Since $\Phi(\alpha_n \tilde{g}) = \alpha_{m-2n}  (\alpha_n\tilde{g})$ for all $n\in M_D$ and since $\#M_D$ is odd, there exists a unique $n\in M_D$ such that $\alpha_n \tilde{g}\in \Aut(\mathcal{G}_D, \iota)$, namely $n=\frac{1}{2}m$.
    This shows $\Aut(\mathcal{G}_D,\iota)\rightarrow \Sp(M_D)$ is a bijection, hence an isomorphism.
\end{proof}

\begin{theorem}\label{theorem: symplectic module odd order admits theta group}
    Let $(M,e)$ be a symplectic module of type $D$ over a scheme $S$.
    Assume that $\#D$ is odd.
    Then there exists a symmetric theta group for $(M,e)$.
    If $(\mathcal{G}, \iota)$ and $(\mathcal{G}',\iota')$ are symmetrized theta groups for $(M,e)$, then there exists a unique framed isomorphism between $(\mathcal{G}, \iota)$ and $(\mathcal{G}', \iota')$. 
    Every symmetric theta group for $(M,e)$ is linear.
\end{theorem}
\begin{proof}
    By Lemma \ref{lemma: splitting fundamental sequence odd order}, the class of $(M,e)$ in $\HH^1(S, \SpS(M_D))$ (using Lemma \ref{lemma: symplectic modules vs SpD torsors}) lifts to a class in $\HH^1(S,\AutS(\mathcal{G}_D, \iota))$, hence to a class $c\in \HH^1(S, \AutS(\mathcal{G}_D))$.
    By Lemma \ref{lemma: theta group exists iff class lifts}, this implies that there exists a theta group for $(M,e)$.
    The fact that $c$ comes from a class in $\HH^1(S, \AutS(\mathcal{G}_D, \iota))$ means that there exists a symmetric theta group for $(M,e)$.
    This proves existence.
    To prove uniqueness, let $(\mathcal{G}, \iota), (\mathcal{G}', \iota')$ be symmetrized theta groups for $(M,e)$.
    To prove that there exists a unique framed isomorphism between them, we may apply an \'etale base change and hence by Lemma \ref{lemma: theta groups locally isomorphic} we may assume that $\mathcal{G} = \mathcal{G}'$ and that $M$ is constant.
    Then $\iota' = \alpha_m \circ \iota$ for some $m\in M(S)$, where $\alpha\colon M\rightarrow \AutS(\mathcal{G};\Id)$ is the isomorphism of Lemma \ref{lemma: automorphism groups theta groups}.
    The formula $\iota' \circ \alpha_n = \alpha_n \circ (\alpha_{m-2n} \iota)$ shows that there exists a unique $n\in M(S)$ such that $\alpha_n$ is an isomorphism $(\mathcal{G}, \iota)\xrightarrow{\sim} (\mathcal{G}', \iota')$, namely $n = \frac{1}{2}m$.

    It suffices to prove that every symmetric theta group $\mathcal{G}$ for $(M,e)$ is linear.
    Equivalently, by the discussion in \S\ref{subsection: linear theta groups}, we need to show the class $c = [\mathrm{Sch}_{\mathcal{G}}]\in \HH^2(S,\G_m)$ vanishes.
    By Proposition \ref{proposition: bounds order line bundle obstruction class}, $c$ has odd order.
    On the other hand, Lemma \ref{lemma: symmetric theta group has 8-torsion obstruction class} shows that $8c=0$.
    We conclude that $c=0$ and hence that $\mathcal{G}$ is linear.
\end{proof}

\begin{remark}
It is possible to explicitly construct a symmetric theta group for $(M,e)$: let $b\colon M\times M\rightarrow \G_m$ be an alternating pairing with $b^2 = e$ (such a $b$ exists since $\#M$ is odd) and let $\mathcal{G}$ be the group $\G_m\times M$ endowed with multiplication $(\lambda,m)\cdot (\lambda',n) = (\lambda \lambda' b(m,n), m+n)$.
A computation shows that $\mathcal{G}$ is a theta group for $(M,e)$ and $\iota(\lambda,m) = (\lambda,-m)$ is an inversion on $\mathcal{G}$, so $\mathcal{G}$ is symmetric. 
Theorem \ref{theorem: symplectic module odd order admits theta group} shows that $\mathcal{G}$ is linear, but we do not know how to construct a Schr\"odinger representation for $\mathcal{G}$ explicitly.
\end{remark}

Theorem \ref{theorem: intro odd degree case} follows from the next stronger result.

\begin{corollary}\label{corollary: strong form odd order polarization represented by symmetric bundle}
    Let $A\rightarrow S$ be an abelian scheme and let $\lambda\colon A\rightarrow A^{\vee}$ be a polarization of type $D$ whose degree is odd and invertible on $S$.
    Then there exists a symmetrized $A$-torsor $(X, \tau)$ and a $\tau$-symmetric line bundle $L$ on $X$ with $\phi_L = \lambda$.
    If $(X',\tau',L')$ is another such triple, there exists a unique isomorphism of symmetrized torsors $f\colon (X, \tau)\rightarrow (X',\tau')$ that satisfies $f^*[L'] = [L]$ in $\PicS_X(S)$.
\end{corollary}
\begin{proof}
    Combine Theorems \ref{theorem: symplectic module odd order admits theta group} and \ref{theorem:shrodinger gerbe and obstruction gerbe are isomorphic} with the equivalence of groupoids of Theorem \ref{theorem: main iso symmetric version}.
\end{proof}

\subsection{Realizing symplectic modules as $2$-torsion of Jacobians}\label{subsection: symplectic F2 vector spaces as Jacobians}

In this section and most of the remainder of the paper, we assume the base $S$ is the spectrum of a field $k$.
We now show that, under certain assumptions on the characteristic, every symplectic module of type $(2), (2,2)$ or $(2,2,2)$ over $k$ admits a theta group.
The crucial observation is to realize every such symplectic module as the $2$-torsion in a principally polarized abelian variety. 
More precisely, we prove:

\begin{proposition}\label{proposition: symplectic modules of F2 dimension at most 6 are J[2]}
    Let $k$ be an infinite field of characteristic not $2$ and let $g\in \{1,2,3\}$.
    If $g=3$, additionally assume that $k$ has characteristic zero.
    Let $D$ be the type $(2, \cdots, 2)$ of length $g$.
    Then for every symplectic module $(M,e)$ of type $D$ over $k$, there exists a genus-$g$ curve $C$ with Jacobian $J$ and an isomorphism $(M,e)\simeq (J[2],e_2)$, where $e_2$ is the Weil pairing on $J[2]$.
\end{proposition}

Since the proof is much easier for $g=1,2$, we start by treating those cases.

\begin{proof}[Proof of Theorem \ref{proposition: symplectic modules of F2 dimension at most 6 are J[2]} when $g=1,2$:]
    First suppose $g=1$.
    Let $(M,e)$ be a symplectic module of type $(2)$ over $k$.
    The $\Gal_k$-action on the three nonzero elements of $M(k^{\sep})\simeq (\Z/2\Z)^2$ determines (after labeling these elements) a homomorphism $\Gal_k\rightarrow S_3$ to the symmetric group.
    This homomorphism corresponds to an \'etale cubic $k$-algebra $K$.
    There exists a monic separable polynomial $f(x)\in k[x]$ of degree $3$ such that $K\simeq k[x]/(f(x))$. 
    Let $E/k$ be the elliptic curve with Weierstrass equation $y^2 = f(x)$.
    Since the nonzero elements of $E[2](k^{\sep})$ are in bijection with the roots of $f$, there is an isomorphism of finite $k$-groups $M\simeq E[2]$. 
    Since there is a unique nondegenerate alternating pairing on an $\F_2$-vector space of dimension $2$, every such isomorphism must intertwine $e$ with $e_2$.

    Now consider the case $g=2$, where roughly the same principles apply.
    Let $b_1, \dots, b_6$ be the standard basis of $\F_2^6$, let $\Sigma\colon \F_2^6\rightarrow \F_2$ be the map which sums all the coordinates and let $\Delta\colon \F_2\rightarrow \F_2^6$ be the diagonal embedding.
    Let $N = \ker(\Sigma)/\image(\Delta)$. 
    The coordinate-permuting action of $S_6$ on $\F_2^6$ induces an action of $S_6$ on $N$, hence defines a homomorphism $\rho\colon S_6\rightarrow \GL(N)$.
    The standard bilinear pairing $\F_2^6 \times \F_2^6\rightarrow \F_2$ (in which the basis $b_1, \dots, b_6$ is orthonormal) restricts to an alternating pairing on $\ker(\Sigma)$ whose radical equals $\image(\Delta)$, so induces an alternating bilinear pairing $e_N\colon N\times N\rightarrow \F_2 \simeq \{\pm 1\}$.
    Since the original pairing on $\F_2^6$ is $S_6$-invariant, $e_N$ is also $S_6$-invariant, hence $\rho$ lands in $\Sp(N,e_N)\subset \GL(N)$.
    A computation (using that the kernel of $\rho$ must be $\{1\}$, $A_6$ or $S_6$ and that the orders of $S_6$ and $\Sp_4(\F_2)$ are equal) shows that $\rho$ is an isomorphism $S_6\xrightarrow{\sim} \Sp(N,e_N)\simeq \Sp_4(\F_2)$. 
    
    Now let $(M,e)$ be a symplectic module of type $(2,2)$ over $k$. 
    View $(N,e_N)$ as a symplectic module over $k^{\sep}$, and fix an isomorphism $(M,e)_{k^{\sep}}\xrightarrow{\sim}(N,e_N)$. 
    This choice encodes the Galois action on $(M,e)$ as a homomorphism $\varphi\colon \Gal_k\rightarrow \Sp(N,e_N)$. 
    The composition $\rho^{-1}\circ \varphi\colon \Gal_k\rightarrow S_6$ corresponds to a (marked) $\Gal_k$-set $S$. 
    Let $f(x)\in k[x]$ be a monic separable polynomial of degree $6$ such that there is an isomorphism of $\Gal_k$-sets between $S$ and the roots of $f$.
    Let $C/k$ be the projective genus-$2$ hyperelliptic curve with affine equation $y^2 = f(x)$, and let $J$ be its Jacobian variety. 
    Given a root $\omega_i$ of $f(x)$, let $P_i = (\omega_i,0) \in C(k^{\sep})$ be the associated point. 
    By \cite[Proposition 6.2]{PoonenSchaefer-descentprojectiveline}, the assignment $e_i\mapsto P_i$ induces an isomorphism $N\rightarrow J[2](k^{\sep})$ which by \cite[Section 7]{PoonenSchaefer-descentprojectiveline} intertwines $e_N$ with the Weil pairing $e_2$. 
    Under this isomorphism, the Galois action $\rho\colon \Gal_k\rightarrow \Sp(N,e_N)$ is identified with the Galois action on $J[2](k^{\sep})$, so it determines an isomorphism of symplectic modules $(M,e)\simeq (J[2],e_2)$ over $k$.
    \end{proof}

    To prove the $g=3$ case of Proposition \ref{proposition: symplectic modules of F2 dimension at most 6 are J[2]} we need some preparatory results and notation.
    Let $\Phi\subset \mathbb{R}^7$ be a root system of type $E_7$ and let $\Lambda = \Z\Phi$ be the associated root lattice, which comes equipped with a bilinear positive definite pairing $(\cdot, \cdot)\colon \Lambda\times \Lambda\rightarrow \Z$ with the property that $(\alpha, \alpha)=2$ for every root $\alpha\in \Phi$. 
    Basic properties of root lattices and a construction of $\Lambda$ may be found in \cite[Section 2.2]{Lurie-minisculereps}.
    Let $\Lambda^{\vee} = \Hom(\Lambda,\Z)$.
    The pairing $(\cdot,\cdot)$ induces a map $\Lambda\rightarrow \Lambda^{\vee}$ which is injective with cokernel isomorphic to $\Z/2\Z$.
    Let $N = \image(\Lambda/2\Lambda\rightarrow \Lambda^{\vee}/2\Lambda^{\vee})\simeq \F_2^6$.
    The pairing $(\Lambda/2\Lambda)\times (\Lambda/2\Lambda)\rightarrow \{\pm 1\}, (x,y)\mapsto (-1)^{(x,y)}$ induces a perfect alternating pairing on $N$, denoted by $e_N$.
    View $(N,e_N)$ as a symplectic module over $k$ where $N$ is a constant group scheme.
    Let $W\leq \Aut(\Lambda,(\cdot,\cdot))$ be the Weyl group of $\Phi$.
    Let $W^1 = \ker(\det)\leq W$ be the index-$2$ subgroup of $W$. 
    There is a decomposition $W = W^1 \times \{\pm 1\}$.
    The $W$-action on $\Lambda$ induces a $W$-action on $N$, hence determines a homomorphism $\rho\colon W\rightarrow \Sp(N,e_N)\simeq \Sp_6(\F_2)$.
    This map is surjective with kernel $\{\pm 1\}$, and the restriction of $\rho$ to $W^1$ is an isomorphism; see \cite[p.\,229, Exercise 4]{Bourbaki-groupesalgebresdelieChapIV-VI}.

    Assume $k$ is of characteristic zero and let $V = \Lambda\otimes_{\Z} k$.
    Then $W$ acts faithfully on $V$.
    Let $k[V] = \Sym^{\bullet}(V^{\vee})$ be the ring of polynomial functions on $V$, let $k[V]^W$ be the subring of $W$-invariant polynomials, let $V\GIT W = \Spec(k[V]^W)$ and let $\pi\colon V\rightarrow V\GIT W$ be the morphism induced by the inclusion $k[V]^W\subset k[V]$.
    Let $U\subset V$ be the open subset on which $G$ acts freely. 
    Then $U$ is $W$-stable, $U/W$ is an open subset of $V\GIT G$, and the restriction of $\pi$ to $U$ is a $G$-torsor $U\rightarrow U/W$.
    For a tuple $b = (p_2, p_6,p_8,p_{10},p_{12},p_{14},p_{18})\in k^7$, let $C_b\subset \P^2$ be the projective curve with affine equation 
    \begin{align}\label{equation: plane quartics flex point}
        y^3 = x^3y +p_{10}x^2 +x(p_2y^2 + p_8y + p_{14}) + p_6y^2 + p_{12}y + p_{18}.
    \end{align}
    To state the following proposition, note that since $(N,e_N) \simeq (M_D,e_D)$ for $D=(2,2,2)$, Lemma \ref{lemma: symplectic modules vs SpD torsors} shows that $(M,e)\mapsto \IsomS((M,e),(N,e_N))$ induces a bijection between symplectic modules of type $(2,2,2)$ and $\Sp(N,e_N)$-torsors.
    \begin{proposition}\label{proposition: monodromy J[2] laga thesis}
        There exists an isomorphism of $k$-algebras $k[V]^W \simeq k[p_2, p_6,p_8,p_{10},p_{12},p_{14},p_{18}]$ (where the elements $p_i$ are algebraically independent) that induces an isomorphism of varieties $V\GIT W \simeq \A^7$ and that satisfies the following properties:
        \begin{enumerate}
            \item If $b = (p_2, \dots,p_{18})\in (U/W)(k)$, then $C_b$ is a smooth projective curve of genus $3$; let $J_b$ denote its Jacobian variety.
            \item If $b\in (U/W)(k)$, let $\mathcal{T}_b$ be the push-out of the $W$-torsor $\pi^{-1}(b)$ along $\rho\colon W\rightarrow \Sp(N,e_N)$ and let $(M_b,e_b)$ denote the symplectic module whose associated $\Sp(N,e_N)$-torsor is isomorphic to $\mathcal{T}_b$.
            Then there is an isomorphism of symplectic modules $(J_b[2],e_2)\simeq (M_b,e_b)$ over $k$.
        \end{enumerate}
    \end{proposition}
    \begin{proof}
        In \cite{Laga-ADEpaper}, a family of projective curves $C\rightarrow V\GIT W$ is constructed for every root lattice of type $A,D,E$ (beware that the notation in that paper is different to ours: what we call $V$ is denoted by $\frak{t}$). 
        In the $E_7$ case, the explicit description \eqref{equation: plane quartics flex point} follows from \cite[Proposition 3.13(3)]{Laga-ADEpaper}.
        The smoothness of $C_b$ when $b\in (U/W)(k)$ follows from \cite[Lemma 3.14]{Laga-ADEpaper}.
        The isomorphism of symplectic modules of Part 2 follows from the description of the monodromy of $J_b[2]$ in \cite[Proposition 3.22]{Laga-ADEpaper}.
    \end{proof}

\begin{proof}[Proof of Proposition \ref{proposition: symplectic modules of F2 dimension at most 6 are J[2]} when $g=3$]
    Using Proposition \ref{proposition: monodromy J[2] laga thesis} and its notation, it suffices to prove that every symplectic module of type $(2,2,2)$ over $k$ is isomorphic to $(M_b,e_b)$ for some $b\in (U/W)(k)$. 
    Equivalently, it suffices to prove that every $\Sp(N,e_N)$-torsor over $k$ is isomorphic to $\mathcal{T}_b$ for some $b\in (U/W)(k)$. 
    Equivalently, using the decomposition $W = W^1\times\{\pm 1\}$ and the fact that $\rho\colon W\rightarrow \Sp(N,e_N)$ is an isomorphism when restricted to $W^1$, it suffices to show that every $W$-torsor is isomorphic to $\pi^{-1}(b)$ for some $b\in (U/W)(k)$. 
    This follows from the ``versal torsors trick'' \cite[Part 1, \S5.4]{garibaldimerkurjevserre-cohomologicalinvariants}; for completeness, we give a proof here.
    
    Let $T$ be a $W$-torsor.
    Let $U' = U \times^W T = (U\times_k T)/W$, where $W$ acts via $w\cdot (u,t) = (w u ,w^{-1}t)$. 
    Define a $W$-action on $U'$ via the assignment on representatives $w\cdot [(u,t)] = [(w u,t)]$.
    Let $\pi'\colon U'\rightarrow U'/W$ be the quotient morphism.
    The projection $U\times_k T\rightarrow U$ induces an isomorphism $U'/W\simeq U/W$ which we use to identify $U'/W$ with $U/W$.
    The set of elements $b\in (U/W)(k)$ satisfying $T\simeq \pi^{-1}(b)$ as $W$-torsors equals $\pi'(U'(k))\subset (U'/W)(k) = (U/W)(k)$.
    Therefore it suffices to prove that $U'(k)$ is nonempty.
    But $U'$ is a nonempty open subset of $V'= V\times^W T$, which is the variety obtained by twisting the affine space $V$ along the cocycle determined by the image of $[T]\in \HH^1(k, W)$ under $\HH^1(k,W)\rightarrow \HH^1(k, \AutS(V))$. 
    Since the $W$-action on $V$ is linear, this map factors through $\HH^1(k, \GL_V)$, which is trivial by Hilbert's Theorem 90.
    Therefore $V'$ is isomorphic to $V$ and $U'$ is a nonempty open subset of the affine space $V'$.
    Since $k$ is infinite, we conclude that $U'(k)$ is nonempty.
\end{proof}

\begin{remark}
    Reformulating the $g=3$ case of Proposition \ref{proposition: symplectic modules of F2 dimension at most 6 are J[2]}, we have proved that for every field $k$ of characteristic zero and Galois representation $\rho\colon \Gal_k\rightarrow \Sp_6(\F_2)$, there exists a genus-$3$ curve $C/k$ of the form \eqref{equation: plane quartics flex point} with Jacobian $J$ such that the Galois representation associated to $J[2](k^{\sep})$ is isomorphic to $\rho$.
    This answer a question of Chidambaram \cite[Question 1.2]{Chidambaram-modpgaloisrepsnotarisingfromAVs} affirmatively, for every field of characteristic zero.
    It seems likely that a similar proof would answer the same question affirmatively for every field of characteristic $\neq 2$, using a generalization of Proposition \ref{proposition: monodromy J[2] laga thesis}, but we have not pursued this.
\end{remark}

\begin{remark}
    The proof of the $g=1,2$ cases of Proposition \ref{proposition: symplectic modules of F2 dimension at most 6 are J[2]} can also be written in the language of root lattices and correspond to the cases where $\Phi$ has type $A_2$ and $A_5$, in which case $W$ is isomorphic to $S_3$ and $S_6$ respectively.
\end{remark}

\begin{corollary}\label{corollary: symplectic modules small F2 vector spaces admit theta groups}
    Let $k$ be a field of characteristic not $2$ and let $g\in \{1,2,3\}$.
    Let $D$ be the type $(2,\dots,2)$ of length $g$.
    If $g=3$, additionally assume that $k$ has characteristic zero.
    Let $(M,e)$ be a symplectic module of type $D$ over $k$. 
    Then there exists a symmetric linear theta group for $(M,e)$.
\end{corollary}
\begin{proof}
    Suppose $k$ is finite.
    By Proposition \ref{proposition: theta groups exist over fields of cohomological dimension 1}, there exists a linear theta group $\mathcal{G}$ for $(M,e)$. 
    Since $\Id_M = -\Id_M$, $\Id_{\mathcal{G}}$ is an inversion on $\mathcal{G}$ and so $\mathcal{G}$ is also symmetric. 
    So we may assume that $k$ is infinite.
    By Proposition \ref{proposition: symplectic modules of F2 dimension at most 6 are J[2]}, there exists a smooth projective curve $C$ of genus $g$ over $k$ with Jacobian $J$ and an isomorphism $(M,e) \simeq (J[2],e_2)$.
    The pair $(J[2],e_2)$ arises from the polarization $2\lambda\colon J \rightarrow J^{\vee}$, where $\lambda$ is the canonical principal polarization on $J$. 
    Let $L = (1, \lambda)^*\mathcal{P}$ be the pullback of the Poincar\'e bundle $\mathcal{P}$ on $J\times J^{\vee}$.
    Then $L$ is a symmetric line bundle on $A$ with $\phi_L = 2\lambda$, by \cite[Proposition 6.10]{MumfordFogartyKirwan-GIT}.
    Hence $\mathcal{G}(L)$ is a symmetric linear theta group for $(J[2], e_2)\simeq (M,e)$.
\end{proof}

\subsection{Proof of Theorem \ref{theorem: intro positive result}}

We now combine the results of Sections \ref{subsection: symplectic modules odd order} and \ref{subsection: symplectic F2 vector spaces as Jacobians} to obtain Theorem \ref{theorem: intro positive result} whose proof is given at the end of this section.
We first need to discuss direct sums of symplectic modules and their interaction with theta groups.

Let $(M_1,e_1), (M_2,e_2)$ be two symplectic modules over a field $k$.
Then the direct sum $(M,e) = (M_1,e_1) \oplus (M_2,e_2)$ is the symplectic module with underlying group scheme $M=M_1\oplus M_2$ and pairing given by $e((m_1,m_2) ,(n_1,n_2)) =e_1(m_1,n_1)e_2(m_2,n_2)$.
Let $1\rightarrow \G_m\rightarrow \mathcal{G}\rightarrow M\rightarrow 1$ be a theta group for $(M,e)$.
Then the restriction $\mathcal{G}_1 = \mathcal{G}|_{M_1}$ of $\mathcal{G}$ to $M_1\subset M$ is a theta group for $(M_1,e_1)$, similarly $\mathcal{G}_2 = \mathcal{G}|_{M_2}$ is a theta group for $(M_2,e_2)$.
Every framed isomorphism restricts to framed isomorphisms on both pieces, so the association $\mathcal{G}\mapsto (\mathcal{G}_1, \mathcal{G}_2)$ can be upgraded to a functor (in the notation of \S\ref{subsec: abstract theta groups}):
\begin{align}\label{equation: functor direct sum of theta groups}
\ThetaCat_{(M_1,e_1)\oplus (M_2,e_2)} (k)\rightarrow \ThetaCat_{(M_1,e_1)}(k)\times \ThetaCat_{(M_2,e_2)}(k).
\end{align}
\begin{lemma}\label{lemma: theta groups direct sum module}
    \begin{enumerate}
        \item The functor \eqref{equation: functor direct sum of theta groups} is an equivalence of groupoids.
        \item $\mathcal{G}$ is symmetric if and only if $\mathcal{G}_1$ and $\mathcal{G}_2$ are symmetric.
        \item If $\mathcal{G}_1$ and $\mathcal{G}_2$ are linear then $\mathcal{G}$ is linear; the converse holds when $\#M_1$ and $\#M_2$ are coprime.
    \end{enumerate}
\end{lemma}
\begin{proof}
    For $i=1,2$, let $\mathcal{G}_i$ be a theta group for $(M_i,e_i)$ and let $\pi_i\colon M\rightarrow M_i$ be the projection.
    The pullback $\pi_i^*\mathcal{G}_i$ is a central extension $1\rightarrow \G_m\rightarrow \pi_i^*\mathcal{G}_i\xrightarrow{\beta_i} M\rightarrow 1$.
    Let $\mathcal{G}$ be the Baer sum of $\pi_1^*\mathcal{G}_1$ and $\pi_2^*\mathcal{G}_2$; explicitly, $\mathcal{G}$ is the quotient of 
    \[
    \{
    (g_1, g_2)\in \pi_1^*\mathcal{G}_1\times \pi_2^*\mathcal{G}_2 \colon \beta_1(g_1) = \beta_2(g_2) 
    \}
    \]
    by the subgroup $\{(\lambda, \lambda^{-1})\colon \lambda\in \G_m\}$.
    A calculation (which we omit) shows that $\mathcal{G}$ is a theta group for $(M,e)$ and provides a quasi-inverse to the functor $\mathcal{G}\mapsto (\mathcal{G}|_{M_1}, \mathcal{G}|_{M_2})$, proving Part 1.
    Part 2 follows from the fact that $\iota\mapsto (\iota_1, \iota_2)$ induces a bijection between the set of inversions on $\mathcal{G}$ and the set of pairs on inversions on $\mathcal{G}_1$ and $\mathcal{G}_2$.
    Polishchuk has shown (see \cite[Proposition 2.2]{Polishchuk-abelianvarietiesthetafunctions}) that $[\mathrm{Sh}_{\mathcal{G}}] =[\mathrm{Sh}_{\mathcal{G}_1}]+[\mathrm{Sh}_{\mathcal{G}_2}]$ in $\HH^2(k, \G_m)$. 
    Since $[\mathrm{Sh}_{\mathcal{G}_i}]$ is killed by $\#M_i$ (Proposition \ref{proposition: bounds order line bundle obstruction class}), Part 3 follows.
\end{proof}

Given a symplectic module $(M,e)$ and a prime $p$, let $M(p)$ be the submodule of elements of $p$-power order, and let $e(p)$ be the restriction of $e$ to $M(p)\times M(p)$. 
Then $(M(p),e(p))$ is a symplectic module and $(M,e) \simeq \oplus_p (M(p), e(p))$ and so combining Theorem \ref{theorem: symplectic module odd order admits theta group} and Lemma \ref{lemma: theta groups direct sum module} shows:
\begin{corollary}\label{corollary: existence theta group only depends on 2-primary part}
    In the above notation, $(M,e)$ admits a theta group if and only if $(M(2),e(2))$ does.
    The same statement holds with ``theta group'' replaced by ``linear theta group'', or by ``symmetric theta group'', or by ``linear symmetric theta group''. 
\end{corollary}

Combining Corollaries \ref{corollary: existence theta group only depends on 2-primary part} and \ref{corollary: symplectic modules small F2 vector spaces admit theta groups} immediately shows:

\begin{corollary}\label{corollary: symplectic module with M(2) small has theta group}
    Let $(M,e)$ be a symplectic module over $k$.
    Suppose that $(M(2), e(2))$ has type $(1, \dots, 1)$, $(1,\dots, 1,2)$, $(1,\dots, 1,2,2)$, or $(1, \dots, 2,2,2)$.
    In the last case, additionally assume that $k$ has characteristic zero.
    Then there exists a symmetric linear theta group for $(M,e)$.
\end{corollary}

\begin{proof}[Proof of Theorem \ref{theorem: intro positive result}]
   In the notation of the theorem, $(A[\lambda], e_{\lambda})$ satisfies the assumptions of Corollary \ref{corollary: symplectic module with M(2) small has theta group}, so there exists a linear theta group for $(A[\lambda],e_{\lambda})$.
    We conclude by Theorem \ref{theorem: intro theta group condition}.
\end{proof}

\subsection{A survey of known cases}
\label{subsection: survey of known cases}

For convenience of the reader and for future reference, we collect known positive results towards Question \ref{question: main question poonen-stoll} and its variants, either proven in this paper or already present in the literature.
Let $A$ be a $g$-dimensional abelian variety over a field $k$.
Let $\lambda\colon A\rightarrow A^{\vee}$ be a polarization of type $D = (d_1, \dots, d_g)$, and assume $\#D$ is invertible in $k$.
Consider the following properties that $(A, \lambda)$ could satisfy:
\begin{enumerate}
    \item[(1)] There exists a symmetric line bundle $L$ on $A$ with $\phi_L = \lambda$.
    \item[(2)] There exists a line bundle $L$ on $A$ with $\phi_L = \lambda$.
    \item[(3)] There exists a symmetrized $A$-torsor $(X, \tau)$ and a $\tau$-symmetric line bundle $L$ on $X$ with $\phi_L = \lambda$ (see \S\ref{subsec: symmetric line bundles representing lambda}).
    \item[(4)] There exists an $A$-torsor $X$ and a line bundle $L$ on $X$ with $\phi_L = \lambda$.
    \item[(5)] There exists a symmetrized $A$-torsor $(X, \tau)$ and element $\ell \in \PicS_X^{\tau,\lambda}(k)$.
    \item[(6)] There exists an $A$-torsor $X$ and element $\ell \in \PicS_X^{\lambda}(k)$.
\end{enumerate}
These properties are connected by the following implications:
\[\begin{tikzcd}
	{(1)} & {(3)} & {(5)} \\
	{(2)} & {(4)} & {(6)}
	\arrow[Rightarrow, from=1-1, to=1-2]
	\arrow[Rightarrow, from=1-1, to=2-1]
	\arrow[Rightarrow, from=1-2, to=1-3]
	\arrow[Rightarrow, from=1-2, to=2-2]
	\arrow[Rightarrow, from=1-3, to=2-3]
	\arrow[Rightarrow, from=2-1, to=2-2]
	\arrow[Rightarrow, from=2-2, to=2-3]
\end{tikzcd}\]
Theorems \ref{theorem: intro theta group condition} and Corollary \ref{corollary: equivalent conditions existence symmetric line bundle class on torsor} show that the validity of Properties $(3)$, $(4)$, $(5)$ and $(6)$ only depends on the isomorphism class of the symplectic module $(A[\lambda],e_{\lambda})$. 
This is not true for Property $(2)$, since there exist principally polarized abelian varieties for which $(2)$ does not hold (see the example of Jacobians in the introduction).
This is also not true for Property $(1)$ for similar reasons, see \cite[p.\, 1317, Example 3.20]{PoonenRains-thetacharacteristics}.

Complementing the criteria we have given in Section \ref{sec: polarizations and theta groups}, we state some conditions which guarantee that one of the above properties holds.
Let $2^{n_i}$ be the largest power of $2$ dividing $d_i$ and let $D_2 = (2^{n_1}, \dots, 2^{n_g})$.
\begin{itemize}
    \item Property $(1)$ holds if and only if $(A[2], e_2^{\lambda})$ admits a quadratic refinement (see the beginning of \S\ref{subsec: symmetric line bundles representing lambda}).
    This is the case whenever one of the following conditions is satisfied \cite[Proposition 3.12]{PoonenRains-thetacharacteristics}: $k$ is $\mathbb{R}$ or a finite field, or $n_2\geq 1$ (in other words $d_2$ is even). 
    \item Property $(2)$ holds if $k$ is a local field \cite[\S4, Lemma 1]{PoonenStoll-CasselsTate}.
    \item Property $(3)$ holds if $D_2 = (1, \dots, 1)$, $(1, \dots, 1,2)$, $(1, \dots, 1,2,2)$.
    It also holds if $k$ has characteristic zero and $D_2 = (1, \dots, 1,2,2,2)$. (Theorem \ref{theorem: intro positive result})
    \item Property $(3)$ holds if $k$ has cohomological dimension $\leq 1$ (Proposition \ref{proposition: theta groups exist over fields of cohomological dimension 1}). 
\end{itemize}

The first bullet point shows that if $(A, \lambda)$ does not satisfy $(1)$, then $D_2 = (1,1,2^{n_3}, \dots, 2^{n_g})$.

\begin{proposition}\label{proposition: answer to question yes when dim A is 2}
    Suppose $\dim A\leq 2$. 
    Then Property $(3)$ holds for $(A, \lambda)$.
\end{proposition}
\begin{proof}
    When $\dim A=1$, Property $(1)$ holds.
    If $\dim A=2$, then $(1)$ holds except if $D_2 = (1,1)$, in which case $(3)$ holds by Theorem \ref{theorem: intro positive result}.
\end{proof}

\begin{remark}
    We do not know whether Property $(3)$ always holds when $\dim A=3$.
This would involve analyzing pairs $(A, \lambda)$ of type $D$ with $D_2 = (1,1,2^n)$ for $n\geq 2$.
\end{remark}

By considering Jacobians of curves, one can show that $(2)$ does not necessarily imply $(1)$, that $(3)$ does not imply $(1)$ and that $(4)$ does not imply $(2)$. 
We do not know of a pair $(A, \lambda)$ for which $(6)$ holds but $(3)$ does not. 

\section{Proof of Theorem \ref{theorem: intro negative result}}\label{section: negative result}

In this section, we construct pairs $(A, \lambda)$ of type $(1, \dots, 1, 2, \dots, 2)$ for which the answer to Question \ref{question: main question poonen-stoll} is no.
In \S\ref{subsec: reformulation galois cohomology} we reformulate the existence of theta groups in terms of Galois cohomology and lifting problems.
In \S\ref{subsec: symplectic modules not admitting theta groups}, we show that for every type $(2,\dots, 2)$ of length $\geq 4$, there exists a symplectic module $(M,e)$ over some field $k$ that does not admit a theta group.
To find such an example of the form $(A[\lambda],e_{\lambda})$, we recall Siegel modular varieties in \S\ref{subsec: paramodular groups}-\ref{subsec: moduli of polarised AVs}. 
In \S\ref{subsec: negligible etale cohomology} we use a calculation of Totaro to reduce our problem to a group theory computation. 
We carry out this computation in \S\ref{subsec: polarized abelian varieties not admitting theta groups}-\ref{subsec: group theory computations}.
Explicitly, Theorem \ref{theorem: intro negative result} follows from Propositions \ref{proposition: reduce etale negligible to group cohomology calculation} and \ref{proposition: c_Gamma not in subgroup generated by phiGamma}.

\subsection{Reformulation in terms of Galois cohomology}\label{subsec: reformulation galois cohomology}

Let $k$ be a field.
Let $D = (d_1, \dots, d_g)$ be a type in the sense of Definition \ref{definition:type} and suppose $d_g$ is invertible in $k$.
In Section \ref{sec: polarizations and theta groups} we have constructed group schemes $M_D, \mathcal{G}_D$ and a sequence of finite $k$-groups
\begin{align}\label{equation: fundamental exact sequence type D field case}
1\rightarrow M_D \rightarrow \AutS(\mathcal{G}_D)\rightarrow \SpS(M_D) \rightarrow 1
\end{align}
which is exact by Lemma \ref{lemma: automorphism groups theta groups}.
Since $M_D$ is an abelian normal subgroup of $\AutS(\mathcal{G}_D)$, the conjugation action $g\cdot m = gmg^{-1}$ of $\AutS(\mathcal{G}_D)$ on $M_D$ factors through an action of $\SpS(M_D)$.
A computation shows that this agrees with the standard action of $\SpS(M_D)$ on $M_D$ via the defining embedding $\SpS(M_D)\hookrightarrow \AutS(M_D)$.
We have seen (Lemma \ref{lemma: theta group exists iff class lifts}) that a symplectic module $(M,e)$ of type $D$ admits a theta group if and only if its class $[(M,e)]\in \HH^1(k, \SpS(M_D))$ lifts to $\HH^1(k, \AutS(\mathcal{G}_D))$. 
We now show that when $k$ contains sufficiently many roots of unity, we may formulate this condition in terms of lifting Galois representations.

Let $n = d_g$ if $n$ is odd and $n=2d_g$ if $n$ is even.
Let $\mathcal{H}_D = \mathcal{G}_D[n] = \{x\in \mathcal{G}_D \colon x^n = 1\}$ be the subscheme of elements of order $n$.
Lemma \ref{lemma: equivalence Gm-theta groups and mun-theta groups} shows that $\mathcal{H}_D$ is a subgroup scheme fitting in a central extension $1\rightarrow \mu_n \rightarrow \mathcal{H}_D\rightarrow M_D\rightarrow 1$.

\begin{lemma}\label{lemma: root of unity implies constant group schemes}
    Suppose that $k$ contains a primitive $n$th root of unity.
    Then $M_D$, $\SpS(M_D)$ and $\AutS(\mathcal{G}_D)$ are constant group schemes.
 \end{lemma}
\begin{proof}
    The assumption implies $M_D$ is constant, hence $\SpS(M_D)$ is constant too.
    A calculation shows that in the notation of \eqref{equation:standardthetagroupformula} we have $\mathcal{H}_D \simeq \mu_n \times M_D$ (as schemes, not as groups).
    Therefore $\mathcal{H}_D$ is constant, hence $\AutS(\mathcal{H}_D)$ is constant too.
    By Lemma \ref{lemma: equivalence Gm-theta groups and mun-theta groups}, $\AutS(\mathcal{G}_D) \simeq \AutS(\mathcal{H}_D)$, so $\AutS(\mathcal{G}_D)$ is constant.
 \end{proof}

Suppose that $k$ contains a primitive $n$th root of unity.
Then \eqref{equation: fundamental exact sequence type D field case} is an extension of constant group schemes, so can be interpreted as an extension of (abstract) groups
 \begin{align}\label{equation: fundamental exact sequence type D abstract groups}
 1\rightarrow M_D\rightarrow \Aut(\mathcal{G}_D)\rightarrow \Sp(M_D) \rightarrow 1.
 \end{align}
Since $M_D$ is an abelian normal subgroup, this extension defines a class $c_D\in \HH^2(\Sp(M_D), M_D)$ in group cohomology, where $\Sp(M_D)$ acts on $M_D$ via the defining representation $\Sp(M_D) \hookrightarrow \Aut(M_D)$.
If $(M,e)$ is a symplectic module of type $D$ over $k$, its class $[(M,e)] \in \HH^1(k ,\SpS(M_D))$ (under the bijection of Lemma \ref{lemma: symplectic modules vs SpD torsors}) can be interpreted as a conjugacy class of continuous homomorphisms $\rho_M\colon \Gal_k\rightarrow \Sp(M_D)$.
\begin{lemma}\label{lemma: theta group exists iff embedding problem has solution}
    Assume $k$ contains a primite $n$th root of unity, where $n = d_g$ if $d_g$ is odd and $n=2d_g$ if $d_g$ is even.
    Let $\rho_M\colon \Gal_k\rightarrow \Sp(M_D)$ be a homomorphism corresponding to a symplectic module $(M,e)$ of type $D$ over $k$.
    Then there exists a theta group for $(M,e)$ if and only if $\rho_M$ lifts to a homomorphism $\Gal_k\rightarrow \Aut(\mathcal{G}_D)$ if and only if $\rho_M^*c_D = 0$ in $\HH^2(k, M_D)$. 
\end{lemma}
\begin{proof}
    The first equivalence follows from Lemma \ref{lemma: theta group exists iff class lifts}.
    The second equivalence follows from the fact that a lifting for $\rho_M$ exists if and only if the pullback of \eqref{equation: fundamental exact sequence type D abstract groups} along $\rho_M$ splits if and only if $\rho_M^*(c_D)=0$. 
\end{proof}
The sequence \eqref{equation: fundamental exact sequence type D abstract groups} is not always split: 
\begin{proposition}\label{proposition: extension class nonzero type (2,..,2)}
    Let $D$ be the type $(2, \dots, 2)$ of length $g\geq 3$.
    Then the class $c_D \in \HH^2(\Sp(M_D), M_D)$ of the extension \eqref{equation: fundamental exact sequence type D abstract groups} is nonzero. 
\end{proposition}
\begin{proof}
    The group $\mathcal{H}_D$ is an almost extraspecial $2$-group of order $2^{2+ 2g}$ (see \cite[Section 3]{Bensonetal-cohomologysymplecticgroups}).
    By Lemma \ref{lemma: equivalence Gm-theta groups and mun-theta groups}, $\Aut(\mathcal{G}_D)$ can be identified with the subgroup of group automorphisms of $\mathcal{H}_D$ that induce the identity on the center $\mu_4$.
    The nonvanishing of $c_D$ is therefore the content of \cite[p.\, 407, Corollary 2]{Griess-automorphismsofextraspecialgroups}.
\end{proof}
\begin{remark}
    In the appendix of \cite{Schmid-automorphismgroupextraspecial}, it is claimed that $c_D\neq 0$ when $D=(2, 2)$, but no proof is given.
\end{remark}

\subsection{Symplectic modules not admitting a theta group}\label{subsec: symplectic modules not admitting theta groups}

In preparation of the proof of Theorem \ref{theorem: intro negative result}, we show the following (a priori) weaker statement, using recent work of Merkurjev--Scavia \cite{MerkurjevScavia-negligiblecohomology}.
\begin{proposition}\label{proposition: symplectic modules admitting no theta group}
    Suppose that $D= (2,\dots, 2)$ has length $g\geq 4$.
    Then there exists a field $k$ of characteristic zero and a symplectic module $(M,e)$ of type $D$ over $k$ such that there exists no theta group for $(M,e)$.
\end{proposition}

This statement is weaker, since the symplectic module $(M,e)$ not admitting a theta group might not be of the form $(A[\lambda],e_{\lambda})$ for some polarized abelian variety $(A, \lambda)$.
Nevertheless, the intermediate results used to prove Proposition \ref{proposition: symplectic modules admitting no theta group} will also be used to prove Theorem \ref{theorem: intro negative result}.

Let $G$ be a finite group, $M$ a finite $G$-module and $k$ a field.
We say an element $c\in \HH^i(G,M)$ is \emph{negligible} over $k$ if for every field extension $K/k$ and every continuous homomorphism $f\colon \Gal_K\rightarrow G$, the pullback $f^*(c) \in \HH^i(\Gal_K, M) =  \HH^i(K,M)$ is trivial.
The subset of such elements forms a subgroup denoted by $\HH^i(G,M)_{\mathrm{neg},k}\subset \HH^i(G,M)$.
Under minor assumptions on $k$, Merkurjev--Scavia give an explicit description of $\HH^2(G,M)_{\mathrm{neg},k}$. 
Given $m\in M$, let $G_m = \Stab_G(m)$ and define the composite map 
\begin{align}\label{equation: varphi maps finite groups}
\varphi_m\colon \HH^2(G_m,\Z)\rightarrow \HH^2(G_m,M)\rightarrow \HH^2(G,M),
\end{align}
where the first map is induced by the $G_m$-module homomorphism $\Z\rightarrow M$ sending $1$ to $m$, and the second map is corestriction along $G_m\subset G$.
Given a finite group $H$, let $e(H)$ be its exponent, namely the least common multiple of the orders of its elements.
\begin{theorem}[Merkurjev--Scavia]\label{theorem: merkurjev-scavia}
    Suppose that $k$ contains a primitive root of unity of order $e(G)e(M)$.
    Then $\HH^2(G,M)_{\mathrm{neg},k}$ is generated by the images of $\varphi_m$ where $m$ ranges over all elements of $M$.
\end{theorem}
\begin{proof}
    This follows from \cite[Corollary 4.2]{MerkurjevScavia-negligiblecohomology}, after observing that if $m\in M$ and $\chi \in \HH^2(G_m, \Z)$, the cup product class $m\cup \chi\in \HH^2(G_m,M)$ equals the image of $\chi$ under the homomorphism $\HH^2(G_m,\Z)\rightarrow \HH^2(G_m,M)$ induced by $\Z\rightarrow M, 1\mapsto m$.
\end{proof}
Let $D$ be the type $(2, \dots, 2)$ of length $g$.
We will apply Theorem \ref{theorem: merkurjev-scavia} to the group $\Sp(M_D) \simeq \Sp_{2g}(\F_2)$ acting on $M_D \simeq  \F_2^{2g}$.

\begin{lemma}\label{lemma: abelianization of stabilizer Sp2g(F2) trivial}
    Let $m\in M_D$ be nonzero, let $G_m = \Stab_{\Sp(M_D)}(m)$ and suppose that $g\geq 4$.
    Then the abelianization of $G_m$ is trivial.
\end{lemma}
\begin{proof}
    For ease of notation, denote the form $e_D$ on $M_D$ by $\langle \cdot , \cdot \rangle$.
    Let $e_1, \dots, e_g, f_g, \dots, f_1$ be an $\F_2$-basis for $M_D$.
    Since all nondegenerate alternating forms on $M_D$ are conjugate, we may assume that $\langle e_i, f_j\rangle  = \delta_{ij}$ and $\langle e_i, e_j\rangle  = \langle f_i, f_j\rangle =0$ for all $1\leq i,j\leq g$.
    Since $\Sp_{2g}(\F_2)$ acts transitively on the nonzero elements of $\F_2^{2g}$, we may also assume $m = e_1$.
    Let $F_1 = \F_2e_1$ and $F_2 = F_1^{\perp} = \F_2\{e_1, \dots, e_g, f_g, \dots, f_2\}$.
    Then $G_m$ preserves the flag $F_1\subset F_2\subset M_D$.
    Let $U$ be the subgroup of $G_m$ of those elements that induce the identity map on $F_2/F_1$. 
    Let $L$ be the subgroup of $G_m$ that preserves the subspaces $\F_2e_1$, $\F_2\{e_2, \dots, e_g, f_g, \dots, f_2\}$ and $\F_2f_1$.
    Then $G_m = L\ltimes U$ and $L\simeq \Sp_{2g-2}(\F_2)$.
    As explicit matrix subgroups of $\Sp_{2g}(\F_2)$, we have 
    \[
    G_m = \left\{
    \left(
    \begin{array}{c| c c c | c }
    1 & * & \cdots & * & * \\\hline
    0 & * & \cdots & * & * \\
    \vdots & \vdots & \ddots & \vdots & \vdots\\
    0 & * & \cdots & * & * \\ \hline
    0 & 0 & \cdots & 0 &1
    \end{array}   
    \right)
    \right\},\,
    L = \left\{
    \left(
    \begin{array}{c| c c c | c }
    1 & 0 & \cdots & 0 & 0 \\\hline
    0 & * & \cdots & * & 0 \\
    \vdots & \vdots & \ddots & \vdots & \vdots\\
    0 & * & \cdots & * & 0 \\ \hline
    0 & 0 & \cdots & 0 &1
    \end{array}   
    \right)
    \right\},\,
    U = \left\{
    \left(
    \begin{array}{c| c c c | c }
    1 & * & \cdots & * & * \\\hline
    0 & 1 & \cdots & 0 & * \\
    \vdots & \vdots & \ddots & \vdots & \vdots\\
    0 & 0 & \cdots & 1 & * \\ \hline
    0 & 0 & \cdots & 0 &1
    \end{array}   
    \right)
    \right\}.
    \]
    It is well known (see for example \cite[Appendix]{Bensonetal-cohomologysymplecticgroups}) that $\Sp_{2n}(\F_2)$ is perfect for all $n\geq 3$.
    Since $g\geq 4$, $\Sp_{2g-2}(\F_2)$ is perfect and hence $[L,L] = L$.
    We calculate that $[U,U]$ is of order $2$, generated by the element $I + E$, where $E$ is the matrix with $1$ in the top right corner and zeroes everywhere else.
    The conjugation action of $L$ on $U$ induces an action on $U/[U,U]\simeq \F_2^{2g-2}$ and can be identified with the standard action of $\Sp_{2g-2}(\F_2)$ in its defining representation.
    When $g\geq 2$, there are no covariants for this action. 
    In conclusion, we have shown that the homology groups $\HH_1(L, \Z)$ and $\HH_0(L, \HH_1(U,\Z))$ both vanish.
    Using the low terms of the Hochschild--Serre spectral sequence $\HH_p(L,\HH_q(U,\Z))\Rightarrow \HH_{p+q}(G_m\Z)$, it follows that $G_m^{\mathrm{ab}} = \HH^1(G_m,\Z)$ vanishes too.
\end{proof}

\begin{proof}[Proof of Proposition \ref{proposition: symplectic modules admitting no theta group}] 
    We analyze the maps $\varphi_m$ of \eqref{equation: varphi maps finite groups} for the pair $(G,M) = (\Sp(M_D),M_D)$. 
    If $m=0$, then $\varphi_m = 0$, since the first map $\HH^2(G_m, \Z)\rightarrow \HH^2(G_m,M_D)$ is induced by the zero map $\Z\rightarrow M_D$.
    If $m\neq 0$, then the abelianization $G_m^{\mathrm{ab}}$ is trivial by Lemma \ref{lemma: abelianization of stabilizer Sp2g(F2) trivial}, so
    \[
    \HH^2(G_m,\Z)\simeq \HH^1(G_m,\Q/\Z) = \Hom(G_m, \Q/\Z) = \Hom(G_m^{\mathrm{ab}}, \Q/\Z) = 0.
    \]
    Let $k$ be a field of characteristic zero containing a primitive root of unity of order $e(\Sp(M_D)) e(M_D)$.
    Using Theorem \ref{theorem: merkurjev-scavia}, we conclude that $\HH^2(\Sp(M_D),M_D)_{\mathrm{neg},k}=\{0\}$.
    On the other hand, Proposition \ref{proposition: extension class nonzero type (2,..,2)} shows that $c_D\in \HH^2(\Sp(M_D), M_D)$ is nonzero.
    We conclude that $c_D$ is not negligible.
    In other words, after possibly enlarging $k$, there exists a homomorphism $\rho\colon \Gal_k\rightarrow \Sp(M_D)$ that does not lift to a homomorphism $\Gal_k\rightarrow \Aut(\mathcal{G}_D)$.
    Using Lemma \ref{lemma: symplectic modules vs SpD torsors}, $\rho$ corresponds to a symplectic module $(M,e)$.
    Since $k$ contains a primitive $4$th root of unity, Lemma \ref{lemma: theta group exists iff embedding problem has solution} shows that there does not exist a theta group for $(M,e)$.
\end{proof}

\begin{remark}
    The assumption $g\geq 4$ in Proposition \ref{proposition: symplectic modules admitting no theta group} is optimal. 
    Indeed, if $g=3$, then $c_D\neq 0$ by Proposition \ref{proposition: extension class nonzero type (2,..,2)}.
    However, Corollary \ref{corollary: symplectic modules small F2 vector spaces admit theta groups} shows that if $k$ is of characteristic zero, then every symplectic module of type $D$ over every field extension of $k$ admits a theta group.
    Therefore, by Lemma \ref{lemma: theta group exists iff embedding problem has solution}, if $k$ contains a primitive $4$th root of unity, then $c_D$ is a nonzero element of $\HH^2(\Sp(M_D), M_D)_{\mathrm{neg}, k}$.
\end{remark}

\subsection{Paramodular groups}\label{subsec: paramodular groups}

Let $D=(d_1, \dots, d_g)$ be a type of length $g$, viewed as a diagonal $g\times g$ matrix.
Define the block $2g\times 2g$-matrix
\[
J_D = 
\begin{pmatrix}
    0 & D \\
    -D & 0
\end{pmatrix}
\]
and let $(-,-)_D$ be the alternating form on $\Z^{2g}$ with Gram matrix $J_D$. 
Let $\Lambda_D$ be $\Z^{2g}$, thought of as being equipped with the pairing $(-,-)_D$. 
Define the paramodular group
\[\pargroup{g}{D}(\Z) = \Sp(\Lambda_D) = \{g\in \GL(\Lambda_D) \colon g \text{ preserves }(-,-)_D\}.\]
More generally, we view $\pargroup{g}{D}$ as a group scheme over $\Z$, where for a ring $R$ we let $\pargroup{g}{D}(R)$ be the subgroup of $\GL(\Lambda_D\otimes R)$ preserving $(-,-)_D$.

Let $D_g$ be the type $(1, \dots, 1)$ of length $g$.
Then we simply write $\Sp_{2g}(R)$ for $\pargroup{g}{D_g}(R)$ and we recover the usual symplectic group.
Denote the standard basis of $\Z^{2g}$ by $e_1, \dots, e_g, f_1, \dots, f_g$.
Consider the linear map $\alpha\colon \Lambda_D\rightarrow \Lambda_{D_g}$ sending $e_i$ to $d_ie_i$ and $f_i$ to $f_i$. 
Then $\alpha$ intertwines $(-,-)_D$ with $(-,-)_{D_g}$, hence conjugation by $\alpha$ induces an isomorphism $\pargroup{g}{D}(R)\xrightarrow{\sim} \Sp_{2g}(R)$ for every ring $R$ in which $d_g$ is invertible.
In particular, we may view $\pargroup{g}{D}(\Z)$ as a subgroup of $\pargroup{g}{D}(\Q)\simeq \Sp_{2g}(\Q)$ commensurable with $\Sp_{2g}(\Z)$.

Define $\Lambda^{\vee}_D = \{x\in \Lambda_D \otimes \Q \colon (x, y)\in \Z\text{ for all }y\in \Lambda_D\}$.
Then $\Lambda_D \subset \Lambda^{\vee}_D$ has finite index and the quotient $\Lambda_D^{\vee}/\Lambda_D$ is a finite abelian group.
Using the standard $\Z$-basis $\{e_1, \dots, e_g, f_1, \dots, f_g\}$ of $\Lambda_D$, $\Lambda_D^{\vee}$ has $\Z$-basis $\frac{1}{d_1}e_1, \dots, \frac{1}{d_g} e_g, \frac{1}{d_1}f_1, \dots, \frac{1}{d_g} f_g$.
The $\Q$-valued bilinear pairing $(-,-)_D\colon \Lambda_D^{\vee}\times \Lambda_D^{\vee} \rightarrow \Q$ induces a bilinear pairing $b_D\colon \Lambda_D^{\vee}/\Lambda_D \times \Lambda_D^{\vee}/\Lambda_D \rightarrow \Q/\Z$.

For each $n\geq 1$, let $\zeta_n = e^{2\pi i/n} \in \mathbb{C}$.
Recall the standard symplectic module $M_D$ of type $D$ from Section \ref{subsec: symplectic modules}, which we consider here as a (constant) group scheme over $\mathbb{C}$. 
In the notation of that section, consider the homomorphism 
\[
\beta\colon M_D = (\Z/d_1\Z) \times\cdots \times (\Z/d_g\Z) \times \mu_{d_1} \times \cdots \times \mu_{d_g} \rightarrow \Lambda^{\vee}_D/\Lambda_D
\]
which maps $1\in \Z/d_i\Z$ to $\frac{1}{d_i}e_i$ and $\zeta_{d_i} \in \mu_{d_i}$ to $\frac{1}{d_i}f_i$. 
Then a direct computation shows that $\beta$ is an isomorphism that intertwines the pairing $e_D$ with the pairing $e^{2\pi i b_D}$. 

Since every element of $\pargroup{g}{D}(\Z)$ induces an automorphism of $\Lambda_D^{\vee}/\Lambda_D$ preserving $b_D$, we obtain a group homomorphism $\pargroup{g}{D}(\Z)\rightarrow \Aut((\Lambda_D^{\vee}/\Lambda_D, b_D))$. 
Conjugating by $\beta$, we obtain a homomorphism $\mathrm{red}_D\colon \pargroup{g}{D}(\Z)\rightarrow  \Sp(M_D)$. 
\begin{lemma}\label{lemma: reduction map to Sp(M_D) paramodular group surjective}
    The reduction map $\mathrm{red}_D\colon \pargroup{g}{D}(\Z)\rightarrow \Sp(M_D)$ is surjective.
\end{lemma}
\begin{proof}
    This is the main result of \cite{Brasch-liftinglevelD}.
\end{proof}
Denote the kernel of this reduction map by $\Gamma(D)$. 
We then obtain an exact sequence 
\begin{align}\label{equation: exact sequence paramodular group and reduction}
    1\rightarrow \Gamma(D) \rightarrow \pargroup{g}{D}(\Z) \rightarrow \Sp(M_D)\rightarrow 1.
\end{align}
If $n\geq 1$ is an integer, define the type $nD = (nd_1, \dots, nd_g)$. 
Note that $(\cdot, \cdot)_{nD} = n (\cdot, \cdot)_D$ and $\pargroup{g}{nD}(\Z) = \pargroup{g}{D}(\Z)$ for all $n$.
We call a subgroup $\Gamma\subset \pargroup{g}{D}(\Z)$ a congruence subgroup if $\Gamma(nD)\subset \Gamma$ for some $n\geq 1$.
For later use, we will summarize some group-theoretic properties of finite-index subgroups of $\pargroup{g}{D}(\Z)$.
\begin{proposition}\label{proposition: every finite index subgroup is congruence}
    If $g\geq 2$, then every finite-index subgroup of $\pargroup{g}{D}(\Z)$ is a congruence subgroup.
\end{proposition}
\begin{proof}
    This follows from the resolution of the congruence subgroup problem for $\Sp_{2g}$ for $g\geq 2$ \cite{BassMilnorSerre}.
\end{proof}

\begin{proposition}\label{proposition: abelianization is finite}
    Let $g\geq 2$ and let $\Gamma \subset \pargroup{g}{D}(\Z)$ be a finite-index subgroup.
    Then the abelianization $\Gamma^{\mathrm{ab}}$ is finite and $\HH^2(\Gamma, \Z)$ is a finitely generated abelian group.
\end{proposition}
\begin{proof}
    Since $\Gamma$ is an arithmetic group, it is finitely presented \cite[\S4.4, Theorem 4.8]{PlatonovRapinchuk-secondeditionvolume1}.
    Therefore $\Gamma^{\mathrm{ab}}$ is a finitely generated abelian group.
    Since $\Sp_{2g}$ has real rank $\geq 2$, $\Gamma$ satisfies Kazhdan's property (T) \cite[Theorem 7.1.4]{zimmer-ergodictheory}.
    Consequently \cite[Theorem 7.1.7]{zimmer-ergodictheory}, there are no nontrivial homomorphisms $\Gamma\rightarrow \mathbb{R}$.
    Therefore the finitely generated group $\Gamma^{\mathrm{ab}}$ must be finite.

    Since $\Gamma$ is finitely presented, Hopf's formula \cite[Chapter II, \S5, Theorem 5.3]{Brown-cohomologyofgroups} shows that $\HH_2(\Gamma, \Z)$ is finitely generated. 
    By the universal coefficient theorem and the finiteness of $\Gamma^{\mathrm{ab}}$, we conclude that $\HH^2(\Gamma, \Z)$ is also finitely generated.
\end{proof}

\begin{proposition}[Borel]\label{proposition: borel computation}
    Let $g\geq 3$ and let $\Gamma\subset \pargroup{g}{D}(\Z)$ be a finite-index subgroup.
    Then $\HH^2(\Gamma,\Q)\simeq \Q$ and the restriction map $\HH^2(\pargroup{g}{D}(\Z), \Q)\rightarrow \HH^2(\Gamma, \Q)$ is an isomorphism.
\end{proposition}
\begin{proof}
    This follows from results of Borel \cite{Borel-stablerealcohomologyarithmeticgroupsI} with optimized stable range, see \cite[Theorem 2]{Tshishiku-borelsstablerange}. 
    (The result in \cite{Tshishiku-borelsstablerange} is only stated for finite-index subgroups of $\Sp_{2g}(\Z)$, but the proof applies more generally to finite-index subgroups of $\pargroup{g}{D}(\Z)$.)
\end{proof}

Combining Propositions \ref{proposition: abelianization is finite} and \ref{proposition: borel computation} we obtain, for every $g\geq 3$ and finite-index subgroup $\Gamma\subset \Sp_{2g}(\Z)$, an exact sequence
\begin{align}\label{equation: structure H2 of Gamma}
    1\rightarrow \HH^1(\Gamma, \Q/\Z)\rightarrow \HH^2(\Gamma, \Z)\rightarrow \Z\rightarrow 1
\end{align}
induced by the sequence $1\rightarrow \Z\rightarrow \Q\rightarrow \Q/\Z\rightarrow 1$. 
Since $\HH^1(\Gamma, \Q/\Z)\simeq \Hom(\Gamma^{\mathrm{ab}}, \Q/\Z)$, computing $\HH^2(\Gamma, \Z)$ boils down to computing the abelianization of $\Gamma$.

To compute such abelianizations, it will be useful to consider $p$-adic and adelic variants of congruence subgroups.
Let $g\geq 2$ and let $\Gamma \subset \pargroup{g}{D}(\Z)$ be a finite-index subgroup.
For each $n\geq 1$, let $\Gamma_n$ denote the image of $\Gamma$ under the reduction map $\pargroup{g}{D}(\Z) = \pargroup{g}{nD}(\Z)\rightarrow \Sp(M_{nD})$. 
Denote the profinite completion of $\Gamma$ by $\hat{\Gamma}$.
By Proposition \ref{proposition: every finite index subgroup is congruence}, the natural map $\hat{\Gamma} \rightarrow  \varprojlim_{n} \Gamma_n$ (where $n$ ranges over positive integers) is an isomorphism. 
For example, if $\Gamma = \pargroup{g}{D}(\Z)$, then by Lemma \ref{lemma: reduction map to Sp(M_D) paramodular group surjective} we have $\hat{\Gamma}\simeq \pargroup{g}{D}(\hat{\Z})$, where $\hat{\Z} = \varprojlim_n \Z/n\Z$.
Given a topological group $G$ (such as a discrete or profinite group), denote by $G^{\mathrm{der}}$ the closure of the commutator subgroup of $G$ and let $G^{\mathrm{ab}} = G/G^{\mathrm{der}}$.

\begin{lemma}\label{lemma: completion of abelianization}
    Let $g\geq 2$ and let $\Gamma\subset \pargroup{g}{D}(\Z)$ be a finite-index sugroup.
    Then the natural map $\Gamma^{\mathrm{ab}}\rightarrow (\hat{\Gamma})^{\mathrm{ab}}$ is an isomorphism of finite groups.
\end{lemma}
\begin{proof}
    This follows from the identity $\widehat{\Gamma^{\mathrm{ab}}} \simeq (\hat{\Gamma})^{\mathrm{ab}}$ and the fact that $\Gamma^{\mathrm{ab}}$ is finite by Proposition \ref{proposition: abelianization is finite}.
\end{proof}

\begin{lemma}\label{lemma: abelianizations symplectic groups}
    Let $g\geq 2$ be an integer and $p$ a prime.
    \begin{enumerate}
        \item $\Sp_{2g}(\Z)^{\mathrm{ab}}$ is trivial when $g\geq 3$.
        \item If $g\geq 3$ or $p\geq 3$, then $\Sp_{2g}(\Z_p)^{\mathrm{ab}}$ is trivial.
        \item If $R = \Z$ or $\Z_p$, let $\Sp_{2g}(R,p) = \ker(\Sp_{2g}(R)\rightarrow \Sp_{2g}(R/pR))$.
        If $p$ is odd, then $\Sp_{2g}(\Z,p)^{\mathrm{ab}}\simeq \Sp_{2g}(\Z_p,p)^{\mathrm{ab}} \simeq (\Z/p\Z)^{2g^2+g}$.
    \end{enumerate}
\end{lemma}
\begin{proof}
    Parts 1 and 2 are classical, see for example \cite[Lemma 1 and Proposition 1]{LandesmanSwaminathanetal-liftingsubgroups}.
    The third part follows from \cite[Proposition 10.1 and Corollary 10.2]{sato-abelianizationlevelD}, the fact that $\widehat{\Sp_{2g}(\Z,p)} \simeq \Sp_{2g}(\Z_p, p) \times \prod_{\ell\neq p} \Sp_{2g}(\Z_{\ell})$ and Lemma \ref{lemma: completion of abelianization}.
\end{proof}

\subsection{Moduli of polarized abelian varieties}\label{subsec: moduli of polarised AVs}

Let $D = (d_1, \dots, d_g)$ be a type of length $g$.
Let $\mathcal{A}_D\rightarrow \Spec(\mathbb{C})$ be the moduli stack of $g$-dimensional polarized abelian varieties of type $D$ \cite[Section 1]{deJong-modulispacepolarized}.
It is classical (see \cite[Theorem 2.1.11]{olsson-survey} or \cite[Chapter 7, Proposition 7.9 and Theorem 7.10]{MumfordFogartyKirwan-GIT}) that $\mathcal{A}_D$ is a smooth Deligne--Mumford stack whose coarse space is a quasi-projective variety.
Let $\mathcal{A}_{\Gamma(D)}\rightarrow \Spec(\mathbb{C})$ be the moduli stack of triples $(A,\lambda, \alpha)$, where $(A, \lambda)$ is a polarized abelian variety of type $D$ and $\alpha$ is an isomorphism of symplectic modules $(A[\lambda], e_{\lambda})\xrightarrow{\sim} (M_D, e_D)$. 
Then $\mathcal{A}_{\Gamma(D)}\rightarrow \mathcal{A}_D$ is a torsor under the group $\Sp(M_D)$, so is a smooth Deligne--Mumford stack too. 
If $d_1\geq 3$, then objects of $\mathcal{A}_{\Gamma(D)}$ have no nontrivial automorphisms (see \cite[Corollary 5.1.10]{BirkenhakeLange-AVs}) and $\mathcal{A}_{\Gamma(D)}$ is representable by a smooth quasi-projective variety over $\mathbb{C}$.

Let $\Gamma \subset \pargroup{g}{D}(\Z)$ be subgroup containing $\Gamma(nD)$ for some $n$ with $nd_1\geq 3$. 
Then the finite group $\Gamma/\Gamma(nD)$ acts on the variety $\mathcal{A}_{\Gamma(nD)}$; we denote the corresponding quotient stack by $\mathcal{A}_{\Gamma}$. 
If $\Gamma$ is torsion-free (i.e., every element of finite order equals the identity) then this $\Gamma/\Gamma(nD)$-action is free, and $\mathcal{A}_{\Gamma}$ is again a smooth variety. 
It carries a universal abelian scheme $A^{\mathrm{univ}}_{\Gamma}\rightarrow \mathcal{A}_{\Gamma}$ equipped with a polarization of type $D$.

Let $\mathfrak{h}_g = \{Z\in \mathrm{Mat}_{g\times g}(\mathbb{C})\colon Z^t = Z\text{ and }\mathrm{Im}(Z) \text{ is positive definite}\}$ be the Siegel upper-half space of genus $g$.
The group $\pargroup{g}{D}(\mathbb{R})$ acts on $\mathfrak{h}_g$ via the association 
\[
g\cdot Z =  (\alpha Z + \beta)(\gamma Z + \delta)^{-1}  \quad \text{ if } 
g = \begin{pmatrix} \alpha & \beta \\ \gamma & \delta \end{pmatrix}\in \pargroup{g}{D}(\Q), \, Z\in \frak{h}_g
\]
Let $\mathcal{A}_{\Gamma}^{\mathrm{an}}$ be the analytification of $\mathcal{A}_{\Gamma}$. 
To every $Z\in \frak{h}_g$, we can associate a polarized abelian variety $(A_Z, \lambda_Z)$ with a symplectic basis (see \cite[Section 8.1]{BirkenhakeLange-AVs}), and this association defines an isomorphism of orbifolds
\[
\mathcal{A}_{\Gamma}^{\mathrm{an}} \simeq [\Gamma \backslash \frak{h}_g].
\]
If $\Gamma$ is torsion-free, then $\mathcal{A}_{\Gamma}^{\mathrm{an}}$ is a connected complex manifold whose universal cover $\frak{h}_g$ is contractible.
Therefore, $\mathcal{A}_{\Gamma}^{\mathrm{an}}$ is a classifying space for its fundamental group $\Gamma$ (equivalently, a $K(\Gamma,1)$ space).
Consequently, if $\sh{F}$ is a locally constant sheaf of finite commutative groups on the \'etale site of $\mathcal{A}_{\Gamma}$, then there are canonical isomorphisms
\begin{align}\label{equation: isos etale betti group cohomology}
    \HH^2(\mathcal{A}_{\Gamma},\sh{F}) \simeq \HH^2(\mathcal{A}_{\Gamma}^{\mathrm{an}}, \sh{F}^{\mathrm{an}}) \simeq \HH^2(\Gamma, M_{\sh{F}}),
\end{align}
where $\HH^2$ denotes \'etale, singular and group cohomology respectively, $\sh{F}^{\mathrm{an}}$ denotes the analytification of $\sh{F}$, and $M_{\sh{F}}$ denotes the $\Gamma$-representation corresponding to the local system $\sh{F}$ on $\mathcal{A}_{\Gamma}^{\mathrm{an}}$.
The first isomorphism is the comparison isomorphism between \'etale and singular cohomology \cite[Expos\'e IX, Th\'eor\`eme 4.4]{SGA4Tome3}; the second isomorphism is the comparison isomorphism between the singular cohomology of a classifying space and the group cohomology of its fundamental group \cite[Section III.1, p.\, 59]{Brown-cohomologyofgroups}.


\subsection{Twisted Chow groups and negligible \'etale cohomology}\label{subsec: negligible etale cohomology}

Let $X$ be a smooth and geometrically integral separated scheme of finite type over a field $k$ with function field $k(X)$.
Let $\sh{F}$ be a locally constant \'etale sheaf of finite commutative groups which is killed by a positive integer $n$ that is invertible in $k$.
\begin{definition}\label{definition: negligible etale cohomology}
    Say a class $c\in \HH^i(X, \sh{F})$ is \emph{negligible} if for every field extension $K/k$ and $k$-morphism $f\colon \Spec(K) \rightarrow X$, $f^*(c) =0$ in $\HH^i(\Spec(K), f^*\sh{F})$.
\end{definition}
(Compare this definition with the one made in \S\ref{subsec: symplectic modules not admitting theta groups} in the context of group cohomology; morally this definition coincides with the group cohomology one when ``$X = BG$''.)
Write $\HH^i(X, \sh{F})_{\mathrm{neg}}$ for the subgroup of negligible classes.
Reinterpreting and generalizing the computation of Merkurjev--Scavia in group cohomology \cite{MerkurjevScavia-negligiblecohomology}, Totaro \cite{Totaro-Chowringtwistedcoefficients} has determined $\HH^2(X, \sh{F})_{\mathrm{neg}}$ when $\sh{F}$ has finite monodromy.
He states this (and much more) in the language of twisted Chow groups; we extract the concrete statement that we need for our purposes.

Let $G$ be a finite group and $\pi\colon Y\rightarrow X$ a $G$-torsor with $Y$ connected. 
Suppose that $\pi^*\sh{F}$ is constant and write $M$ for the corresponding finite abelian group with $G$-action.
For each $m\in M$, let $G_m = \Stab_G(m)$ and consider the quotient $Y_m = Y/G_m$ and the factorization $Y\rightarrow Y_m \xrightarrow{\pi_m} X$ of $\pi$. 
Define the composition 
\begin{align}\label{equation: definition phi_m etale case}
    \varphi_m\colon \Pic(Y_m)\otimes \Z/n\Z \rightarrow \HH^2(Y_m, \mu_n) \rightarrow
    \HH^2(Y_m, \pi_m^*\sh{F}(1)) \rightarrow \HH^2(X, \sh{F}(1)),
\end{align}
where:
\begin{itemize}
    \item $\Pic(Y_m)\otimes \Z/n\Z \rightarrow \HH^2(Y_m,\mu_n)$ is the connecting map coming from the Kummer sequence $1\rightarrow \mu_n\rightarrow \G_m\rightarrow \G_m\rightarrow 1$; 
    \item $\mu_n\rightarrow \pi_m^*\sh{F}(1) \coloneqq \pi_m^*\sh{F} \otimes_{\Z/n\Z} \mu_n$ is the twist of the map $\Z/n\Z\rightarrow \pi_m^*\sh{F}$ of etale sheaves on $Y_m$ corresponding to the $G_m$-module map $\Z/n\Z\rightarrow M$ sending $1$ to $m$; and
    \item $\HH^2(Y_m, \pi_m^*\sh{F}(1))\rightarrow \HH^2(Y,\sh{F}(1))$ is the transfer (or corestriction) map associated to the finite etale cover $\pi_m\colon Y_m\rightarrow X$.
\end{itemize}

\begin{theorem}[Totaro]\label{theorem: Totaro negligible version}
In the above notation, $\ker(\HH^2(X,\sh{F}(1))\rightarrow \HH^2(k(X), \sh{F}(1)))$ equals the subgroup generated by the images of $\varphi_m$, where $m$ ranges over the elements of $M$.
\end{theorem}
\begin{proof}
    Totaro defines the twisted Chow group $\CH^1(X, \sh{F})$ and cycle map $\mathrm{cl}_{X, \sh{F}}\colon\CH^1(X, \sh{F})\rightarrow \HH^2(X, \sh{F}(1))$ whose image equals the kernel of $\HH^2(X, \sh{F}(1)) \rightarrow \HH^2(k(X), \sh{F}(1))$, see \cite[Lemma 8.2]{Totaro-Chowringtwistedcoefficients}
    On the other hand, he shows \cite[Theorem 8.1]{Totaro-Chowringtwistedcoefficients} that $\CH^1(X, \sh{F})$ is generated by the images of maps $\psi_m\colon \Pic(Y_m) =\CH^1(Y_m) \rightarrow \CH^1(Y_m, \sh{F}) \rightarrow \CH^1(Y,\sh{F})$, as $m$ ranges over $M$. 
    Unwinding the definition of these maps shows that $\mathrm{cl}_{X, \sh{F}}\circ \psi_m = \varphi_m$.
\end{proof}

\begin{corollary}
    In the above notation, the following are equivalent for a class $c\in \HH^2(X, \sh{F}(1))$:
    \begin{enumerate}
        \item $c$ lies in the subgroup generated by the images of $\varphi_m$, where $m$ ranges over the elements of $M$;
        \item $c$ is negligible in the sense of Definition \ref{definition: negligible etale cohomology};
        \item if $\eta\colon \Spec( k(X)) \rightarrow X$ denotes the generic point, then $\eta^*c=0$.
    \end{enumerate}
\end{corollary}
\begin{proof}
    The equivalence $(1)\Leftrightarrow (3)$ is Theorem \ref{theorem: Totaro negligible version}, and the implication $(2) \Rightarrow (3)$ follows from the definition.
    It remains to show $(1) \Rightarrow (2)$.
    To this end, it suffices to show for each $m\in M$ that the image of $\varphi_m$ lands in $\HH^2(X, \sh{F}(1))_{\mathrm{neg}}$. 
    Fix such an $m$, let $\ell \in \Pic(Y_m)$ and let $f\colon \Spec(K)\rightarrow X$ be a $k$-morphism for some field $K/k$.
    The pullback of $\pi_m\colon Y_m \rightarrow X$ along $f$ is an \'etale $K$-algebra $\Spec(L) \rightarrow \Spec(K)$ and comes equipped with a morphism $f_m\colon \Spec(L)\rightarrow Y_m$ lifting $f$.
    By compatibility between restriction and corestriction, $f^*\varphi_m(\ell)$ equals the image of $f_m^*\ell \in \Pic(\Spec(L))$ under a map $\Pic(\Spec(L))\otimes \Z/n\Z\rightarrow \HH^2(\Spec(K), \sh{F}(1))$. 
    Since $\Pic(\Spec(L)) = 0$, we conclude that $f^*\varphi_m(\ell)=0$, as desired.
\end{proof}

We now apply these considerations to our situation of interest.
Let $D = (d_1, \dots, d_g)$ be a type and let $\Gamma\leq \pargroup{g}{D}(\Z)$ be a torsion-free arithmetic subgroup.
Let $\mathcal{A}_{\Gamma}\rightarrow \Spec(\mathbb{C})$ be the corresponding Siegel modular variety constructed in Section \ref{subsec: moduli of polarised AVs}.
Assume that the restriction of the reduction map $\mathrm{red}_{D}\colon \pargroup{g}{D}(\Z)\rightarrow \Sp(M_D)$ to $\Gamma$ is surjective.
Restricting the sequence \eqref{equation: exact sequence paramodular group and reduction} to $\Gamma$, we obtain the exact sequence
\begin{align}\label{equation: exact sequence Gamma reduction}
1\rightarrow \Gamma\cap \Gamma(D) \rightarrow \Gamma \xrightarrow{\mathrm{red}_{\Gamma}}\Sp(M_D)\rightarrow 1.
\end{align}
Recall that $c_D\in \HH^2(\Sp(M_D), M_D)$ denotes the class of the extension \eqref{equation: fundamental exact sequence type D abstract groups}; let $c_{D,\Gamma}\in \HH^2(\Gamma,M_D)$ denote the pullback of $c_D$ along $\mathrm{red}_{\Gamma}\colon \Gamma\rightarrow \Sp(M_D)$. 
View $M_D$ as a locally constant \'etale sheaf on $\mathcal{A}_{\Gamma}$, trivialized along the $\Sp(M_D)$-cover $\mathcal{A}_{\Gamma\cap \Gamma(D)}\rightarrow \mathcal{A}_{\Gamma}$ and with monodromy given by the $\Sp(M_D)$-action on $M_D$.
Under the comparison isomorphisms \eqref{equation: isos etale betti group cohomology}, $c_{D,\Gamma}$ corresponds to a class $c_{D, \Gamma}^{\mathrm{et}}\in \HH^2(\mathcal{A}_{\Gamma}, M_D)$. 

Fix an element $m\in M_D$.
Let $\Gamma_m\subset \Gamma$ be the stabilizer of $m$; this is the preimage of $\Stab_{\Sp(M_D)}(m)$ under the reduction map $\Gamma\rightarrow \Sp(M_D)$. 
Consider the composition 
\begin{align}\label{equation: definition phi_m Gamma case}
    \varphi_{\Gamma,m}\colon \HH^2(\Gamma_m, \Z)\rightarrow \HH^2(\Gamma_m,M_D)\rightarrow \HH^2(\Gamma,M_D),
\end{align}
where the first map is induced by the $\Gamma_m$-module map $\Z\rightarrow M_D$ sending $1$ to $m$, and the second map is corestriction.
Denote the generic fiber of the universal abelian scheme $A_{\Gamma}^{\mathrm{univ}} \rightarrow \mathcal{A}_{\Gamma}$ by $A^{\mathrm{gen}}_{\Gamma}$; it is an abelian variety over the function field $\mathbb{C}(\mathcal{A}_{\Gamma})$ equipped with a polarization $\lambda$ of type $D$.

\begin{proposition}\label{proposition: reduce etale negligible to group cohomology calculation}
    Suppose that $c_{D,\Gamma}$ does not lie in the subgroup generated by the images of $\varphi_{\Gamma,m}$, where $m$ ranges over the elements of $M_D$.
    Then $c_{D, \Gamma}^{\mathrm{et}}\in \HH^2(\mathcal{A}_{\Gamma}, M_D)$ is not negligible and the answer to Question \ref{question: main question weaker variant} is no for the pair $(A_{\Gamma}^{\mathrm{gen}}, \lambda)$ over $\mathbb{C}(\mathcal{A}_{\Gamma})$. 
\end{proposition}
\begin{proof}
    The module $M_D$ is killed by $n\coloneq d_g$.
    Since we are working over $\mathbb{C}$, we may identify $\mu_n$ with $\Z/n\Z$, hence $M_D(1)$ with $M_D$.
    Consider the Chern class map $c_1\colon \Pic(\mathcal{A}_{\Gamma_m})\rightarrow \HH^2(\mathcal{A}_{\Gamma_m}^{\mathrm{an}}, \Z)\simeq \HH^2(\Gamma_m, \Z)$ arising from the exponential sequence and the comparison isomorphism \eqref{equation: isos etale betti group cohomology}.
    The maps defining $\varphi_m$ of \eqref{equation: definition phi_m etale case} and $\varphi_{\Gamma,m}$ of \eqref{equation: definition phi_m Gamma case} fits into a commutative diagram:
\[\begin{tikzcd}
	{\Pic(\mathcal{A}_{\Gamma_m})} & {\HH^2(\mathcal{A}_{\Gamma_m}, \Z/n\Z)} & {\HH^2(\mathcal{A}_{\Gamma_m}, M_D)} & {\HH^2(\mathcal{A}_{\Gamma}, M_D)} \\
	{\HH^2(\Gamma_m, \Z)} & {\HH^2(\Gamma_m, \Z/n\Z)} & {\HH^2(\Gamma_m,M_D)} & {\HH^2(\Gamma,M_D)}
	\arrow[from=1-1, to=1-2]
	\arrow["{c_1}", from=1-1, to=2-1]
	\arrow[from=1-2, to=1-3]
	\arrow["\simeq", from=1-2, to=2-2]
	\arrow[from=1-3, to=1-4]
	\arrow["\simeq", from=1-3, to=2-3]
	\arrow["\simeq", from=1-4, to=2-4]
	\arrow[from=2-1, to=2-2]
	\arrow[from=2-2, to=2-3]
	\arrow[from=2-3, to=2-4]
\end{tikzcd}\]
    The commutativity of this diagram follows from the fact that the comparison isomorphisms \eqref{equation: isos etale betti group cohomology} identify $c_1$ with the \'etale cycle class map and corestriction in \'etale cohomology with corestriction in group cohomology.

    It follows that the image of $\varphi_m$ in $\HH^2(\mathcal{A}_{\Gamma},M_D)\simeq \HH^2(\Gamma,M_D)$ is contained in the image of $\varphi_{\Gamma,m}$ for every $m\in M_D$. 
    The assumptions and Theorem \ref{theorem: Totaro negligible version} therefore imply that $c_{D, \Gamma}^{\mathrm{et}}$ is not negligible and that its image $c$ in $\HH^2(k, M_D)$ is nonzero, where $k=\mathbb{C}(\mathcal{A}_{\Gamma})$.
    By definition, $c$ is the pullback of the class $c_D\in \HH^2(\Sp(M_D), M_D)$ along the homomorphism $\rho\colon \Gal_k \rightarrow \Sp(M_D)$ corresponding to the symplectic module $(A_{\Gamma}^{\mathrm{gen}}[\lambda],e_{\lambda})$ over $k$. 
    The fact that $c$ is nonzero means (by Lemma \ref{lemma: theta group exists iff embedding problem has solution}) that there does not exist a theta group for $(A_{\Gamma}^{\mathrm{gen}}[\lambda],e_{\lambda})$.
    By Theorem \ref{theorem: intro theta group condition}, this implies that the answer to Question \ref{question: main question weaker variant} is no for $(A_{\Gamma}^{\mathrm{gen}}, \lambda)$.
\end{proof}



In the remainder of the paper, we will perform group theory computations to verify the assumptions of Proposition \ref{proposition: reduce etale negligible to group cohomology calculation} for various $D$ and $\Gamma$.

\subsection{Polarized abelian varieties not admitting theta groups}\label{subsec: polarized abelian varieties not admitting theta groups}

Let $n,k\geq 0$ be integers and $g=n+k$. 
Let $D = (1, \dots, 1, 2, \dots, 2)$, where $1$ occurs $n$ times and $2$ occurs $k$ times.
Fix a prime $p\geq 3$ and let $\Gamma = \ker(\pargroup{g}{D}(\Z)\rightarrow \pargroup{g}{D}(\Z/p\Z))$.
The arithmetic subgroup $\Gamma$ is torsion-free.
The restriction of the reduction map $\pargroup{g}{D}(\Z)\rightarrow \Sp(M_D)$ to $\Gamma$ is surjective, and we have an exact sequence
\[
1\rightarrow \Gamma(pD)\rightarrow \Gamma \rightarrow \Sp(M_D)\rightarrow 1
\]
\begin{proposition}\label{proposition: c_Gamma not in subgroup generated by phiGamma}
    Suppose that $n\geq 3$ and $k\geq 4$.
    Then the class $c_{D, \Gamma}\in \HH^2(\Gamma, M_D)$ does not lie in the subgroup generated by the images of $\varphi_{\Gamma, m}$ of \eqref{equation: definition phi_m Gamma case} where $m$ ranges over the elements of $M_D$.
\end{proposition}
The proof of this proposition is given at the end of this subsection, after some preparations.

\begin{proposition}\label{proposition: abelianization torsion-free congruence subgroup has odd order}
    If $n\geq 3$ and $k\geq 2$, then $\Gamma(pD)^{\mathrm{ab}} \otimes \Z/2\Z =0$.
\end{proposition}
\begin{proof}
    Let $\Gamma(pD)_2 = \ker(\pargroup{g}{D}(\Z_2)\rightarrow \Sp(M_D))$ and $\Gamma(pD)_p = \ker(\Sp_{2g}(\Z_p)\rightarrow \Sp_{2g}(\Z/p\Z))$.
    By Lemma \ref{lemma: reduction map to Sp(M_D) paramodular group surjective}, we have an isomorphism of profinite groups
    \[
    \widehat{\Gamma(pD)} \simeq \Gamma(pD)_2\times \Gamma(pD)_p \times \prod_{\ell\neq 2,p} \Sp_{2g}(\Z_{\ell}).
    \]
    By Lemma \ref{lemma: completion of abelianization}, it suffices to compute $\widehat{\Gamma(pD)}^{\mathrm{ab}} $.
    By Lemma \ref{lemma: abelianizations symplectic groups}, $\Sp_{2g}(\Z_{\ell})^{\mathrm{ab}} = 1$ for all $\ell\neq 2,p$ and $\Gamma(pD)_p^{\mathrm{ab}}\simeq (\Z/p\Z)^{2g^2+g}$.
    Therefore it suffices to prove that $\Gamma(pD)_2^{\mathrm{ab}} = \{1\}$.
    This is a tedious exercise in group theory; we postpone the proof to the appendix (Theorem \ref{theorem: abelianization level D subgroup over 2-adics is trivial}).
\end{proof}

\begin{proposition}\label{proposition: extension cDGamma not zero}
    If $n\geq 3$ and $k\geq 3$, then $c_{D, \Gamma}\neq 0$.
\end{proposition}
\begin{proof}
    The five-term exact sequence associated to the Hochschild–Serre spectral sequence $\HH^p(\Sp(M_D),\HH^q(\Gamma(pD),M_D)) \Rightarrow \HH^{p+q}(\Gamma,M)$
    has the form
    \begin{align*}
    \cdots \rightarrow \HH^1(\Gamma(pD), M_D)^{\Sp(M_D)}
    \rightarrow \HH^2(\Sp(M_D), M_D)\xrightarrow{f} \HH^2(\Gamma, M_D).
\end{align*}
    Since the $\Gamma(pD)$-action on $M_D$ is trivial, we have $\HH^1(\Gamma(pD), M_D) = \Hom(\Gamma(pD)^{\mathrm{ab}}, M_D)$, which is trivial by Proposition \ref{proposition: abelianization torsion-free congruence subgroup has odd order} and the fact that $M_D$ is killed by $2$.
    Therefore the final map $f$ is injective.
    By definition, $f$ sends $c_D$ to $c_{D, \Gamma}$.
    Since $k\geq 3$, $c_D\neq 0$ by Proposition \ref{proposition: extension class nonzero type (2,..,2)}, so $c_{D,\Gamma} = f(c_D)$ is nonzero.
\end{proof}

\begin{proposition}\label{proposition: stabilizer m in Gamma has odd abelianization and cokernel on H2 odd order}
    Let $m\in M_D$ be nonzero and let $\Gamma_m\subset \Gamma$ be the stabilizer of $m$.
    Assume $n\geq 3$ and $k\geq 4$.
    Then:
    \begin{enumerate}
        \item $\Gamma_m^{\mathrm{ab}} \otimes (\Z/2\Z) = 0$.
        \item The cokernel of the restriction map $\mathrm{res}\colon \HH^2(\Gamma,\Z)\rightarrow \HH^2(\Gamma_m,\Z)$ is finite of odd order.
    \end{enumerate}
\end{proposition}
\begin{proof}
    \begin{enumerate}
        \item Let $G_m\subset \Sp(M_D)$ be the stabilizer of $m$.
    Since $\Gamma\rightarrow \Sp(M_D)$ is surjective, $\Gamma_m/\Gamma(pD)\simeq G_m$.
    The five-term exact sequence in group homology associated to the spectral sequence $\HH_p(G_m , \HH_q(\Gamma(pD), \F_2))\Rightarrow \HH_{p+q}(\Gamma_m,\F_2)$ looks like
    \[
    \cdots \rightarrow \HH_0(\Sp(M_D), \HH_1(\Gamma(pD), \F_2))\rightarrow \HH_1(\Gamma_m, \F_2)\rightarrow \HH_1(G_m, \F_2)\rightarrow 1.
    \]
    By Proposition \ref{proposition: abelianization torsion-free congruence subgroup has odd order}, $\HH_1(\Gamma(pD), \F_2) =\Gamma(pD)^{\mathrm{ab}} \otimes \F_2 = 0$. 
    By Lemma \ref{lemma: abelianization of stabilizer Sp2g(F2) trivial} and the assumption $k\geq 4$, $\HH_1(G_m, \F_2)=0$.
    Therefore $\Gamma_m\otimes \F_2 = \HH_1(\Gamma_m, \F_2) = 0$.
    \item The exact sequences \eqref{equation: structure H2 of Gamma} for $\Gamma$ and $\Gamma_m$ are connected by a commutative diagram 
\[\begin{tikzcd}
	1 & {\HH^1(\Gamma, \Q/\Z)} & {\HH^2(\Gamma, \Z)} & \Z & 1 \\
	1 & {\HH^1(\Gamma_m,\Q/\Z)} & {\HH^2(\Gamma_m, \Z)} & \Z & 1
	\arrow[from=1-1, to=1-2]
	\arrow[from=1-2, to=1-3]
	\arrow["\alpha", from=1-2, to=2-2]
	\arrow[from=1-3, to=1-4]
	\arrow["{\mathrm{res}}", from=1-3, to=2-3]
	\arrow[from=1-4, to=1-5]
	\arrow["\beta", from=1-4, to=2-4]
	\arrow[from=2-1, to=2-2]
	\arrow[from=2-2, to=2-3]
	\arrow[from=2-3, to=2-4]
	\arrow[from=2-4, to=2-5]
\end{tikzcd}\]
    By the snake lemma, it suffices to prove that the cokernels of $\alpha$ and $\beta$ are finite of odd order.
    For the cokernel of $\alpha$, this follows from Part 1 and the fact that $\HH^1(\Gamma_m, \Q/\Z) = \Hom(\Gamma_m^{\mathrm{ab}}, \Q/\Z)$.
    To analyze $\beta$, let $\mathrm{cores}\colon \HH^2(\Gamma_m,\Z)\rightarrow \HH^2(\Gamma, \Z)$ the corestriction map.
    Then $\mathrm{cores}\circ \mathrm{res} = [\Gamma:\Gamma_m]$.
    Since the action of $\Sp(M_D)$ on the nonzero elements of $M_D$ is transitive, $[\Gamma:\Gamma_m] = [\Sp(M_D):G_m] = \#M_D-1 = 2^{2k}-1$.
    Therefore $\beta(1) \in \Z$ must be an integer $n$ dividing $2^{2k}-1$.
    In particular, $n$ must be odd.
    So $\coker(\beta)\simeq \Z/n\Z$ is finite of odd order.
    \end{enumerate}
\end{proof}

\begin{proof}[Proof of Proposition \ref{proposition: c_Gamma not in subgroup generated by phiGamma}]
    By Proposition \ref{proposition: extension cDGamma not zero}, $c_{D, \Gamma}\neq0$.
    Hence it suffices to show that $\varphi_{\Gamma,m}=0$ for each $m\in M_D$.
    Let $m\in M_D$ be an element.
    If $m=0$, then the first map in \eqref{equation: definition phi_m Gamma case} is induced by the zero map $\Z\rightarrow M_D$, so evidently $\varphi_{\Gamma,m}=0$ in this case.
    So we may assume $m\neq 0$.
    Let $\mathrm{res}_m\colon \HH^2(\Gamma, \Z) \rightarrow \HH^2(\Gamma_m, \Z)$ be the restriction map.
    We first claim that $\varphi_{\Gamma,m} \circ \mathrm{res}_m = 0$.
    To prove this, note that $\varphi_{\Gamma, m}\circ \mathrm{res}_m$ is $\HH^2(\Gamma,-)$ applied to the composition of $\Gamma$-module maps
\begin{align}\label{equation: proof phim resm is zero}
   \Z\rightarrow \Z[\Gamma/\Gamma_m]\rightarrow M_D, 
\end{align}
where $\Z[\Gamma/\Gamma_m]$ is the permutation module on the $\Gamma$-set $\Gamma/\Gamma_m$, $\Z\rightarrow \Z[\Gamma/\Gamma_m]$ sends $1$ to $\sum g\cdot \Gamma_m$ (where $g$ ranges over coset representatives in $\Gamma/\Gamma_m$), and $\Z[\Gamma/\Gamma_m]\rightarrow M_D$ sends $g\cdot \Gamma_m$ to $g\cdot m$.
So the composite \eqref{equation: proof phim resm is zero} sends $1$ to $m'' = \sum_{m'\in \Gamma \cdot m} m'$.
Since $m''$ is an $\Sp(M_D)$-invariant element of $M_D$ and since there are no such nonzero elements, we must have $m''=0$ and the composite \eqref{equation: proof phim resm is zero} must be zero. 
Therefore the composite $\varphi_m \circ \mathrm{res}_m$ is also zero, as claimed.

We conclude that $\varphi_{\Gamma,m}\colon \HH^2(\Gamma_m,\Z)\rightarrow \HH^2(\Gamma, M_D)$ factors through the cokernel of $\mathrm{res}_m\colon \HH^2(\Gamma, \Z)\rightarrow \HH^2(\Gamma_m,\Z)$.
Since $M_D$ is killed by $2$, the target of $\varphi_{\Gamma,m}$ is killed by $2$.
By Proposition \ref{proposition: stabilizer m in Gamma has odd abelianization and cokernel on H2 odd order}, the cokernel of $\mathrm{res}_m$ has odd order.
Combining the last two sentences shows that $\varphi_{\Gamma,m}=0$.
\end{proof}

\begin{proof}[Proof of Theorem \ref{theorem: intro negative result}]
Combine Propositions \ref{proposition: reduce etale negligible to group cohomology calculation} and \ref{proposition: c_Gamma not in subgroup generated by phiGamma}.    
\end{proof}

\appendix

\section{Group theory computations}\label{subsec: group theory computations}

In this appendix, we finish the proof of Proposition \ref{proposition: abelianization torsion-free congruence subgroup has odd order} by computing the abelianization of an explicit $2$-adic congruence subgroup.
Let $n, k\geq 0$ be integers and let $g = n+k$.
Consider the type $D = (1,\dots, 1, 2, \dots, 2)$, where $1$ occurs $n$ times and $2$ occurs $k$ times.
Consider the profinite groups $\pargroup{g}{D}(\Z_2)$ and $\Gamma = \ker(\pargroup{g}{D}(\Z_2)\rightarrow \Sp(M_D))$. 
(We deviate from our standard notation that $\Gamma$ is a congruence subgroup of $\pargroup{g}{D}(\Z)$.)

\begin{theorem}\label{theorem: abelianization level D subgroup over 2-adics is trivial}
    If $n\geq 3$ and $k\geq 2$, then $\Gamma^{\mathrm{ab}}$ is trivial.
\end{theorem}

The proof of this theorem is given at the end of this section, after some preliminary explicit lemmas. In the words of Mumford \cite[p.\,202]{Mumford-tataI}, there is nothing very difficult in any of these. 
In the standard basis $e_1, \dots, e_g, f_1, \dots, f_g$ of $\Lambda_D\otimes \Z_2 = \Z_2^{2g}$, we have
\begin{align*}
\pargroup{g}{D}(\Z_2) &=\{\gamma\in \GL_{2g}(\Z_2) \colon \gamma^t J_D \gamma = J_D\} \\
&=
\left\{ \begin{pmatrix}
    X & Y \\ Z & W 
\end{pmatrix} \in \GL_{2g}(\Z_2)
\colon X^t DZ\text{ and } Y^tDW \text{ symmetric}, X^tDW -Z^tDY = D 
\right\}.
\end{align*}
We will break each $g\times g$-block $X,Y,Z,W$ into subblocks, and we will write $X = \begin{psmallmatrix}X_{11} & X_{12} \\ X_{21} & X_{22}  \end{psmallmatrix}$, where $X_{11}$ has size $n\times n$, $X_{12}$ has size $n\times k$, $X_{21}$ has size $k\times n$ and $X_{22}$ has size $n\times n$.
Similarly we have $Y_{ij}, Z_{ij}$ and $W_{ij}$ for $1\leq i,j\leq 2$.
The condition $J_DgJ_D^{-1} = g^{-t} \in \GL_{2g}(\Z_2)$ implies that $DXD^{-1}, DYD^{-1}, DZD^{-1}, DWD^{-1}$ have entries in $\Z_2$. 
This implies that all the entries of $X_{12}, Y_{12}, Z_{12},W_{12}$ are in $2\Z_2$.

We record two automorphisms of $\pargroup{g}{D}(\Z_2)$. 
The first one is $\gamma\mapsto (\gamma^*)^{-1} = (J_D^{-1} \gamma^t J_D)^{-1}$, and the second one is conjugation by $h = \begin{psmallmatrix}
    0 & I \\ -I & 0 
\end{psmallmatrix}
\in \pargroup{g}{D}(\Z_2)$. 
Explicitly, we have
\begin{align}\label{equation: involutions on paramodular group}
    \begin{pmatrix}
    X & Y \\ Z & W 
\end{pmatrix}^* = \begin{pmatrix}
    D^{-1}W^tD & -D^{-1}Y^tD \\ -D^{-1}Z^t D & D^{-1}X^tD 
\end{pmatrix}
,\quad
h\begin{pmatrix}
    X & Y \\ Z & W 
\end{pmatrix} h^{-1}
=\begin{pmatrix}
    W & -Z \\ -Y & X 
\end{pmatrix}.
\end{align}
Recall that $M_D = \Lambda_D^{\vee}/\Lambda_D$ and that the reduction map $\pargroup{g}{D}(\Z_2)\rightarrow \Sp(M_D) \simeq \Sp_{2k}(\F_2)$ is surjective. 
In the coordinates of $M_D$ fixed in Section \ref{subsec: paramodular groups}, this reduction map is given by 
\[
\begin{pmatrix}
    X & Y \\ Z & W 
\end{pmatrix} 
\mapsto 
\begin{pmatrix}
    X_{22} & Y_{22} \\ Z_{22} & W_{22} 
\end{pmatrix} 
\bmod 2.
\]
The restriction of the symplectic form to the submodule $\Lambda_1$ with basis $\{e_1, \dots, e_n, f_1, \dots, f_n\}$ has Gram matrix $J_{D_1}$, where $D_1 = (1, \dots, 1)$ has length $n$.
Similarly, the restriction of the form to the submodule $\Lambda_2$ spanned by $\{e_{n+1}, \dots, e_g, f_{n+1}, \dots, f_g\}$ has Gram matrix $J_{D_2}$, where $D_2 = (2, \dots, 2)$ has length $k$.
The subgroup of elements of $\pargroup{g}{D}(\Z_2)$ that preserve $\Lambda_1$ and $\Lambda_2$ is isomorphic to $\Sp(\Lambda_1)\times \Sp(\Lambda_2)\simeq \Sp_{2n}(\Z_2)\times \Sp_{2k}(\Z_2)$.
This defines an embedding $i\colon \Sp_{2n}(\Z_2) \times \Sp_{2k}(\Z_2)\hookrightarrow \pargroup{g}{D}(\Z_2)$ whose image equals those matrices for which $X_{12}, X_{21}, Y_{12},Y_{21}, Z_{12}, Z_{21}, W_{12},W_{21}$ are all zero.
Let $\Sp_{2k}(\Z_2,2) = \ker(\Sp_{2k}(\Z_2)\rightarrow \Sp_{2k}(\F_2))$. 
The restriction of $i$ to $\Gamma$ is an embedding $\Sp_{2n}(\Z_2)\times \Sp_{2k}(\Z_2, 2)\hookrightarrow \Gamma$. 
\begin{lemma}\label{lemma: abelianization level subgroups symplectic group}
    If $n\geq 3$, $\Sp_{2n}(\Z_2)^{\mathrm{ab}}$ is trivial.
    If $k\geq 2$, the commutator subgroup of $\Sp_{2k}(\Z_2,2)$ equals 
    \[
    \Sp_{2k}(\Z_2,4,8):=
    \{ A = (a_{ij}) \in \Sp_{2k}(\Z_2)\colon A\equiv I \bmod 4 \text{  and }a_{i,i+k}\equiv a_{i+k,i}\equiv 0 \bmod{8} \text{ for all }i=1,\dots, k\}.
    \]
\end{lemma}
\begin{proof}
    The first claim is a special case of Lemma \ref{lemma: abelianizations symplectic groups}.
    The second part follows from the isomorphism $\widehat{\Sp_{2g}(\Z, 2)}\simeq \Sp_{2g}(\Z_2, 2)\times \prod_{p\geq 3} \Sp_{2g}(\Z_p)$, Lemmas \ref{lemma: completion of abelianization} and \ref{lemma: abelianizations symplectic groups}, and \cite[Proposition 10.1]{sato-abelianizationlevelD}. 
\end{proof}

Let $L \subset \Gamma$ be the subgroup of elements with $Y=Z=0$.
Let 
\begin{align}\label{equation: explicit description Gamma1prime}
L' = \{X\in \GL_g(\Z_2) \colon X_{12} \equiv 0 \bmod 2, X_{22} \equiv 1 \bmod 2\}.
\end{align}
Then the assignment $X\mapsto \begin{psmallmatrix} X & 0 \\ 0 &D^{-1}X^{-t}D\end{psmallmatrix}$ defines an isomorphism of groups $\alpha\colon L' \xrightarrow{\sim}L$.
For a ring $R$ and integers $a,b\geq 1$, let $\mathrm{Mat}_{a,b}(R)$ be the $R$-module of $a\times b$-matrices.
\begin{lemma}\label{lemma: block diagonal elements are in the commutator subgroup}
    If $n\geq 3$ and $k\geq 2$, then $L\subset \Gamma^{\mathrm{der}}$.
\end{lemma}
\begin{proof}
    By considering block diagonal elements of $\Sp_{2n}(\Z_2)$, we see that every element of $L$ of the form 
    \begin{align}\label{equation: first commutator lemma, first type element}
    \alpha\left(\begin{pmatrix}
        X_{11} & 0 \\ 0 & I 
    \end{pmatrix}\right), \quad X_{11}\in \GL_n(\Z_2)
    \end{align}
    lies in the image of $i\colon \Sp_{2n}(\Z_2)\rightarrow \Gamma$. 
    By Lemma \ref{lemma: abelianization level subgroups symplectic group} and the assumption $n\geq 3$, every element in this image lies in $\Gamma^{\mathrm{der}}$, so every element of the form \eqref{equation: first commutator lemma, first type element} lies in $\Gamma^{\mathrm{der}}$.
    Similarly, every element of the form 
    \begin{align}\label{equation: first commutator lemma, second type element}
        \alpha\left( 
        \begin{pmatrix}
            I & 0 \\ 0 & X_{22} 
        \end{pmatrix}
        \right),\quad 
        X_{22}\in \ker(\GL_k(\Z_2)\rightarrow \GL_k(\Z/4\Z))
    \end{align}
    lies in $\Gamma^{\mathrm{der}}$.
    Next, we consider strictly upper triangular elements.
    If $X\in \GL_n(\Z_2)$ and $Y\in 2\mathrm{Mat}_{n,k}(\Z_2)$, we calculate the commutator
    \[
    \left[
    \begin{pmatrix}X & 0 \\ 0 & I \end{pmatrix},
    \begin{pmatrix}I & Y \\ 0 & I \end{pmatrix}
    \right]
    =
    \begin{pmatrix}I & XY-Y \\ 0 & I \end{pmatrix}
    \]
    in $L'$.
    We claim that the $\Z$-span of the set $\{XY-Y\colon X\in \GL_n(\Z_2), Y\in 2\mathrm{Mat}_{n,k}(\Z_2) \}$ equals $2\mathrm{Mat}_{n,k}(\Z_2)$. 
    To prove this, we may assume $Y$ consists of a single column, in which case it can be explicitly checked by taking $X$ to be upper and lower triangular matrices (and using that $n\geq 2$).
    We conclude that every element of the form $\alpha\left( \begin{psmallmatrix}
        I & * \\ 0 & I 
    \end{psmallmatrix}\right) \in L$ lies in $L^{\mathrm{der}}$.
    By an identical argument, every element of the form $\alpha\left( \begin{psmallmatrix}
        I & 0 \\ * & I 
    \end{psmallmatrix}\right) \in L$ lies in $L^{\mathrm{der}}$.
    Now let $\alpha\left(\begin{psmallmatrix}X_{11} & X_{12} \\ X_{21} & X_{22}  \end{psmallmatrix}\right)\in L$ be a general element.
    Then the condition $X_{12} \equiv 0 \bmod 2$ implies that $X_{11}$ is invertible modulo $2$, so $X_{11}\in \GL_n(\Z_2)$. 
    Therefore, left multiplying by $\alpha\left(
    \begin{psmallmatrix}
        I & 0 \\ -X_{21} & I
    \end{psmallmatrix}
    \begin{psmallmatrix}
        X_{11}^{-1} & 0 \\ 0 & I
    \end{psmallmatrix}
    \right)\in \Gamma^{\mathrm{der}}$ and right multiplying by $\alpha\left(
    \begin{psmallmatrix}
        I & -X_{11}^{-1} X_{12} \\ 0 & I
    \end{psmallmatrix}
    \right)\in \Gamma^{\mathrm{der}}$ shows that
    \begin{align}\label{equation: first commutator lemma, reduction to lower block matrix}
    \alpha\left(\begin{pmatrix}X_{11} & X_{12} \\ X_{21} & X_{22}  \end{pmatrix} \right)\equiv
    \alpha\left(\begin{pmatrix}I & 0 \\ 0 & X_{22}-X_{21}X_{11}^{-1}X_{12}  \end{pmatrix}\right) 
    \mod \Gamma^{\mathrm{der}}.
    \end{align}
    To prove the lemma, it therefore suffices to show that every element of $L$ the form $\alpha\left( \begin{psmallmatrix}
        I & 0 \\ 0 & * 
    \end{psmallmatrix}\right)$ lies in $\Gamma^{\mathrm{der}}$.
    We calculate using \eqref{equation: first commutator lemma, reduction to lower block matrix} that for all $X\in \Mat_{n,k}(\Z_2)$ and $Y\in \Mat_{k,n}(\Z_2)$:
    \[
    \alpha\left(\left[ 
    \begin{pmatrix}I & 2X \\ 0 & I \end{pmatrix},
    \begin{pmatrix}I & 0 \\ Y & I \end{pmatrix}
    \right]\right)=
    \alpha\left(\begin{pmatrix}* & -4XYX \\ 2YXY & I-2YX \end{pmatrix}\right) \equiv 
    \alpha\left(\begin{pmatrix}I & 0 \\ 0 & I-2YX+ 8(\cdots) \end{pmatrix} \right)
    \mod \Gamma^{\mathrm{der}}
    ,
    \]
    where $(\cdots)$ denotes an element of $\Mat_{k,k}(\Z_2)$.
    By combining this identity with the elements of the form \eqref{equation: first commutator lemma, second type element}, it suffices to prove that the set $\{I-2YX\colon X\in \Mat_{n,k}(\Z_2), \,Y\in\Mat_{k,n}(\Z_2)\}$ generates a subgroup of $\GL_k(\Z_2)$ that surjects onto $\ker(\GL_k(\Z/4\Z)\rightarrow \GL_k(\Z/2\Z))$.  
    This is true, as can be seen by taking $X$ and $Y$ to be $\Z_2$-multiples of elementary matrices.
\end{proof}

Let $U\subset \Gamma$ be the subgroup of strictly upper triangular elements, i.e., those elements with $X=W=I$ and $Z=0$. 
Define the $\Z_2$-module
\begin{align*}
U' 
=\{ Y\in \Mat_{g,g}(\Z_2)\colon Y_{11}, Y_{22} \text{ symmetric, }Y_{12} =2Y_{21}^t \text{ and }Y_{22} \equiv 0 \bmod 2
\}
\end{align*}
Then the assignment $Y\mapsto \begin{psmallmatrix}I & Y \\ 0 & I \end{psmallmatrix}$ is an isomorphism of abelian groups $\beta\colon U' \rightarrow U$.
Let $U^{\mathrm{opp}}\subset \Gamma$ be the subgroup of strictly lower triangular matrices, i.e., those with $X=W=I$ and $Y=0$.

\begin{lemma}\label{lemma: unipotent elements are in the commutator subgroup}
    If $n\geq 3$ and $k\geq 2$, then $U\subset \Gamma^{\mathrm{der}}$ and $U^{\mathrm{opp}}\subset \Gamma^{\mathrm{der}}$.
\end{lemma}
\begin{proof}
    Denote the submodule of symmetric matrices of $\Mat_{n,n}(\Z_2)$ by $\Sym_n(\Z_2)$. 
    By considering strictly upper triangular elements in $\Sp_{2n}(\Z_2)$ and $\Sp_{2k}(\Z_2)$ and using Lemma \ref{lemma: abelianization level subgroups symplectic group}, we see that every element of the form 
    \begin{align}\label{equation: second commutator lemma, first type element}
    \beta\left(
    \begin{pmatrix}
        Y_{11} & 0 \\
        0 & Y_{22}
    \end{pmatrix}
    \right),\quad
    Y_{11}\in \Sym_n(\Z_2), Y_{22}\in 8 \Sym_k(\Z_2)
    \end{align}
    lies in $\Gamma^{\mathrm{der}}$. 
    We now calculate commutators between elements of $L$ and $U$. 
    If $X\in L'$ and $Y\in U'$, then
    \begin{align}\label{equation: second commutator lemma, second type element}
    \left[\alpha(X), \beta(Y) \right]
    =
    \begin{pmatrix}I & XYW^{-1}-Y \\ 0 & I \end{pmatrix},
    \end{align}
    where $W = D^{-1}X^{-t} D$.
    We claim that every element of $U$ is a product of elements of the form \eqref{equation: second commutator lemma, first type element} and \eqref{equation: first commutator lemma, second type element}.
    To prove this claim, note that
    \[
XYW^{-1} = 
    \begin{pmatrix}X_{11} & X_{12} \\ X_{21} & X_{22}  \end{pmatrix}
    \begin{pmatrix}Y_{11} & Y_{12} \\ Y_{21} & Y_{22}  \end{pmatrix}
    \begin{pmatrix}X^t_{11} & 2X^t_{21} \\ \frac{1}{2}X^t_{12} & X^t_{22}  \end{pmatrix}.
    \]
    To consider off-diagonal blocks, we compute that if $X_{12} = 0$, $X_{21}=0$, $X_{22}=0$, $Y_{11}=0$ and $Y_{22}=0$, then 
    \begin{align}\label{equation: second commutator lemma, third type element}
    XYW^{-1}-Y = 
    \begin{pmatrix}
        0 & X_{11}Y_{12}-Y_{12} \\
        Y_{21}X_{11}^t-Y_{21} & 0
    \end{pmatrix}.
    \end{align}
    By an argument similar to the proof of Lemma \ref{lemma: block diagonal elements are in the commutator subgroup}, the $\Z$-span of $\{YX^t - Y\colon X\in \GL_n(\Z_2), \, Y\in \Mat_{k,n}(\Z_2)\}$ equals $\Mat_{k,n}(\Z_2)$.
    So every element of $U'$ of the form $\begin{psmallmatrix}0 & * \\ * & 0\end{psmallmatrix}$ is of the form \eqref{equation: second commutator lemma, third type element}.
    To consider the lower-diagonal block, we compute that if $X_{11} = I, X_{12}=0$, $X_{22} = I$, $Y_{12}=0$, $Y_{21} =0$ and $Y_{22}=0$, then 
    \begin{align}\label{equation: second commutator lemma, fourth type element}
    XYW^{-1}-Y = 
    \begin{pmatrix}
        * & * \\
        * & 2X_{21}Y_{11}X_{21}^t
    \end{pmatrix}.
    \end{align}
    By considering elementary matrices, we find that the $\Z$-span of the set $\{2XYX^t\colon X\in \Mat_{k,n}(\Z_2), \, Y\in \Sym_n(\Z_2)  \}$ equals $2\Sym_k(\Z_2)$. 
    Combining all elements of the form \eqref{equation: second commutator lemma, first type element}, \eqref{equation: second commutator lemma, third type element} and \eqref{equation: second commutator lemma, fourth type element} proves the claim, hence proves that $U\subset \Gamma^{\mathrm{der}}$.
    Applying the second automorphism displayed in \eqref{equation: involutions on paramodular group} to this containment also proves that $U^{\mathrm{opp}}\subset \Gamma^{\mathrm{der}}$.
\end{proof}

\begin{lemma}\label{lemma: level D subgroup generated by L, U and Sp's}
    Let $H$ be the closure of the subgroup of $\Gamma$ generated by $L, U$ and $U^{\mathrm{opp}}$.
    Then $H = \Gamma$.
\end{lemma}
\begin{proof}
    The proof is inspired by the methods of \cite{Rosati-generatingparamodulargroups}.
    Recall that we have an embedding $i\colon \Sp_{2n}(\Z_2) \times \Sp_{2k}(\Z_2, 2)\hookrightarrow \Gamma$ which we use to view $\Sp_{2n}(\Z_2)\times \Sp_{2k}(\Z_2, 2)$ as a subgroup of $\Gamma$.
    
    \underline{Claim 1}: $\Sp_{2n}(\Z_2)\subset H$ and $\Sp_{2k}(\Z_2, 2)\subset H$.
    Indeed, \cite[Corollary 12.5]{BassMilnorSerre} shows that $\Sp_{2n}(\Z)$ is generated by elements in $U$ and $U^{\mathrm{opp}}$, and \cite[p.\,207, Proposition A3]{Mumford-tataI} shows $\Sp_{2k}(\Z,2)$ is generated by elements in $U$, $U^{\mathrm{opp}}$ and $L$. 
    Since the inclusions $\Sp_{2n}(\Z)\subset\Sp_{2n}(\Z_2)$ and $\Sp_{2k}(\Z,2) \subset \Sp_{2k}(\Z_2, 2)$ are dense, the claim is proven.
    
    We will prove the lemma by induction on $n$.
    If $n=0$, then $\Gamma = \Sp_{2k}(\Z_2,2)$, so the lemma is true by Claim 1.
    Therefore we may suppose $n\geq 1$.
    Let $\gamma \in \Gamma$ be an element and let $u = \begin{psmallmatrix}v \\ w\end{psmallmatrix}$ be the first column of $\gamma$, where $v,w\in \Z_2^{g}$.
    We write $v = \begin{psmallmatrix}v_1 \\ v_2\end{psmallmatrix}$ and $w = \begin{psmallmatrix}w_1 \\ w_2\end{psmallmatrix}$, where $v_1, w_1 \in \Z_2^n$ and $v_2,w_2\in \Z_2^{k}$.

    \underline{Claim 2}: There exists an element $\delta_1 \in \Sp_{2n}(\Z_2)$ such that $\delta_1\cdot \begin{psmallmatrix}v_1 \\ w_1\end{psmallmatrix}=\begin{psmallmatrix} 1 & 0 & \cdots & 0 \end{psmallmatrix}^t$
    Indeed, since $\Sp_{2n}(\Z_2)$ acts transitively on primitive vectors (in other words, elements of $\Z_2^{2n} \setminus (2\Z_2^{2n})$) it suffices to prove that $\begin{psmallmatrix}v_1 \\ w_1\end{psmallmatrix}$ is primitive.
    Examining the $(g+1)$-th row of $\gamma^*$ (using the first formula of \eqref{equation: involutions on paramodular group}) shows that $u' = \begin{pmatrix}
        -w_1^t & -2w_2^t & v_1^t & 2v_2^t
    \end{pmatrix}$ appears as a row of an element of $\pargroup{g}{D}(\Z_2) \subset \GL_{2g}(\Z_2)$, so $u'$ must be primitive. 
    This implies that $\begin{psmallmatrix}v_1 \\ w_1\end{psmallmatrix}$ is primitive, proving the claim.

    \underline{Claim 3}: Suppose $v_1 = \begin{psmallmatrix} 1 & 0 & \cdots & 0 \end{psmallmatrix}^t$. Then there exists an element $\delta_2 \in U^{\mathrm{opp}}$ such that $\delta_2 \cdot u$ has first $n$ coordinates equal to $v_1$ and last $k$ coordinates equal to zero.
    Indeed, let $Z_{21}\in \Mat_{k,n}(\Z_2)$ be the matrix with first column equal to $-w_2$ and all other columns equal to zero, and let $Z = \left(\begin{smallmatrix}
        0 & 2Z_{21}^t \\ Z_{21} & 0
    \end{smallmatrix}\right)$.
    Then $\delta_2 = \begin{psmallmatrix}
        I & 0 \\ Z & I
    \end{psmallmatrix}$ satisfies the conclusion of the claim.

    \underline{Claim 4}: There exists an element $\delta \in H$ such that $\delta \cdot u = \begin{psmallmatrix} 1 & 0 &\cdots & 0 \end{psmallmatrix}^t$.
    Indeed, by Claims 1 and 2, there exists a $\delta_1\in H$ such that the first $n$ coordinates of $\delta_1 \cdot u$ equal $1, 0, \dots, 0$.
    By Claim 3, there exists a $\delta_2\in H$ such that the first $n$ coordinates of $\delta_2\cdot (\delta_1 \cdot u)$ equal $1, 0, \dots, 0$ and the last $k$ coordinates equal $0, \dots, 0$.
    Applying Claims 1 and 2 again, there is an element $\delta_3\in \Sp_{2n}(\Z_2)\subset H$ such that $u' = \delta_3\cdot (\delta_2\delta_1 \cdot u)$ has first $n$ coordinates $1,0,\dots,0$ and last $g$ coordinates equal to $0$.
    Let $v' \in \Z_2^g$ be the vector formed of the first $g$ coordinates of $u'$.
    The explicit description \eqref{equation: explicit description Gamma1prime} shows that there exists an element $X\in L'$ such that $X\cdot v' = \begin{psmallmatrix} 1 & 0 & \cdots & 0 \end{psmallmatrix}^t$.
    Then the element $\alpha(X)\in L$ satisfies $\alpha(X) \cdot u' = \begin{psmallmatrix} 1 & 0 & \cdots & 0 \end{psmallmatrix}^t$.
    Taking $\delta = \alpha(X) \delta_3\delta_2\delta_1 \in H$ proves the claim.

    Using Claim 4, it suffices to prove that every element $\gamma \in \Gamma$ whose first column equals $\begin{psmallmatrix} 1 & 0 & \cdots & 0 \end{psmallmatrix}^t$ lies in $H$.
    Let $\gamma$ be such an element.
    Recall that $e_1, \dots, e_g, f_1, \dots, f_g$ denotes the standard basis of $\Lambda_D \otimes \Z_2 = \Z_2^{2g}$.
    By assumption, $\gamma\cdot e_1 = e_1$.
    Therefore $\gamma$ also preserves the submodule $\langle e_1\rangle^{\perp} = \langle e_1, \dots, e_g, f_2, \dots, f_g\rangle$.
    In terms of matrices, this means that $\gamma$ is of the form 
    \begin{align*}
    \left(
\begin{array}{c;{2pt/2pt} c   | c;{2pt/2pt} c }
1 & *    & * & * \\
\hdashline[2pt/2pt]
0  & X_1& * & Y_1 \\
\hline
0  & 0   & *  & 0    \\
\hdashline[2pt/2pt]
 0 &   Z_1  &  *  &  W_1  \\
\end{array}   
\right),
\end{align*}
where we have divided $\gamma$ into $g\times g$-blocks using the solid lines, and where each $g\times g$-block is further divided into blocks of size $1\times 1$, $1\times (g-1)$, $(g-1)\times 1$ and $(g-1)\times (g-1)$ by the dashed lines.
The quotient $N=\langle e_1\rangle^{\perp}/\langle e_1\rangle$ has $\Z_2$-basis given by the images of $\{e_2, \dots, e_g, f_2, \dots, f_g\}$, and the symplectic form on $\Lambda_D \otimes \Z_2$ induces a symplectic form on $N$ of type $D'= (1, \dots, 1, 2, \dots, 2)$, where $1$ occurs $n-1$ times and $2$ occurs $k$ times.
The map $\gamma$ induces a map $\gamma'\in \GL(N)$ that preserves the induced symplectic form.
In the above basis of $N$, $\gamma'=\begin{psmallmatrix} X_1 & Y_1 \\ Z_1 & W_1 \end{psmallmatrix}$.
Writing $\Gamma' = \ker(\pargroup{g}{D'}(\Z_2)\rightarrow \Sp(M_{D'}))$, we see that $\gamma' \in \Gamma'$.
Let $\gamma''\in \Gamma$ be the element that acts on $\langle e_2, \dots, e_g, f_2, \dots, f_g\rangle$ via $\gamma'$ and acts trivially on $\langle e_1, f_1\rangle$.
By the induction hypothesis, $\gamma'$ lies in the analogue $H'\subset \Gamma'$ of the subgroup $H\subset \Gamma$.
Therefore $\gamma'' \in H$, and the matrix $(\gamma'')^{-1}\gamma$ is of the form 
\begin{align*}
    \left(
\begin{array}{c;{2pt/2pt} c   | c;{2pt/2pt} c }
1 & *    & * & * \\
\hdashline[2pt/2pt]
0  & I & * & 0 \\
\hline
0  & 0   & *  & 0    \\
\hdashline[2pt/2pt]
 0 &   0  &  *  &  I  \\
\end{array}   
\right).
\end{align*}
This matrix is block upper triangular, so lies in the subgroup generated by $L$ and $U$, which both lie in $H$.
This shows that $(\gamma'')^{-1}\gamma\in H$, so $\gamma\in H$, proving the lemma.
\end{proof}

\begin{proof}[Proof of Theorem \ref{theorem: abelianization level D subgroup over 2-adics is trivial}]
Combine Lemmas \ref{lemma: block diagonal elements are in the commutator subgroup}, \ref{lemma: unipotent elements are in the commutator subgroup} and \ref{lemma: level D subgroup generated by L, U and Sp's}.
\end{proof}


\begin{thebibliography}{BCRR18}

\bibitem[ACM24]{AchterCasalainaMartin-puttingpinPrym}
Jeff Achter and Sebastian Casalaina-Martin.
\newblock Putting the p back in {P}rym.
\newblock Arxiv Preprint, available at
  \url{https://arxiv.org/abs/2312.13263v2}, 2024.

\bibitem[Ale02]{alexeev-completemodulisemiabelian}
Valery Alexeev.
\newblock Complete moduli in the presence of semiabelian group action.
\newblock {\em Ann. of Math. (2)}, 155(3):611--708, 2002.

\bibitem[BCRR18]{Bensonetal-cohomologysymplecticgroups}
Dave Benson, Caterina Campagnolo, Andrew Ranicki, and Carmen Rovi.
\newblock Cohomology of symplectic groups and {M}eyer's signature theorem.
\newblock {\em Algebr. Geom. Topol.}, 18(7):4069--4091, 2018.

\bibitem[BL04]{BirkenhakeLange-AVs}
Christina Birkenhake and Herbert Lange.
\newblock {\em Complex abelian varieties}, volume 302 of {\em Grundlehren der
  mathematischen Wissenschaften [Fundamental Principles of Mathematical
  Sciences]}.
\newblock Springer-Verlag, Berlin, second edition, 2004.

\bibitem[BLR90]{BLR-NeronModels}
Siegfried Bosch, Werner L\"{u}tkebohmert, and Michel Raynaud.
\newblock {\em N\'{e}ron models}, volume~21 of {\em Ergebnisse der Mathematik
  und ihrer Grenzgebiete (3) [Results in Mathematics and Related Areas (3)]}.
\newblock Springer-Verlag, Berlin, 1990.

\bibitem[BMS67]{BassMilnorSerre}
H.~Bass, J.~Milnor, and J.-P. Serre.
\newblock Solution of the congruence subgroup problem for {${\rm
  SL}_{n}\,(n\geq 3)$} and {${\rm Sp}_{2n}\,(n\geq 2)$}.
\newblock {\em Inst. Hautes \'{E}tudes Sci. Publ. Math.}, (33):59--137, 1967.

\bibitem[Bor74]{Borel-stablerealcohomologyarithmeticgroupsI}
Armand Borel.
\newblock Stable real cohomology of arithmetic groups.
\newblock {\em Ann. Sci. \'{E}cole Norm. Sup. (4)}, 7:235--272 (1975), 1974.

\bibitem[Bou68]{Bourbaki-groupesalgebresdelieChapIV-VI}
N.~Bourbaki.
\newblock {\em \'{E}l\'{e}ments de math\'{e}matique. {F}asc. {XXXIV}. {G}roupes
  et alg\`ebres de {L}ie. {C}hapitre {IV}: {G}roupes de {C}oxeter et syst\`emes
  de {T}its. {C}hapitre {V}: {G}roupes engendr\'{e}s par des r\'{e}flexions.
  {C}hapitre {VI}: syst\`emes de racines}.
\newblock Actualit\'{e}s Scientifiques et Industrielles [Current Scientific and
  Industrial Topics], No. 1337. Hermann, Paris, 1968.

\bibitem[Bra93]{Brasch-liftinglevelD}
Hans-J\"{u}rgen Brasch.
\newblock Lifting level {$D$}-structures of abelian varieties.
\newblock {\em Arch. Math. (Basel)}, 60(6):553--562, 1993.

\bibitem[Bro82]{Brown-cohomologyofgroups}
Kenneth~S. Brown.
\newblock {\em Cohomology of groups}, volume~87 of {\em Graduate Texts in
  Mathematics}.
\newblock Springer-Verlag, New York-Berlin, 1982.

\bibitem[Chi24]{Chidambaram-modpgaloisrepsnotarisingfromAVs}
Shiva Chidambaram.
\newblock {${\rm Mod}\text{-}p$} {G}alois representations not arising from
  abelian varieties.
\newblock {\em J. Number Theory}, 259:219--237, 2024.

\bibitem[dJ93]{deJong-modulispacepolarized}
A.~J. de~Jong.
\newblock The moduli spaces of polarized abelian varieties.
\newblock {\em Math. Ann.}, 295(3):485--503, 1993.

\bibitem[EvdGM]{AVBOOK}
Bas Edixhoven, Gerard van der Geer, and Ben Moonen.
\newblock Abelian varieties book project.
\newblock Available at
  \url{http://van-der-geer.nl/~gerard/AV.pdf}, November 2025.

\bibitem[FC90]{FaltingsChai-degenerations}
Gerd Faltings and Ching-Li Chai.
\newblock {\em Degeneration of abelian varieties}, volume~22 of {\em Ergebnisse
  der Mathematik und ihrer Grenzgebiete (3) [Results in Mathematics and Related
  Areas (3)]}.
\newblock Springer-Verlag, Berlin, 1990.
\newblock With an appendix by David Mumford.

\bibitem[Gir71]{Giraud-cohomologienonabelienne}
Jean Giraud.
\newblock {\em Cohomologie non ab\'{e}lienne}.
\newblock Die Grundlehren der mathematischen Wissenschaften, Band 179.
  Springer-Verlag, Berlin-New York, 1971.

\bibitem[GMS03]{garibaldimerkurjevserre-cohomologicalinvariants}
Skip Garibaldi, Alexander Merkurjev, and Jean-Pierre Serre.
\newblock {\em Cohomological invariants in {G}alois cohomology}, volume~28 of
  {\em University Lecture Series}.
\newblock American Mathematical Society, Providence, RI, 2003.

\bibitem[Gri73]{Griess-automorphismsofextraspecialgroups}
Robert~L. Griess, Jr.
\newblock Automorphisms of extra special groups and nonvanishing degree {$2$}
  cohomology.
\newblock {\em Pacific J. Math.}, 48:403--422, 1973.

\bibitem[ILF97]{embeddingproblem}
V.~V. Ishkhanov, B.~B. Lure, and D.~K. Faddeev.
\newblock {\em The embedding problem in {G}alois theory}, volume 165 of {\em
  Translations of Mathematical Monographs}.
\newblock American Mathematical Society, Providence, RI, 1997.
\newblock Translated from the 1990 Russian original by N. B. Lebedinskaya.

\bibitem[Lag24]{Laga-ADEpaper}
Jef Laga.
\newblock Graded {L}ie algebras, compactified {J}acobians and arithmetic
  statistics.
\newblock {Journal of the European Mathematical Society},
  \url{https://doi.org/10.4171/jems/1526}, 2024.

\bibitem[LSTX17]{LandesmanSwaminathanetal-liftingsubgroups}
Aaron Landesman, Ashvin~A. Swaminathan, James Tao, and Yujie Xu.
\newblock Lifting subgroups of symplectic groups over {$\Bbb{Z}/\ell\Bbb{Z}$}.
\newblock {\em Res. Number Theory}, 3:Paper No. 14, 12, 2017.

\bibitem[Lur01]{Lurie-minisculereps}
Jacob Lurie.
\newblock On simply laced {L}ie algebras and their minuscule representations.
\newblock {\em Comment. Math. Helv.}, 76(3):515--575, 2001.

\bibitem[MB85]{moretbailly-pinceaux}
Laurent Moret-Bailly.
\newblock Pinceaux de vari\'{e}t\'{e}s ab\'{e}liennes.
\newblock {\em Ast\'{e}risque}, (129):266, 1985.

\bibitem[MFK94]{MumfordFogartyKirwan-GIT}
D.~Mumford, J.~Fogarty, and F.~Kirwan.
\newblock {\em Geometric invariant theory}, volume~34 of {\em Ergebnisse der
  Mathematik und ihrer Grenzgebiete (2) [Results in Mathematics and Related
  Areas (2)]}.
\newblock Springer-Verlag, Berlin, third edition, 1994.

\bibitem[MS23]{MorganSmith-CasselsTateFiniteGalois}
Adam Morgan and Alexander Smith.
\newblock The {C}assels-{T}ate pairing for finite {G}alois modules.
\newblock Arxiv Preprint, available at
  \url{https://arxiv.org/abs/2103.08530v3}, 2023.

\bibitem[MS25]{MerkurjevScavia-negligiblecohomology}
Alexander Merkurjev and Federico Scavia.
\newblock Galois representations modulo $p$ that do not lift modulo $p^2$.
\newblock Journal of the American Mathematical Society:
  \url{https://doi.org/10.1090/jams/1059}, 2025.

\bibitem[Mum66]{Mumford-equationsdefiningabelianvarieties}
D.~Mumford.
\newblock On the equations defining abelian varieties. {I}.
\newblock {\em Invent. Math.}, 1:287--354, 1966.

\bibitem[Mum67]{Mumford-equationsdefiningAVs-part2}
D.~Mumford.
\newblock On the equations defining abelian varieties. {II}.
\newblock {\em Invent. Math.}, 3:75--135, 1967.

\bibitem[Mum83]{Mumford-tataI}
David Mumford.
\newblock {\em Tata lectures on theta. {I}}, volume~28 of {\em Progress in
  Mathematics}.
\newblock Birkh\"{a}user Boston, Inc., Boston, MA, 1983.
\newblock With the assistance of C. Musili, M. Nori, E. Previato and M.
  Stillman.

\bibitem[Mum91]{Mumford-tataIII}
David Mumford.
\newblock {\em Tata lectures on theta. {III}}, volume~97 of {\em Progress in
  Mathematics}.
\newblock Birkh\"{a}user Boston, Inc., Boston, MA, 1991.
\newblock With the collaboration of Madhav Nori and Peter Norman.

\bibitem[Ols08]{olsson-compactifymoduli}
Martin~C. Olsson.
\newblock {\em Compactifying moduli spaces for abelian varieties}, volume 1958
  of {\em Lecture Notes in Mathematics}.
\newblock Springer-Verlag, Berlin, 2008.

\bibitem[Ols12]{olsson-survey}
Martin Olsson.
\newblock Compactifications of moduli of abelian varieties: an introduction.
\newblock In {\em Current developments in algebraic geometry}, volume~59 of
  {\em Math. Sci. Res. Inst. Publ.}, pages 295--348. Cambridge Univ. Press,
  Cambridge, 2012.

\bibitem[Pol02]{Polishchuk-analogueWeilrepresentation}
A.~Polishchuk.
\newblock Analogue of {W}eil representation for abelian schemes.
\newblock {\em J. Reine Angew. Math.}, 543:1--37, 2002.

\bibitem[Pol03]{Polishchuk-abelianvarietiesthetafunctions}
Alexander Polishchuk.
\newblock {\em Abelian varieties, theta functions and the {F}ourier transform},
  volume 153 of {\em Cambridge Tracts in Mathematics}.
\newblock Cambridge University Press, Cambridge, 2003.

\bibitem[Poo17]{Poonen-rationalpointsonvarieties}
Bjorn Poonen.
\newblock {\em Rational points on varieties}, volume 186 of {\em Graduate
  Studies in Mathematics}.
\newblock American Mathematical Society, Providence, RI, 2017.

\bibitem[PR11]{PoonenRains-thetacharacteristics}
Bjorn Poonen and Eric Rains.
\newblock Self cup products and the theta characteristic torsor.
\newblock {\em Math. Res. Lett.}, 18(6):1305--1318, 2011.

\bibitem[PRR23]{PlatonovRapinchuk-secondeditionvolume1}
Vladimir Platonov, Andrei Rapinchuk, and Igor Rapinchuk.
\newblock {\em Algebraic groups and number theory. {V}ol. {I}}, volume 205 of
  {\em Cambridge Studies in Advanced Mathematics}.
\newblock Cambridge University Press, Cambridge, [2023] \copyright 2023.
\newblock Second edition [of 1278263], The translation of the first Russian
  edition was prepared by Rachel Rowen.

\bibitem[PS97]{PoonenSchaefer-descentprojectiveline}
Bjorn Poonen and Edward~F. Schaefer.
\newblock Explicit descent for {J}acobians of cyclic covers of the projective
  line.
\newblock {\em J. Reine Angew. Math.}, 488:141--188, 1997.

\bibitem[PS99]{PoonenStoll-CasselsTate}
Bjorn Poonen and Michael Stoll.
\newblock The {C}assels-{T}ate pairing on polarized abelian varieties.
\newblock {\em Ann. of Math. (2)}, 150(3):1109--1149, 1999.

\bibitem[PSV10]{PrasadShapiroVemuri-symplectic}
Amritanshu Prasad, Ilya Shapiro, and M.~K. Vemuri.
\newblock Locally compact abelian groups with symplectic self-duality.
\newblock {\em Adv. Math.}, 225(5):2429--2454, 2010.

\bibitem[Ray70]{Raynaud-Faisceauxamples}
Michel Raynaud.
\newblock {\em Faisceaux amples sur les sch\'{e}mas en groupes et les espaces
  homog\`enes}.
\newblock Lecture Notes in Mathematics, Vol. 119. Springer-Verlag, Berlin-New
  York, 1970.

\bibitem[Ros78]{Rosati-generatingparamodulargroups}
Mario Rosati.
\newblock On a generalization of the {S}iegel-{C}onforto symplectic modular
  group {$\Gamma (n,\,\Delta )$}.
\newblock {\em Rend. Sem. Mat. Univ. Padova}, 60:229--236 (1979), 1978.

\bibitem[Sat10]{sato-abelianizationlevelD}
Masatoshi Sato.
\newblock The abelianization of the level {$d$} mapping class group.
\newblock {\em J. Topol.}, 3(4):847--882, 2010.

\bibitem[Sch00]{Schmid-automorphismgroupextraspecial}
Peter Schmid.
\newblock On the automorphism group of extraspecial 2-groups.
\newblock volume 234, pages 492--506. 2000.
\newblock Special issue in honor of Helmut Wielandt.

\bibitem[Ser02]{Serre-galoiscohomology}
Jean-Pierre Serre.
\newblock {\em Galois cohomology}.
\newblock Springer Monographs in Mathematics. Springer-Verlag, Berlin, english
  edition, 2002.
\newblock Translated from the French by Patrick Ion and revised by the author.

\bibitem[SGA73]{SGA4Tome3}
{\em Th\'{e}orie des topos et cohomologie \'{e}tale des sch\'{e}mas. {T}ome 3}.
\newblock Lecture Notes in Mathematics, Vol. 305. Springer-Verlag, Berlin-New
  York, 1973.
\newblock S\'{e}minaire de G\'{e}om\'{e}trie Alg\'{e}brique du Bois-Marie
  1963--1964 (SGA 4), Dirig\'{e} par M. Artin, A. Grothendieck et J. L.
  Verdier. Avec la collaboration de P. Deligne et B. Saint-Donat.

\bibitem[{Sta}18]{stacksproject}
The {Stacks Project Authors}.
\newblock \textit{Stacks Project}, 2018.
\newblock \url{https://stacks.math.columbia.edu}.

\bibitem[Tot25]{Totaro-Chowringtwistedcoefficients}
Burt Totaro.
\newblock Chow groups with twisted coefficients.
\newblock Arxiv Preprint, available at
  \url{https://arxiv.org/abs/2502.20618v1}, 2025.

\bibitem[Tsh19]{Tshishiku-borelsstablerange}
Bena Tshishiku.
\newblock Borel's stable range for the cohomology of arithmetic groups.
\newblock {\em J. Lie Theory}, 29(4):1093--1102, 2019.

\bibitem[Vis05]{Vistoli-grothendiecktopologies}
Angelo Vistoli.
\newblock Grothendieck topologies, fibered categories and descent theory.
\newblock In {\em Fundamental algebraic geometry}, volume 123 of {\em Math.
  Surveys Monogr.}, pages 1--104. Amer. Math. Soc., Providence, RI, 2005.

\bibitem[Zim84]{zimmer-ergodictheory}
Robert~J. Zimmer.
\newblock {\em Ergodic theory and semisimple groups}, volume~81 of {\em
  Monographs in Mathematics}.
\newblock Birkh\"{a}user Verlag, Basel, 1984.

\end{thebibliography}
\bibliographystyle{alpha}

 \begin{footnotesize}
 \textsc{Jef Laga  }\;  \texttt{jeflaga@hotmail.com}   \newline
 \textsc{Department of Pure Mathematics and Mathematical Statistics, Wilberforce Road, Cambridge, CB3 0WB, UK}
 \end{footnotesize}

\end{document}